\documentclass[1 [leqno,11pt]{amsart}
\usepackage{amssymb, amsmath}
\def\refer#1{~\ref{#1}}
\def\refeq#1{~(\ref{#1})}
\def\ccite#1{~\cite{#1}}

\def\longformule#1#2{
\displaylines{ \qquad{#1} \hfill\cr \hfill {#2} \qquad\cr } }
\def\inte#1{
\displaystyle\mathop{#1\kern0pt}^\circ }

\let\pa=\partial
\let\al=\alpha
\let\b=\beta

\let\d=\delta
\let\e=\varepsilon

\let\lam=\lambda
\let\r=\rho
\let\s=\sigma
\let\f=\frac
\let\vf=\varphi
\let\p=\psi

\let\D=\Delta

\let\Om=\Omega
\let\wt=\widetilde
\let\wh=\widehat

\def\cB{{\mathcal B}}
\def\cC{{\mathcal C}}
\def\cD{{\mathcal D}}
\def\cE{{\mathcal E}}
\def\cF{{\mathcal F}}
\def\cG{{\mathcal G}}

\def\cL{{\mathcal L}}
\def\cM{{\mathcal M}}

\def\cR{{\mathcal R}}
\def\cS{{\mathcal S}}
\def\cT{{\mathcal T}}

\def\cV{{\mathcal V}}
\def\cW{{\mathcal W}}

\def\grad{\nabla}
\def\la{\lambda}

\def\dH{\dot{H}}
\def\dB{\dot{B}}

\def\h{\frak h}

\def\virgp{\raise 2pt\hbox{,}}
\def\cdotpv{\raise 2pt\hbox{;}}

\def\eqdefa{\buildrel\hbox{\footnotesize def}\over =}
\def\Id{\mathop{\rm Id}\nolimits}

\def\C{\mathop{\mathbb C\kern 0pt}\nolimits}
\def\DD{\mathop{\mathbb D\kern 0pt}\nolimits}
\def\EE{\mathop{{\mathbb E \kern 0pt}}\nolimits}
\def\K{\mathop{\mathbb K\kern 0pt}\nolimits}
\def\N{\mathop{\mathbb N\kern 0pt}\nolimits}
\def\Q{\mathop{\mathbb Q\kern 0pt}\nolimits}
\def\R{{\mathop{\mathbb R\kern 0pt}\nolimits}}
\def\SS{\mathop{\mathbb S\kern 0pt}\nolimits}
\def\ZZ{\mathop{\mathbb Z\kern 0pt}\nolimits}
\def\TT{\mathop{\mathbb T\kern 0pt}\nolimits}
\def\PP{\mathop{\mathbb P\kern 0pt}\nolimits}
\newcommand{\ds}{\displaystyle}

\newcommand{\Z}{{\ZZ}}

\def\dv{\mbox{\rm div}}

\def\dive{\mathop{\rm div}\nolimits}

\def\no{\noindent}
\def\na{\nabla}
\def\p{\partial}
\def\uh{v^{\rm h}}

\def\h{{\rm h}}
\def\v{{\rm v}}
\def\nh{\nabla_{\rm h}}

\def\uapp{ u^{\rm h}_{\e, \rm app}}

\def \uapp {u_{\e,{\rm app}} }
\def\Piapp {\Pi_{\e,{\rm app}} }

\newcommand{\w}[1]{\langle {#1} \rangle}
\newcommand{\beq}{\begin{equation}}
\newcommand{\eeq}{\end{equation}}
\newcommand{\ben}{\begin{eqnarray}}
\newcommand{\een}{\end{eqnarray}}
\newcommand{\beno}{\begin{eqnarray*}}
\newcommand{\eeno}{\end{eqnarray*}}
\newcommand{\andf}{\quad\hbox{and}\quad}
\newcommand{\with}{\quad\hbox{with}\quad}
\newcommand{\fauxatop} {\genfrac{}{}{0pt}{}{}{}}
\newtheorem{defi}{Definition}[section]
\newtheorem{thm}{Theorem}[section]
\newtheorem{lem}{Lemma}[section]
\newtheorem{rmk}{Remark}[section]
\newtheorem{col}{Corollary}[section]
\newtheorem{prop}{Proposition}[section]
\renewcommand{\theequation}{\thesection.\arabic{equation}}
\begin{document}
\title[Inhomogeneous incompressible viscous flows with slowly variable ]
{ Inhomogeneous incompressible viscous flows with slowly varying
initial data }
 \author[J.-Y. CHEMIN]{Jean-Yves Chemin}
\address [J.-Y. Chemin]%
{Laboratoire J.-L. Lions, UMR 7598 \\
Universit\'e Pierre et Marie Curie, 75230 Paris Cedex 05, FRANCE }
\email{chemin@ann.jussieu.fr}
\author[P. ZHANG]{Ping Zhang}%
\address[P. Zhang]
 {Academy of
Mathematics $\&$ Systems Science and  Hua Loo-Keng Key Laboratory of
Mathematics, The Chinese Academy of Sciences, CHINA}
\email{zp@amss.ac.cn}

\date{\today}

\begin{abstract}  The purpose of this paper is to provide a large class of initial data which generates  global
smooth solution of the  3-D  inhomogeneous incompressible
Navier-Stokes system in the whole space~$\R^3$. This class of data
is based on functions which vary slowly   in one direction. The idea
is that 2-D inhomogeneous Navier-Stokes system with large data is
globally well-posedness and we construct the 3-D approximate
solutions by the 2-D solutions with a parameter.  One of the key
point of this study is the investigation of the time decay
properties of the solutions to the  2-D inhomogeneous Navier-Stokes
system. We obtained the same optimal decay estimates as the
solutions of 2-D homogeneous Navier-Stokes system.
\end{abstract}

\maketitle



\setcounter{equation}{0}
\section{Introduction}

In this paper, we investigate the global well-posedness of 3-D
incompressible inhomogeneous Navier-Stokes system with large initial
data slowly varying in one space variable. In general, inhomogeneous
Navier-Stokes system in~$\R^d$  reads
\begin{equation*}
{\rm (INSdD)}\quad \left\{\begin{array}{c} \displaystyle
\p_t\r+u\cdot\na \r=0,
\\
\displaystyle \r\pa_t u + \r u\cdot\na u  -\D u + \nabla \Pi=0,\\
\displaystyle \dive u=0,
\\
\displaystyle  (\r, u)|_{t=0}=(\r_0, u_0).
\end{array}\right.
\end{equation*}
Here the unknown~$\rho$ is a function from~$[0,T]\times\R^d$ into
the interval~$]0,\infty[$ which represents the density of fluid at
time~$t$ and point~$x$\footnote{We do want to avoid vacuum}, the
unknown~$u=(u^1,\cdots,u^d)$ is a time dependent vector field
on~$\R^d$ which represents the velocity  of a particule of  fluid located at
position~$x$ and time~$t$  and~$\Pi$ is a function
from~$[0,T]\times\R^d$ to~$\R$ which represents the pressure  at
point~$x$ and time~$t$  which ensures the incompressibility of the
fluid. The choice of~$\R^d$ as a domain is a real simplification
because as we shall see later on the pressure  is
 uniquely determined by the divergence free condition  on
the vector field~$u$ (the case of periodic boundary condition i.e.
the flat torus~$\TT^d$ as a domain  also works).

Let us notice that in the case when~$\rho_0\equiv 1$, the
system~(INSdD) turns out to be the homogeneous incompressible
Navier-Stokes system. We have to keep in mind that the system of
(INSdD) is more complex than this one.

 This system~(INSdD) can be used as a model to describe  a  fluid that is incompressible but has
nonconstant density. Basic examples are mixture of incompressible
and non reactant flows, flows with complex structure (e.g. blood
flow or model of rivers), fluids containing a melted substance, etc.
 \smallbreak

First of all, this equation satisfies some {\it a priori} estimates.
Let us first study the {\it a priori} estimate on the density.
It is classical to consider the density~$\rho$ as a perturbation of
the homogeneous density arbitrarily  chosen to be equal to~$1$. Let
us introduce the notation
$$
\varrho \eqdefa \rho-1
$$
which will be used all along this text.

\medbreak This system has three major basic  features. First of all,
the incompressibility expressed by the fact that the vector
field~$u$ is divergence free gives \beq \label{transformrho}
 \forall p \in [1,\infty]\,,\ \|\varrho(t)\|_{L^p}
=\|\varrho_0\|_{L^p}\andf \|\rho(t)\|_{L^\infty}
=\|\rho_0\|_{L^\infty}. \eeq Moreover, the second equation
of~(INSdD), called the momentum equation, implies a control of the
total kinetic energy which is formally expressed by \beq \label
{kineticenergy} \frac 12 \int_{\R^d}\rho(t,x)|u(t,x)|^2 dx
+\int_{0}^t\|\nabla u(t')\|_{L^2}^2 dt' =\frac 12
\int_{\R^d}\rho_0(x)|u_0(x)|^2 dx. \eeq This third  basic feature is
the scaling invariance.
 Indeed, if~$(\rho,
u ,\Pi)$ is a solution of~(INSdD) on~$[0,T]\times \R^d$,
then~$(\rho, u ,\Pi)_\lam$ defined by
\beq
\label{scaling}
 (\rho, u
,\Pi)_\lam(t,x) \eqdefa \bigl(\rho(\lam^2t, \lam x) , \lam u
(\lam^2t, \lam x) ,\lam^2 \Pi(\lam^2t,\lam x)\bigr) \eeq is also a
solution of~(INSdD) on $[0,\lam^{-2} T]\times \R^d$. This leads to the
notion of critical regularity.

\medbreak Based on the energy estimate\refeq {kineticenergy},  J.
Simon constructed in\ccite{Simon} global weak solutions of (INSdD)
with finite energy (see also the book by P.-L. Lions \cite{Lions96}
for the variable viscosity case).

In the case of smooth data with no vacuum, Lady\v zenskaja and
Solonnikov first addressed in \cite{LS} the question of unique
solvability of (INSdD). More precisely, they considered the system
(INSdD) in a bounded domain $\Om$ with homogeneous Dirichlet
boundary condition for~$u.$ Under the assumptions that $u_0\in
W^{2-\frac2p,p}(\Om)$ $(p>d)$ is divergence free and vanishes on~$\p\Om$ and that $\r_0$ belongs to~$C^1(\Om)$ is bounded away from zero, then
they  proved in \cite{LS}:
\begin{itemize}
\item the global well-posedness in dimension $d=2;$
\item the local well-posedness in dimension $d=3;$
\item the  global well-posedness  if in addition $u_0$ is small in $W^{2-\frac2p,p}(\Om)$.
\end{itemize}

More  recently,  M. Paicu, Z. Zhang and the second author  proved
in\ccite {PZZ1} the following well-posedness result for (INS3D) with
small data.

\begin{thm}
\label  {smalldata}
{\sl Let us consider an  initial data $(\rho_0,u_0)$ in~$L^\infty(\R^3)\times H^1(\R^3)$. Let us assume that for some  positive constant~$C_0$ ,
\beno
C^{-1}_0\le \rho_0(x)\le C_0.
\eeno
Then there exists a constant $\varepsilon_0>0$ depending
only on~$C_0$ such that if $\|u_0\|_{L^2}\|\na u_0\|_{L^2}\leq \varepsilon_0$, then the system~(INS3D) has a unique global solution $(\rho,u).$
}
\end{thm}
Let us notice that smallness condition in Theorem \ref{smalldata} is
scaling invariant. Moreover, the fact that in dimension two, the
system~(INS2D) is globally well-posed is related to the fact that in
dimension  two, the quantity
$$
\frac 12 \int_{\R^2}\rho(t,x)|u(t,x)|^2 dx +\int_{0}^\infty\|\nabla
u(t')\|_{L^2}^2 dt'
$$
is scaling invariant under the scaling transformation \eqref{scaling}.

\medbreak In this text, we shall consider slowly varying initial
data i.e.  a family of initial data of the form
 \beq
  \label
{thinsslowvareq1} (\rho_{0,\e,\eta}, u_{0,\e,\eta}) \eqdefa
\bigl(1+\eta [{\varsigma}_0]_\e, ([v_0^\h]_\e,0)\bigr), \eeq
 where~$\e$ and~$\eta$
are two positive real parameters, ~$\varsigma_0$ is a smooth
function, and~$v_0^\h$ is a smooth divergence free 2D  vector field
which depends on a real parameter~$z$. All along this text, we use
the notation, for a function~$f$ on~$\R^3$,
$$
[f]_\e (x_\h,x_3)\eqdefa f(x_\h,\e x_3)\with x=(x_\h,x_3)=(x_1,x_2,x_3).
$$

Here we are interested in the size of the initial data. We do not
intent to solve~(INS3D) for rough initial data but simply we want to
exhibit a large class of initial data which are ``large" in the
sense that they do not satisfies any previous smallness hypothesis
that ensures global existence of regular solutions.
The main theorem of this text is the following.

\begin{thm}
\label {insslowvar}
 {\sl Let us consider initial
profiles~$\varsigma_0$ and~$v^\h_0$ which are functions and vector
fields in~$\cS(\R^3)$
 such that~$\dive_\h v_0^\h =0$ and such that for any~$z$ in~$\R$ and any~$j$ in~$\{1,2\}$
\beq \label {S7eq0.1} \int_{\R^2} \varsigma_0(x_\h,z) v_0^\h(x_\h,z)
\,dx_\h=0 \andf \int_{\R^2} x_j\varsigma_0(x_\h,z) v_0^\h(x_\h,z)
\,dx_\h=0. \eeq Then there exists two positive constants~$\eta_0$
and~$\e_0$ which depend on norms of~$\varsigma_0$ and~$v_0^\h$ such
that  if~$\eta \leq \eta_0$ and~$\e\leq \e_0$,  the initial date
defined by\refeq {thinsslowvareq1} generates  a unique global smooth
solution of~(INS3D). }
\end{thm}

Let us make  some comments about this theorem. Slowly varying data
has been introduced by I. Gallagher and the first author in\ccite
{cg3}  in the case of homogenenous incompressible Navier-Stokes
equations, i.e. the case when~$\rho\equiv 1$. The above theorem is
proved in\ccite {cg3} in this case.  The motivation of this work was
to provide a large class of examples of initial data which are large
(in the homogeneous incompressible Navier-Stokes, it means
essentially that the~$B^{-1}_{\infty,\infty}$ norm of the initial
data defined by \beq \label {definB-1elem}
\|a\|_{\dot B^{-1}_{\infty,\infty}} \eqdefa \sup_{t>0} t^{\f12} \|e^{t\D}
a\| _{L^\infty}. \eeq is large), which is  the case here,
because~$\|[v_0^\h]_\e \|_{B^{-1}_{\infty,\infty}}$ has the same
size as the profile~$v_0^\h$. The idea of the proof in \ccite {cg3}
was to use that homogeneous 2D incompressible  Navier-Stokes
equation with initial data $v_0^\h(\cdot ,z)$ is globally well-posed
and  then to prove the real solution was close (in some appropriated
way) to~$[v^\h]_\e$.

Slowly  varying turns out to be a useful tool to study the set~$\cG$
of initial data  in the space~$\dot H^{\frac 12}(\R^3)$  which
generates unique global smooth solutions to 3-D homogeneous
Navier-Stokes system. Since the work in \cite{GIP} by I. Gallagher,
D. Iftimie and F. Planchon, it is known that this set is open and
connected. In\ccite {cgz}, I. Gallagher and the two authors used
slowing varying initial data to prove that through each point
of~$\cG$ passes an uncountable number of arbitrarily long segments
which are included in~$\cG$.

The study of initial data in the homogeneous case as presented above
can be qualified as ``well-prepared" using the language of singular
perturbation theory. The ``ill-prepared" case has been studied  by
I. Gallagher, M. Paicu  and the first author  in\ccite {cgp}; they
proved  that the initial data
$$
\Bigl( [w_0^\h]_\e,\frac 1 \e [w_0^3]_\e\Bigr)
$$
generates a unique global smooth solution of the homogenous
incompressible Navier-Stokes equation when the profile~$w$ is a
divergence free vector field and  which is small in a Banach space
of analytic function with respect to the vertical variable.

Let us see why the result of Theorem\refer {insslowvar} is in some
sense a ``ill-prepared" result.  In order to explain this, we recall
the precise definition of the Besov norms from \cite{BCD} for
instance.
\begin{defi}
\label {S0def1} {\sl  Let us consider a smooth function~$\vf $
on~$\R,$ the support of which is included in~$[3/4,8/3]$ such that
$$
\forall
 \tau>0\,,\ \sum_{j\in\Z}\varphi(2^{-j}\tau)=1\andf \chi(\tau)\eqdefa 1 - \sum_{j\geq
0}\varphi(2^{-j}\tau) \in \cD([0,4/3]).
$$
Let us define
$$
\Delta_ja=\cF^{-1}(\varphi(2^{-j}|\xi|)\widehat{a}),
 \andf S_ja=\cF^{-1}(\chi(2^{-j}|\xi|)\widehat{a}).
$$
Let $(p,r)$ be in~$[1,+\infty]^2$ and~$s$ in~$\R$. We define the Besov norm by
$$
\|a\|_{\dB^s_{p,r}}\eqdefa\big\|\big(2^{js}\|\Delta_j
a\|_{L^{p}}\big)_j\bigr\|_{\ell ^{r}(\ZZ)}.
$$
}
\end{defi}

We remark that in the particular case when $p=r=2,$  the Besov
spaces~$\dB^s_{p,r}$ coincides with the classical homogeneous
Sobolev spaces $\dH^s$.

 The
result of Theorem\refer {insslowvar}  is of ``ill-prepared" type
because of Inequality\refeq {insslowvarcommentseq1} and that in the
case when~$ (s,p,r)=(-1,\infty,\infty),$ it coincides with
Definition\refeq {definB-1elem}. Yet we do not require any analytic
assumption on the initial data.

First of all, let us notice that if a divergence vector field the
component of which are integrable is mean free. Thus,
Hypothesis\refeq {S7eq0.1} implies in particular that \beq
\label{1.7dfg} \forall z\in \R\,,\ \int_{\R^2}
\bigl(1+\eta\varsigma_0(x_\h,z)\bigr) v_0^\h(x_\h,z) \,dx_\h=0. \eeq
Let us notice that the hypothesis about the momentum
of~$\varsigma_0v^\h_0$ ensures in particular
that~$\varsigma_0v^\h_0$ belongs to the anisotropic
space~$\cB^{-1,\frac 32}_2$ (see forthcoming Definition\refer
{anibesov}). Following observations of the first author and I.
Gallagher in\ccite  {cg3}, it is easy to prove that \beq \label
{insslowvarcommentseq1} \| [\varsigma_0]_\e\|_{\dB^{\frac 3 p}_{p,1}
(\R^3)} \gtrsim  \e^{-\frac 1p} \andf \| [v_0]_\e\|_{\dB^{-1+\frac 3
p}_{p,1} (\R^3)} \gtrsim \e^{-\frac 1p}.
 \eeq

All the  well-posedness results  of (INS3D) for small data requires
that for $p$ in~$]1,6[$
$$
 \| [v_0]_\e\|_{\dB^{-1+\frac 3
p}_{p,1} (\R^3)}<< 1 \quad\hbox{with or without}\quad  \eta \|
[\varsigma_0]_\e\|_{\dB^{\frac 3 p}_{p,1}}<<1.
$$
One may check the references \cite{Abidi, AP, AGZ2, AGZ3, D1, DM12,
dm2, DZ14, HPZ, PZ2, PZZ1, PZZ3} for details.

Let us also mention that  M. Paicu and the two authors proved this
theorem in~\cite{CPZ1} in the case when~$\eta \leq \e^{\s}$
with~$\s>\f14$.

\smallskip

Let us complete this section by the notations of the paper:

 For $a\lesssim
b$, we mean that there is a uniform constant $C,$ which may be
different on different lines, such that $a\leq Cb$.  We denote by
$(a|b)_{L^2}$ the $L^2(\R^d)$ inner product of $a$ and $b.$ For~$X,
X_1$  Banach spaces,  $T$ a positive real number and $q$
in~$[1,+\infty],$ we denote the norm $\|\cdot\|_{X\cap
X_1}=\|\cdot\|_{X}+\|\cdot\|_{X_1}$ and $L^q_T(X)$ for the set of
measurable functions on $[0,T]$ with values in $X,$ such that
$t\longmapsto\|f(t)\|_{X}$ belongs to~$L^q([0,T]).$ We denote
$$
L^p_T(L^q_{\rm h}(L^r_{\rm v}))= L^p([0,T]; L^q(\R_{x_{\rm h}};L^r(\R_{z})))
$$
 with~$x_{\rm
h}=(x_1,x_2),$ and $\nabla_{\rm h}=(\p_{x_1},\p_{x_2}),$ $\D_{\rm
h}=\p_{x_1}^2+\p_{x_2}^2.$  Finally  $\D_\e$ stands for $\D_{\rm
h}+\e^2\p_z^2,$  $\na_\e$ for $(\na_\h, \e \p_z),$ and
$\|f\|_{X_\h}$ for the $X$ norm of $f$ in the horizontal variable
$x_\h.$

\setcounter{equation}{0}
\section {Structure and main ideas of the proof}
\label {structureproof}

Because we shall consider seemingly  perturbations of the reference
density $1,$ it is natural to set
$$
a\eqdefa \frac 1 \rho-1 = -\frac {\varrho} {1+\varrho}
 $$
 so that System
(INSdD) translates into
\begin{equation*}
{\rm (INSdD)}\quad \left\{\begin{array}{c} \displaystyle \p_t a
+u\cdot\na a=0,
\\
\displaystyle \pa_t u +  u\cdot\na u  - (1+a) \bigl( \D u - \nabla \Pi)=0,\\
\displaystyle \dive u=0, 
\\
\displaystyle  (a, u)|_{t=0}=(a_0, u_0).
\end{array}\right.
\end{equation*}

Even if our main motivation comes from  dimension~$3$, we shall
consider this system in both~$\R^2$ and~$\R^3$. In~$\R^3$, we
use systematically the notation~$x=(x_\h,x_3)$ and~$x=(x_\h,z)$ in
the case when~$z$ represents~$\e x_3$.

 For proving Theorem\refer{insslowvar}, we follow the idea of\ccite{cg3}, namely,  using the fact
  that the two dimensional
  incompressible inhomogeneous Navier-Stokes system is globally well-posed. We shall construct the approximate
  solutions of (INS3D) with initial data slowly varying in one space
  variable
  in the following way.  Let us denote by~$(a^\h,v^\h,\Pi^\h)$ the (global) solution of~(INS2D)  with initial data
~$\bigl(a_0(\cdot,z), v_0^\h(\cdot,z)\bigr),$ that is
\begin{equation}
\label {INS2dParameter}
\left\{\begin{array}{c} \displaystyle \p_t a^\h+\uh\cdot\na_{\rm
h}a^\h=0,  
\\
\displaystyle \pa_t \uh + \uh\cdot\na_{\rm h}\uh -(1+a^\h)(\D_{\rm h} \uh-\grad_{\rm h} \Pi^\h\bigr)=0,\\
\displaystyle \dv_{\rm h}\,\uh =0, 
 \\
\displaystyle a^\h|_{t=0}=a_0(x_{\rm h},z),\quad v^{\rm
h}|_{t=0}=v_0^{\rm h}(x_{\rm h},z),
\end{array}\right.
\end{equation}
which can also be equivalently written as
\begin{equation}
\label {INS2dParameterb}
\left\{\begin{array}{c} \displaystyle \p_t\r^\h+v^{\rm
h}\cdot\na_{\rm
h}\r^\h=0, 
\\
\displaystyle \pa_t (\r^\h v^{\rm h}) +  \dive_\h(\r^\h v^{\rm h}\otimes v^{\rm h}) -\D_{\rm h} v^{\rm h}+\grad_{\rm h} \Pi^\h=0,\\
\displaystyle \dv_{\rm h}\,v^{\rm h}=0,  
\\
\displaystyle \r^\h|_{t=0}=1+\eta\varsigma_0(x_{\rm h},z),\quad
v^{\rm h}|_{t=0}=v_0^{\rm h}(x_{\rm h},z).
\end{array}\right.
\end{equation}
where~$ \r^\h\eqdefa\ds  \f1{1+a^\h}$ and~$\ds a_0=-\f{
\varsigma_0}{1+\eta\varsigma_0}\eta$. As in\ccite{cg3}, we
consider~$([a^\h]_\e, [v^\h]_\e, [\Pi^\h]_\e)$ as the  first
approximation of the solution to (INS3D) and let us write   the
solution~$(a_\e,u_\e, \wt \Pi_\e)$ as
$$
(a_\e,u_\e, \wt \Pi_\e) = \bigl([a^\h]_\e, ([v^\h]_\e,0), [\Pi^\h]_\e\bigr) + \e (b_\e, \overline R_\e, \overline \Pi_\e).
$$
It is easy to observe that
$$
 \pa_t \overline R_\e +  [v^\h]_\e\cdot\na_\h \overline R_\e +\overline R_\e\cdot\nabla( [v^\h]_\e,0) +\e \overline R_\e\cdot\nabla \overline R_\e - (1+a_\e) \bigl( \D \overline R_\e- \nabla \overline \Pi_\e)=  -\overline E_\e,
$$
where the error term~$\overline E_\e$ is given by
\beq
\label {structureproofeq1}
\begin{split}
&\overline E_\e  \eqdefa  \frac 1 \e \Bigl(\bigl[ v_t^\h +v^\h\cdot\nabla_\h v^\h -(1+a^\h) (\D_\h v^\h-\nabla_\h \Pi^\h)\bigr]_\e,0\Bigr)\\
&\qquad\qquad{} - \e(1+a_\e)\bigl([\partial_z^2v^\h]_\e,0\bigr) -
(1+a_\e) (0,[\partial_z\Pi^\h]_\e)- b_\e\bigl([\D_\h
v^\h-\nabla_\h\Pi^\h]_\e,0\bigr).
\end{split}
\eeq At this stage, we need to define precisely the norms we shall
use to measure the size of all the terms above. As already commented
in the introduction, this point is crucial.  Let us define the
anisotropic Besov norms.
\begin{defi}
\label {anibesov}
{\sl Let us consider two functions~$\vf$ and~$\chi$ given by
Definition\refer {S0def1} and let us define the operators of
localization in horizontal and vertical frequencies by \beno
&\Delta_k^{\rm h}a=\cF^{-1}(\varphi(2^{-k}|\xi_{\rm
h}|)\widehat{a}),\qquad
\Delta_\ell^{\rm v}a =\cF^{-1}(\varphi(2^{-\ell}|\zeta  |)\widehat{a}),\\
&S^{\rm h}_ka=\cF^{-1}(\chi(2^{-k}|\xi_{\rm h}|)\widehat{a}),
\qquad\ S^{\rm v}_\ell a =
\cF^{-1}(\chi(2^{-\ell}|\zeta  |)\widehat{a})
\eeno
Now let us define the norm we are going to use in this text. For~$p$ in~$[1,\infty]$, and~$(s,s')$ in~$\R^2$, we define
$$
\|a\|_{\cB^{s,s'}_p}\eqdefa \sum_{(k,\ell)\in\Z^2} 2^{ks+\ell s'}
\|\D_k^{\rm h}\D_\ell^{\rm v}a\|_{L^p}.
$$
We shall also use the following norm on  force~$f$ which  involves
the action of the heat  flow. Let~$p$ be in~$]3,4[$, and~$T$ a
positive time, we define
\beno
 \|f\|_{\cF_p(T)} &\eqdefa& \Bigl \|\int_0^t e^{(t-t')\D} f(t')dt' \Bigr\|_{X(T)}\with\\
 \|f\|_{X(T)}
 &  \eqdefa&
\|f\|_{L^4_T(\cB^{-\f12+\f2p,\f1p}_p)}+\|f\|_{L^2_T(\cB^{\f2p,\f1p}_p)}+\|\p_3f\|_{L^{\f43}_T(\cB^{-\f12+\f2p,\f1p)}_p}\\
&&{}+\|\na
f\|_{L^1_T(\cB^{\f2p,\f1p}_p)}+\|\p_3^2f\|_{L^1_T(\cB^{-1+\d+\f3p,-\d}_p)}\quad\mbox{for}\
\ \d\in\bigl]0,1-3/p\bigr[. \eeno }
\end{defi}
Let us make some comments about this definition. We first address
that the norms of~$\cB^{s,s'}_p$ are homogeneous with respect to the
vertical variable. More precisely, we have, for any real number~$s$ and any~$p$ in~Ê$[1,\infty]$,
 \beq
 \label {homovertBss'}
  \|[a]_\e\|_{\cB^{s,s'}_p} \sim \e^{s'-\frac 1p}
\|a\|_{\cB^{s,s'}_p}.
\eeq
In particular, the norm~$\|\cdot\|_{\cB_p^{s,\frac 1p} }$ is invariant under the
vertical dilation.

We now investigate the relation between~$L^1$ norm in time with
value in some anisotropic Besov spaces and the norm~$\cF_p(T)$. For
any~$p$ in~$[1,\infty]$, and for any~$(\al,\b)$ in~$\R^2$ such that
$$
\al+\b=-1+\frac 3 p,\quad\ \al \leq -1+\d+\f3p \andf  \b \leq \f1p\,\virgp
$$
then we have \beq \label {cFL1Besovaniso} \|f\|_{\cF_p(T)} \leq
C_{\al,\b} \|f\|_{L^1([0,T];\cB^{\al,\b}_p)}. \eeq We postpone its
proof in the Appendix.

We also use frequently some law of product in particular (see Lemma
2.3 of \ccite{CPZ1})
\beq
\label {lawofproductanisobasic}
\|ab\|_{\cB^{s_1+s_1'-\frac 2 p,s_2+s'_2-\frac 1 p }_p}\lesssim
\|a\|_{\cB^{s_1,s_2}_p}\|b\|_{\cB^{s_1',s'_2}_p}
\eeq
where the two sums~$s_1+s_1'$ and~$s_2+s_2'$ are positive and~$s_1$ and~$s'_1$ (resp.~$s_2$ and~$s_2'$) are less than or equal to~$ 2/p$ (resp.~$ 1/{p}$).


\medbreak Now let us analyze the constraints  we have for the choice
of norms for the different terms in the external force given by
\eqref{structureproofeq1}. For those which  are purely of the
form~$[f]_\e$,  there is in fact no choice. Indeed, since no
negative power of~$\e$ appears, the choice of norm to the
space~$\cB^{\s,\frac 1p}_p$  is mandatory. This space must be~$L^1$
in time because we want~$R_\e$ to be in $L^1$ in time because of the
control of the transport equation. Then the scaling determines the
index~$\s$ of the horizontal regularity. The space must be
$$
L^1_t(\cB^{-1+\frac 2 p ,\frac 1 p }_p).
$$
Let  us see whether  the term~$[\partial_z\Pi^\h]_\e$ which appears
in\refeq {structureproofeq1} belongs to this space or not.
Let us compute this horizontal pressure.  Applying the horizontal
divergence to the momentum equation of\refeq {INS2dParameterb}, we
write that
$$
-\D_\h \Pi^\h = \dive_\h ( \pa_t (\r^\h v^{\rm h})) +
\dive_\h\dive_\h(\r^\h v^{\rm h}\otimes v^{\rm h}).
$$
Using the fact that~$v^\h$ is divergence free, we infer that for
$\varrho^\h\eqdefa \r^\h-1,$ \beq \label{decomppressureh}
\begin{split}
&\qquad\qquad\qquad\Pi^\h = \Pi^\h_{L} +\Pi^\h_Q \with \\
\Pi^\h_{L}  &\eqdefa -\D_\h^{-1}  \dive_\h \bigl( \pa_t (\varrho^\h
v^{\rm h})\bigr)\andf \Pi^\h_Q  \eqdefa -\D_\h^{-1}
\dive_\h\dive_\h(\r^\h v^{\rm h}\otimes v^{\rm h}).
\end{split}
\eeq It will be possible to prove that the
term~$[\partial_z\Pi_Q^\h]_\e$  belongs to~$L^1(\R^+;\cB^{0,\frac
12}_2)$; this will be a consequence  of  Sections\refer
{decayvh}--\ref {decayL2partialz2}. On the other hand, it is not
possible to prove that~$\partial_z\Pi_L$  belongs to the Besov
space~$\dB^0_{2,1}$ horizontally. Indeed, it is equivalent  to the
fact  that a   homogeneous Fourier multiplier  of order~$-1$ applied
to a product belongs to~$\dB^0_{2,1}$ in the horizontal variable.
The lowest possible regularity  of a product is~$L^1$. But the
space~$L^1$ is included in~$\dB^{-1}_{2,\infty}$ in dimension two
and even not  in the homogeneous Sobolev space~$\dot H^{-1}$.
In order to bypass this difficulty, we introduce  a correction term.
In order to define it, let us consider the vector
field~$w_\e(t,x_\h,z)$ the solution~of
\begin{equation}
\label{S2eq0}
\left\{\begin{array}{c} \displaystyle \p_t w_\e^\h
- \D_\e w_\e^\h = -\nh \Pi^1 _\e, 
\\
\displaystyle \pa_t w_\e^3
-\D_\e w_\e^3= -\e^2\p_z \Pi^1_\e+\p_z \Pi_L^\h,\\
\displaystyle \dv w_\e =0 \andf  {w_\e}_{|t=0}=0.
\end{array}\right.
\end{equation}
Let us introduce the following Ansatz. We search the
solution~$(a_\e,u_\e,\wt \Pi_\e)$ of (INS3D) of the form
\beq
\label{S2eq0aq}
\begin{split}
(a_\e,u_\e,\wt \Pi_\e) & =\bigl([a^\h]_\e, \uapp , \Piapp \bigr)+\e (b_\e,R_\e,\Pi_\e)
\with\\
\bigl(\uapp , \Piapp \bigr) & \eqdefa  \bigl( ([\uh]_\e,0),
[\Pi^\h]_\e\bigr)+  \e\bigl((\e[w_\e^\h]_\e,
[w_\e^3]_\e),\e[\Pi_\e^1]_\e\bigr).
\end{split}
\eeq Then $(R_\e, \na\Pi_\e)$ solves the system
 \begin{equation*}
{\rm (INS3D)_{\rm e}}
\left\{
\begin{array}{c}
\displaystyle
\!\! \pa_t R_\e +  u_{\e,{\rm app}}\cdot\na R_\e +R_\e\cdot\nabla \uapp +\e R_\e\cdot\nabla R_\e - (1+a_\e) \bigl( \D R_\e- \nabla \Pi_\e)=  -E_\e,\\
\displaystyle \dive R_\e=0 \andf R_\e |_{t=0}=0.
\end{array}\right.
\end{equation*}
where the error term~$E_\e$ is given by \beq \label
{firstdefinErrorterm} E_\e \eqdefa \frac 1\e \bigl( \partial_t \uapp
+\uapp\cdot\nabla \uapp-(1+a_\e) \bigl(\D\uapp-\nabla \Pi_{\e,{\rm
app}})\bigr). \eeq Of course the key point then is the estimate of
this error term.  Let us  analyze~it.  First we write that
$$
\displaylines {
 E_\e   =  \frac 1 \e \Bigl(\bigl[ \partial_t
v^\h +v^\h\cdot\nabla^\h v^\h -(1+a^\h) (\D_\h v^\h-\nabla_\h
\Pi^\h)\bigr]_\e,0\Bigr)\cr
   + \bigl[ v^\h\cdot\nabla_\h(\e w^\h_\e ,w^3_\e )+\e
w_\e\cdot\nabla (v^\h,0)+\e^2w_\e\cdot\nabla(\e  w^\h_\e ,w^3_\e )
\bigr]_\e\cr
   -\e(1+a_\e)\bigl([\partial_z^2v^\h]_\e,0\bigr) - b_\e\bigl([\D_\h
v^\h-\nabla_\h\Pi^\h]_\e,0\bigr)\cr
   +\bigl[\partial_t (\e w^\h_\e ,w^3_\e ) -\D_\e (\e w^\h_\e ,w^3_\e
) -\e \nabla_\e\Pi^1_\e - (0,
\partial_z\Pi^\h_L)\bigr]_\e\cr
   - a_\e \bigl[ \D_\e (\e w^\h_\e ,w^3_\e ) -(0,\partial_z \Pi^\h)-
\e \nabla_\e \Pi^1_\e \bigr]_\e - \bigl(0,
[\partial_z\Pi_Q^\h]_\e\bigr).
}
$$
By definition of~$(a^\h,v^\h,\Pi^\h)$ and~$(w_\e,\Pi^1_\e )$, we can write
\beq
\label  {analyzeErrortermeq1}
\begin{split}
E_\e & = \sum_{\ell=1}^4 E_\e^\ell\with \\
E_\e^1& \eqdefa
[ v^\h\cdot\nabla_\h(\e w^\h_\e ,w^3_\e )
- (0, \partial_z\Pi_Q^\h)+ \e w_\e\cdot\nabla (v^\h,0)+\e^2w_\e\cdot\nabla(\e  w^\h_\e ,w^3_\e ) \bigr]_\e,\\
E_\e^2 & \eqdefa
 -\e (1+a_\e)\bigl([\partial_z^2v^\h]_\e,0\bigr), \cr
E_\e^3 & \eqdefa
  -a_\e \bigl[ \D_\e (\e w^\h_\e  ,w^3_\e ) -(0,\partial_z\Pi^\h)-\e\nabla_\e\Pi^1_\e \bigr]_\e\andf\\
E_\e^4 & \eqdefa
 -b_\e \bigl([\D_\h v^\h-\nabla_\h\Pi^\h]_\e,0\bigr).
\end{split}
\eeq

Let us  remark that the term~$v^\h$ is ubiquitous in the error
term~$E_\e$, even in~$w_\e$ because ~$\Pi_L^\h$ depends on this
vector field~$v^\h$. Thus the property of~$v^\h$  are crucial for
the understanding  of the error term~$E_\e$.

Section\refer {decayvh} is devoted to the systematic study of
 the time decay of~$v^\h$. This section is  devoted to the 2D case
 and can be of independent interest. We generalize the decay in time
 estimates
 obtained in the case of homogeneous  Navier-Stokes equation by M. Wiegner in\ccite {Wiegner}
(see also the works\ccite{brandolese}, \ccite{HM06},
\ccite{Schonbek}, \ccite{Schonbek2} and see \cite{CZ6} for the
application of this method to a singular perturbation of the~2D Navier-Stokes
system). We remark that to obtain this optimal time decay estimate
for $v^\h,$ we need to use a completely new formulation (see \eqref
{estimuflatdemoeq2} below) of the inhomogeneous Navier-Stokes
system.

\medbreak

As it can be observed in the term~$E^2_\e$, we need~$L^1$ in time
estimate  of term that involves second derivative of~$v^\h$ with
respect to the vertical variable~$z$. This is the purpose of
Sections  \refer {decayL2partialz} and \refer{decayL2partialz2}. A
first consequence of this  study is the following proposition,  the
proof of which will be presented   in Section \ref{Besovvh}.

\begin{prop}
\label {consequenceSection345} {\sl Under the hypothesis of
Theorem\refer {insslowvar},  we have
$$
\longformule{
\|v^\h\|_{L^2(\R^+;\cB^{1,\frac 12}_2)} + \|\p_z
v^\h\|_{L^2(\R^+;\cB^{1,\frac 12}_2)} +\|\nabla v^\h\|_{L^1(\R^+;
\cB^{1,\frac 12}_{2})}
}
{ {}
+\|\p_zv^\h\|_{L^1(\R^+;\cB^{\f34,\f34}_2)}+
\|\partial_z\Pi^\h\|_{L^1(\R^+;\cB^{\f12,\frac 12}_2)}\leq \cC_0.
}
$$
 Here and in all that follows, we always denote $\cC_0$ to be a
positive constant which depends  norms of the two profiles~$\varsigma_0$
and~$v_0^\h$ of the initial data and which may be changed from line
to line.}
\end{prop}

The estimate obtained in Sections \refer{decayvh}, \refer
{decayL2partialz} and \refer{decayL2partialz2} allow to prove the
following  proposition.
\begin{prop}
\label {estimwdetail} {\sl Let~$(w_\e,\Pi^1_\e)$ be the solution of
the System \eqref{S2eq0}, then we have
$$
\longformule{
\e\|(\e w^\h_\e,w_\e^3)\|_{ L^4(\R^+;\cB^{\f12,\frac 12}_2)}+
\|(\e w^\h_\e,w_\e^3)\|_{ L^2(\R^+;\cB^{1,\frac 12}_2)}
}
{{}
+\|\nabla
_\e (\e w_\e^\h,w_\e^3)\|_{L^2(\R^+;\cB^{0,\frac 12}_{2})} +\|\nabla
_\e (\e w_\e^\h,w_\e^3)\|_{L^1(\R^+;\cB^{1,\frac 12}_{2})}   \leq
\cC_0.
}
$$
Moreover for any~$\al$ in~$ ]0,1[$, we have
$$
\bigl\| \D_\e (\e w_\e^\h ,w_\e^3)
-\e\nabla_\e\Pi_\e^1\bigr\|_{L^1(\R^+;\cB_2^{\al,\frac 12})} \leq
C_\al \cC_0.
 $$}
\end{prop}
The proof of this proposition is the purpose of Section\refer {estimw}.

\medbreak Using law of product\refeq {lawofproductanisobasic}, the
above two propositions imply that \beq \label {estimE1}
\|E^1_\e\|_{L^1(\R^+;\cB^{0,\f12}_2)} \leq \cC_0, \eeq which is the
content of Corollary \ref{estimE2E3}.

 The two terms~$E_\e^2$ and~$E_\e^3$ are of a different nature. They contain of course  terms which are rescaled functions
   of~$(a^\h, v^\h,\Pi^\h)$ and~$w_\e$  multiplied by  the function~$a_\e$. Their control demands estimates on the function~$a_\e$.
 This requires the following  induction hypothesis.

Let $p$ be in~$ ]3,4[$ and ~$\cR_0$ be a positive real number which
will be chosen large enough later on, we define~$\overline T_\e $ as
 \beq
 \label {defininductiontime1}
\overline T_\e\eqdefa \sup\bigl \{ t<T^\star_\e\,/\  \|R_\e\|_{L^4_t(\cB^{-\f12+\frac 2 p,\frac1p}_p)}  +\|\na  R_\e\|_{L^1_t(\cB^{\frac 2 p,\frac1p}_p)}
 \leq  \cR_0\bigr\}
 \eeq
 where~$T^\star_\e$ denotes is the life span of the regular solution of~(INS3D) associated with the initial
  data~$\bigl(1+\eta [\varsigma_0]_\e, ([v^\h_0]_\e,0)\bigr)$.
 Under the above induction hypothesis,  the regularity of~$a_\e$ is controlled thanks to  the following proposition.
  \begin{prop}
\label {fulltransport} {\sl Let $(u_\e)_\e$ be a family of
divergence free vector fields and~$a_0$ a function in~$L^p$ with
derivatives also in~$L^p$. Let us consider the family~$(a_\e)_\e$ of
the solutions to \beq \label{fulltransporteq1}
\left\{\begin{array}{c} \displaystyle \p_ta_\e+u_\e\cdot\na a_\e=0,\\
\displaystyle {a_\e}_{|t=0}=[a_0]_\e.
\end{array}\right.
\eeq Then for any $s$ in~$ ]0,1-1/p[,$ we have
\beno
\|a_\e(t)\|_{\cB^{s,\f1p}_p}  & \lesssim &
\|a_0\|_{L^p}^{1-s-\f1p}\|\nabla a_0\|_{L^p}^{s+\f1p} \exp \Bigl(C\int_0^t U_\e(t') dt'\Bigr)
\with \\
U_\e(t') & \eqdefa &  \|\nabla_{\rm h} u^{\rm h}_\e(t')\|_{L^\infty}
+  \frac 1 \e \|\partial_3u^{\rm h}_\e (t')\|_{L^\infty}+\e
\|\nabla_{\h} u^3_\e(t')\|_{L^\infty}. \eeno }
\end{prop}

\begin{proof}
Let us change the variable by defining
$$
\wt a_\e (t,x_\h,z) \eqdefa a_\e \Bigl(t,  x_\h,\frac z \e\Bigr) \andf \wt
u_\e (t,x_\h,z)\eqdefa \biggl(u_\e^\h \Bigl( t,x_\h,\frac z \e\Bigr), \e
u_\e^3  \Bigl( t,x_\h,\frac z \e\Bigr)  \biggr)\cdotp
$$
The transport equation \eqref{fulltransporteq1} becomes
$$
\left\{\begin{array}{c} \displaystyle \p_t\wt a_\e+\wt u_\e\cdot\na \wt a_\e=0,\\
\displaystyle \wt a|_{t=0}=a_0.
\end{array}\right.
$$
Let us remark that, because~$u_\e$ is divergence free, we
have~$\|\nabla \wt u_\e\|_{L^\infty} \sim U_\e(t)$. It is well-known
that {\sl isotropic} Besov  norms with index less than~Ê$1$ are
propagated   by the Lipschitz norm of the convection velocity. More
precisely (see for instance Theorem~3.14 of\ccite{BCD}), we have for
$s$ in~$ \left]0, 1-1/p\right[,$ \beno
 \|\wt a_\e(t)\|_{\dot B^{s+\frac 1 p}_{p,1}(\R^3)}
& \leq&
 \|a_0\|_{\dot B^{s+\frac 1 p}_{p,1}(\R^3)} \exp\Bigl(C \int_0^t \|\nabla \wt u_\e(t')\|_{L^\infty} dt'\Bigr)\\
& \leq&
 \|a_0\|_{\dot B^{s+\frac 1 p}_{p,1}(\R^3)} \exp\Bigl(C \int_0^t U_\e(t') dt'\Bigr).
\eeno
As we have (see Lemma 4.3 of \cite{CZ5} for instance)
$$
\|a\|_{\cB^{s,\frac 1 p} _{p}} \lesssim \|a\|_{\dot B^{s+\frac 1
p}_{p,1}(\R^3)} \andf
 \|a_0\|_{\dot B^{s+\frac 1 p}_{p,1}(\R^3)}  \lesssim \|a_0\|_{L^p}^{1-s-\frac 1p} \|\nabla a_0\|_{L^p}^{s+\frac1p},
  $$
 the proposition is proved because~$\|a(t)\|_{\cB^{s,\frac 1 p} _{p}} \sim \|\wt a(t)\|_{\cB^{s,\frac 1 p} _{p}}$.
\end{proof}

 Under the induction hypothesis \eqref {defininductiontime1}, we have the following corollary.
  \begin{col}
\label {fulltransport0} {\sl Let $(a_\e,u_\e, \wt{\Pi}_\e)$ be the
smooth enough solution of (INS3D) on  the maximal time interval~$[0,T^\star_\e [.$ Then  for
any~$s$ in~$]0,1-1/p[,$ a constant~$C$ exists such that, for
any~time~$t$ less than~$\overline T_\e $ defined by\refeq {defininductiontime1}, we~have
$$
\|a_\e(t)\|_{\cB^{s,\f1p}_p}   \leq
 \cC_0\eta \exp
(C\cR_0).
$$
}
\end{col}

\begin{proof} In view of \eqref{S2eq0aq}, we have
\beno \int_0^tU_\e(t')dt'\lesssim \int_0^t\bigl(\|\na
v^\h(t')\|_{L^\infty}+\e^2\|\na w_\e(t')\|_{L^\infty}+\|\na
R_\e(t')\|_{L^\infty}\bigr)dt', \eeno which together with
Propositions \ref{consequenceSection345} and \ref{estimwdetail}
ensures that \beno \int_0^tU_\e(t')dt'\lesssim \cC_0+\cR_0. \eeno
Then applying Proposition \ref{fulltransport}, we conclude the proof
of the corollary.
\end{proof}

With Corollary \ref{fulltransport0}, we can establish the estimates
of the terms~$E^2_\e$ and~$E^3_\e.$ Indeed for any $p$ in~$]3,4[,$
$q\in \bigl]p/(p-2), 2p/(p-1)\bigr[$ and $\d$
in~$\bigl]1/q-1/p,1-3/p\bigr[,$ so that $-1+\d+1/p+2/q,
1+1/q-1/p-\d\in ]0,1[.$  Then it follows from \eqref{homovertBss'}
and Lemma \ref{lemBern}
 that
\beno \e\|[\p_z^2v^\h]_\e\|_{L^1_T(\cB^{-1+\d+\f3{p},-\d}_p)} &\leq&
C\e^{1-\d-\f1p}\|\p_zv^\h\|_{L^1_T(\cB^{-1+\d+\f3{p},1-\d}_p)}\\
&\leq& C\e^{1-\d-\f1p}\|\p_zv^\h\|_{L^1_T(\cB^{-1+\d+\f1p+\f2{q},
1+\f1{q}-\f1p-\d}_{q})}. \eeno Applying  the law of product \eqref
{lawofproductanisobasic},  gives
 \beno
\e\|a_\e[\p_z^2v^\h]_\e\|_{L^1_T(\cB^{-1+\d+\f3{p},-\d}_p)} \leq
C\e^{1-\d-\f1p}\|a_\e\|_{L^\infty_T(\cB^{\f2p,\f1p}_p)}\|\p_zv^\h\|_{L^1_T(\cB^{-1+\d+\f1p+\f2{q},
1+\f1{q}-\f1p-\d}_{q})}. \eeno
 Therefore, thanks to \eqref{S7eq8} of Lemma \ref{S7lem4.5} and  Corollary\refer{fulltransport0}, we conclude \beq\label{estiE2}
\|E^2_\e\|_{L^1_T(\cB^{-1+\d+\f3{p},-\d}_p)}\leq \cC_0
\e^{1-\d-\f1p}\exp(C\cR_0). \eeq Similarly we deduce from the law of
product \eqref {lawofproductanisobasic} that
$$\longformule{\bigl\|a_\e \bigl[ \D_\e (\e w^\h_\e  ,w^3_\e ) -(0,\partial_z\Pi^\h)-\e\nabla_\e\Pi^1_\e \bigr]_\e\|_{L^1_T(\cB^{-1+\f2p,\f1p}_p)}
\leq \|a_\e\|_{L^\infty_T(\cB^{\f1p,\f1p}_p)}}{{}\times
\Bigl(\|\D_\e (\e w^\h_\e  ,w^3_\e
)-\e\nabla_\e\Pi^1_\e\|_{L^1_T(\cB^{-1+\f3p,\f1p}_p)}+\|\p_z\Pi^\h\|_{L^1_T(\cB^{-1+\f3p,\f1p}_p)}\Bigr),
}
$$ which together with   Propositions \ref{estimwdetail} and \ref
{consequenceSection345}, and \eqref{S0eq2} ensures that
 \beq\label{estiE3}
\|E^3_\e\|_{L^1_T(\cB^{-1+\f2p,\f1p}_p)}\leq \cC_0\eta\exp(C\cR_0).
\eeq

   The term~$E_\e^4$ is much more difficult to be treated. It is indeed here that we encounter
    the difficulty of our method (which comes from the  framework of parabolic system) due to the transport equation.
     Let us investigate the equation on~$b_\e$ which is
\beq \label {transportequationb} \p_tb_\e+u_\e\cdot\na b_\e
+R_\e\cdot\nabla [a^\h]_\e+\e[w_\e\cdot\na a^\h]_\e=0\with
{b_\e}_{|t=0}=0. \eeq

The control  of~$b_\e$ is given by the following proposition.
\begin{prop}
\label {propbRt} {\sl Under the induction hypothesis~\refeq
{defininductiontime1},  we can decompose~$b_\e=\overline b_{\e}+\wt
b_{\e}$ such that, for any~$p$ in~$]3,4[,$   there holds, for
any~$t$ less than~$\overline T_\e $,
\ben
\label{4.13} \|\overline
b_\e(t) \|_{\cB^{\f2p,\f1p}_p} &\leq & \cC_0\eta(1+\cR_0)
\w{t}^{\frac12}
\andf\\
\label{4.18} \|\wt b_\e(t)\|_{L^p} &\leq &
\cC_0\e^{1-\f1p}(1+\cR_0)^2\w{ t}. \een }
\end{prop}
Let us notice that the norms of ~$b_\e$ grows in time. As we
need~$L^1$ in time control on the remainder term $R_\e$, it seems a
disaster. In fact, it is compensated by the time decay of~$v^\h$
established in Section\refer  {decayvh} to Section
\ref{decayL2partialz2}. The proof of this proposition is the purpose
of Section\refer {control_b}. Then we can obtain the following
estimate for $E^4_\e=E^{4,1}_\e+E^{4,2}_\e$,
 \beq
 \label{estiE4}
\|E^{4,1}_\e\|_{{L^1_T(\cB^{-1+\f2p,\f1p}_p)}}\leq \cC_0\eta(1+\cR_0)\andf
\|E^{4,2}_\e\|_{{L^1_T(\cB^{-1+\d+\f3p,-\d}_p)}}\leq
\cC_0\e^{1-\f1p}(1+\cR_0)^2,
\eeq which is the content of Corollary
\ref{Sectcolcontrolb}.

\medbreak

Now we are in position to  solve globally the coupled system~${\rm
(INS3D)_{\rm e}}$ with \eqref {transportequationb}. Let us think
this system as a perturbation of a semi-linear parabolic system. We
first compute~$\nabla \Pi_\e$. The point is that the resolution
operator of the elliptic system
$$
\dive \bigl((1+a)\nabla \Pi- f\bigr) =0
$$
 can be written as
$$
\D \Pi +\dive (a\nabla \Pi) =\dive f
$$
and then \beq \label{eqpressure} (\Id-\cM_a )\nabla \Pi  =\nabla
\D^{-1} \dive f\with \cM_a g\eqdefa  -\nabla \D^{-1}\dive (ag) \eeq
It is obvious that if~$a$ is a bounded function, the
operator~$\cM_a$ is a bounded linear  operator from~$(L^2(\R^d))^d$
into itself and that
$$
\|\cM_a g \|_{L^2} \leq \|a\|_{L^\infty} \|g\|_{L^2}.
$$
Thus, if~$\|a\|_{L^\infty}$ is less than~$1$, the
operator~$\Id-\cM_a$ is invertible on~$(L^2(\R^d))^d,$ and \beq
\label{eqpressure+} \nabla \Pi  = (\Id-\cM_a)^{-1} \nabla
\D^{-1}\dive f. \eeq This leads to the following definition of the
modified Leray projection operator on divergence free vector field.
\begin{defi}
\label{definmodifiedLerayproj} {\sl Let~$a$ be a bounded function
with the~$L^\infty$ norm of which is less than~$1$. We can define
the modified Leray projection operator on divergence free vector
fields associated with~$a$ (denoted~$\PP_a$) by
$$
\PP_a f\eqdefa f- (1+a)(\Id-\cM_a)^{-1} (\nabla\D^{-1}\dive f).
$$
}
\end{defi}
Let us remark that it is a bounded operator on ~$(L^2(\R^d))^d$  and
that if the function~$a$ is identically equal to~$0$, then the
operator~$\PP_a$ is the classical Leray projection operator on
divergence free vector fields.

Moreover, in the case when the~$L^\infty$ norm of~$a_\e$ is less
than~$1$, the system~${\rm (INS3D)_{\rm e}}$  can be equivalently
reformulated as
\begin{equation}
\label{S1eq5}
 \left\{\begin{array}{c}
\partial_tR_\e  -\D R_\e  =\PP_{a_\e}  \bigl(a_\e\D R_\e - \dive \bigl(u_{\e,{\rm app}}\otimes R_\e +R_\e\otimes \uapp +\e R_\e\otimes  R_\e\bigr)-E_\e\bigr),\\
\displaystyle \dive R_\e=0\andf {R_\e}|_{t=0}=0.
\end{array}
\right.
\end{equation}

We shall conclude the proof of Theorem\refer {insslowvar}  in
Section\refer {conclusif}  by proving that the solution of the
coupled system of Equations \eqref {transportequationb} and
\eqref{S1eq5} is global provided that~$\eta$ and~$\e$ are small
enough.

\setcounter{equation}{0}
\section{Decay estimates for  2D  flows}
\label {decayvh} In this section, we investigate the decay
properties of  the global regular solution $(\r, u, \na \Pi)$ of two
dimensional incompressible inhomogeneous Navier-Stokes system
(INS2D).

In this section, we use the following   notations:
$$
\displaylines { E_0(t)  \eqdefa \|\sqrt {\rho} u (t)\|_{L^2} ^2\,,\
E_1(t) \eqdefa \|\nabla u(t) \|_{L^2}^2 \,,\ E_2(t) \eqdefa \|\sqrt
{\rho} \partial_t u(t) \|_{L^2} ^2 +\|\nabla u(t) \|_{L^2}^4,\cr
E_3(t) \eqdefa \|\nabla \partial_t u(t) \|_{L^2} ^2 +E_2^{\frac
32}(t) +E_2(t)\|\nabla \rho(t)\|^2_{L^\infty},\ E_i\eqdefa E_i(0),\
i=0,1,2,3,\cr
 \cC(E_0) {\ \rm is\
an\ increasing\ function\ of}\ E_0,\cr a\lesssim b \Longrightarrow
a\leq \cC(E_0) b. }
$$
Moreover, in this section, we denote~$v_t\eqdefa \partial_t v$ and we denote by~$x$ a generic point of~$\R^2$.
The main result of this section is the following  theorem.
\begin{thm}
\label{ins2Ddecay} {\sl Let us consider  the smooth solution $(\r,
u,\na\Pi)$ of (INS2D) associated with the  initial data~$(\rho_0, u_0).$ In
addition we assume that
 \beq
\label {ins2Ddecahyp1} U_0 \eqdefa \int_{\R^2} |x| \,|u_0(x)|
dx<\infty, \quad \int_{\R^2} \rho_0u_0(x) dx =0,\andf \frac 34\leq
\rho_0(x) \leq \frac 5 4 \cdotp \eeq For  any~$T$ greater than or
equal~$T_0(\r_0,u_0)$ with
$$
T_0(\r_0,u_0)\eqdefa \max\Bigl\{\frac {U_0} {E_0},
\|\varrho_0\|_{L^2} ^2\Bigr\},
$$
we have the following decay property for the total kinetic energy
 \beq
\label {ins2Ddecayeq10} \|u(t)\|^2_{L^2}   \lesssim    E_0 \langle t
\rangle _{{}_T}^{-2} . \eeq
For higher order  derivatives of $u,$ we
get for $T$ large enough,
 \ben
 && \label {ins2Ddecayeq20} \|\nabla
  u(t)\|_{L^2}+\|  u(t)\|_{L^\infty}    \lesssim  E_1^{\f12}  \w{t}_{{}_T}^{-\f32},\\
&&\label {ins2Ddecayeq3} \| \partial_t u (t)\|_{L^2}+ \|\nabla^2
u(t)\|_{L^2}+\|\nabla \Pi(t)\|_{L^2}  \lesssim  E_2^{\f12}  \w{t}_{{}_T}^{-2}, \\
&&\label {ins2Ddecayeq5} \|\nabla  \partial_t u (t)\|_{L^2}\lesssim
E_3^{\f12} \w{t}_{{}_T}^{-\frac 5 2} \andf \|\nabla^3
u(t)\|_{L^2}+\|\nabla^2 \Pi (t)\|_{L^2}  \lesssim E_3^{\f12}
 \w{t}_{{}_T}^{-\frac 5 2}\log \w{t}_{{}_T}, \\
&&\label {ins2Ddecayeq6} \|\na u(t)\|_{L^\infty}\lesssim
E_1^{\f14}E_3^{\f14} \w{t}_{{}_T}^{-2}\log^{\f12} \w{t}_{{}_T}.
 \een
 Here and in the rest of this section, we always denote
 $$
 \langle \tau\rangle \eqdefa
 (e+\tau)\,,\langle t\rangle_{{}_T} \eqdefa   \langle t/T \rangle\andf  h_T(t)\eqdefa h (t/T).
 $$
 }
 \end{thm}

\subsection{Global energy estimates for the linearized system}
\label {GlobalEnergy}

 Let $(\r, u, \na\Pi)$ be a global
classical solution of (INS2D). We first consider some  basic energy
estimate for linearized equations of 2-D inhomogeneous
incompressible Navier-Stokes system
\begin{equation*}
{\rm (LINS2D)}\left\{\begin{array}{c} \displaystyle \p_t\r +u\cdot\na\r =0, \\
\displaystyle \r \pa_t  v  + \r   u \cdot\nabla  v  -\D   v +\grad   \Pi_v =f+L(t)v,\\
\displaystyle \dv \, u   = \dv \, v =0, \\
\displaystyle \r |_{t=0}=\r_0,\quad v|_{t=0}=v_0.
\end{array}\right. 
\end{equation*}

\no$\bullet$ \underline{$L^2$ energy estimate}

If the operator~$L$ is such that~$\|L(t)\|_{\cL(L^2)}$ belongs
to~$L^1(\R^+)$, then by multiplying ~(LINS2D) by the quantity
$$
\exp\Bigl( -\int_0^t\|L(t')\|_{\cL(L^2)} dt'\Bigr)
$$
reduces  to the case when~$L(t)$ is a non positive operator in the
sense that~$(L(t)v|v)_{L^2}$ is non positive. Then the operator~Ê$L$
can be ignored in the energy estimates. We assume this from now on.

We shall assume that  all the vectors fields and functions are
smooth in time with value in any Sobolev space.

First of all, let us  notice that  the  energy estimate, obtained by
taking into account to the fact that  the vector field~$u$ is
divergence free, writes
\beq
\label {Energybasic}
\f12\f{d}{dt}\|\sqrt{\r }v(t)\|_{L^2}^2+\|\na v(t)\|_{L^2}^2\leq
(f|v)_{L^2},
\eeq
By integration this gives
\beno
 \f12\|\sqrt{\r }v(t)\|_{L^2}^2+\int_{t_0} ^t \| \nabla v(t')\|_{L^2}^2\,dt'
 & \leq &
\f 12 \|\sqrt{\r }v(t_0)\|_{L^2}^2 + \int_{t_0}^ t
(f(t')|v(t'))_{L^2}\\
 & \leq &
 \f 12 \|\sqrt{\r }v(t_0)\|_{L^2}^2 + \int_{t_0}^ t
\|\f{f}{\sqrt{\r}}(t')\bigr\|_{L^2} \|\sqrt{\r }v(t')\|_{L^2}  dt'.
\eeno From this, we deduce that for any non negative~$t_0$ and
any~$t$ greater than~$t_0$ \beq \label {Energybasicinteg}
 \f12\|\sqrt{\r }v\|_{L^\infty([t_0,t];L^2)}^2+\int_{t_0} ^t \| \nabla v(t')\|_{L^2}^2\,dt'\leq
\|\sqrt{\r }v(t_0)\|_{L^2}^2 +2\Bigl( \int_{t_0}^ t
\bigl\|\f{f}{\sqrt{\r}}(t')\bigr\|_{L^2} dt'\Bigr)^2. \eeq In
particular, since  $(\r, u, \na\Pi)$ is a classical solution of
(INS2D), the above argument leads to \beq \label{zp1}
\f12\|\sqrt{\r}u(t)\|_{L^2}^2+\int_{t_0}^t\|\na
u(t')\|_{L^2}^2\,dt'=\f12\|\sqrt{\r}u(t_0)\|_{L^2}^2\quad
\mbox{for}\ \ 0\leq t_0\leq t, \eeq which implies in particular that
\beq \label{zp2} \int_{t_0}^t\|u(t')\|_{L^4}^4\,dt'\lesssim
\|u\|_{L^\infty([t_0,t];L^2)}^2\|\na
u\|_{L^2([t_0,t];L^2)}^2\lesssim \|u(t_0)\|_{L^2}^4.
 \eeq

Moreover, the fact that  the vector field~$u$ is divergence free
implies that \beq \label {Transportbasic} \min_{x\in \R^2} \rho(t,x)
= \min_{x\in \R^2} \rho_0(x) \andf \forall p\in [1,\infty]\,,\
\|\varrho(t)\|_{L^p} =  \|\varrho_0\|_{L^p}.
\eeq

\no$\bullet$ \underline{The estimates for the first order derivatives}

  The basic result  is the
following lemma.

\begin{lem}
\label {DecayH1fond} {\sl Let~$\rho,$~$u$,~$v$ and~$f$
satisfy~(LINS2D). Then for any non negative~$t_0$ and $t$ with ~$t$
greater than or equal to~$t_0$, we have, \beq\begin{split} &\|\na
v(t)\|_{L^2}^2+\int_{t_0}^t\Bigl(\|\sqrt{\r }\p_t v(t') \|_{L^2}^2+
\|\nabla^2 v(t') \|_{L^2}^2+ \|\nabla \Pi_v(t')
\|_{L^2}^2\Bigr) dt' \\
&\qquad\qquad\qquad\qquad\quad \leq C\Bigl( \| \nabla
v(t_0)\|_{L^2}^2+
 \int_{t_0}^t \|f(t')\|_{L^2}^2 dt' \Bigr) \exp \bigl(C\|u(t_0)\|_{L^2}^4
 \bigr). \end{split} \label{LemDecayH1fondeq1}
 \eeq
Moreover some  decay on~$\nabla v$ can be obtained through the
following inequalities. If~$s$ is a positive real number, we have.
 \beq\begin{split} & \w{t}_{{}_T}^s\|\na
v(t)\|_{L^2}^2+\int_{t_0}^t\w{t'}_T^s\Bigl(\|\sqrt{\r }\p_t v(t')
\|_{L^2}^2+  \|\nabla^2 v(t') \|_{L^2}^2+ \|\nabla \Pi_v(t')
\|_{L^2}^2\Bigr) dt'\\
&\qquad \leq C\Bigl( \w{t_0}^s_T\|\na v(t_0)\|_{L^2}^2+ \int_{t_0}^t
\w{t'}_T^{s-1}\| \nabla
v(t')\|_{L^2}^2\frac{dt'}T+  \int_{t_0}^t \w{t'}_T^{s}\|f(t')\|_{L^2}^2 dt' \Bigr)\\
 &\qquad\qquad\qquad\qquad\qquad\qquad\qquad\qquad\qquad\qquad\qquad\qquad\qquad\times \exp \bigl(C\|u(t_0)\|_{L^2}^4
 \bigr), \end{split} \label{LemDecayH1fondeq2}
 \eeq
 and
 \beq \label{LemDecayH1fondeq3}
t\|\nabla v (t)\|_{L^2}^2 \leq C \Bigr(\|v(t/2)\|^2_{L^2} +
t\int_{t/ 2}^t  \|f(t')\|_{L^2}^2dt'\Bigr)
  \exp\bigl(C\|u(t/2)\|_{L^2}^4\bigr).
 \eeq}
\end{lem}

\begin{proof}
 Multiplying the momentum part of~(LINS2D)  by $ v_t$ and
integrating the resulting equation over $\R^2,$ we obtain
$$
\|\sqrt{\r } v_t\|_{L^2}^2+\f12\f{d}{dt}\|\na
v(t)\|_{L^2}^2=-\bigl( \r   u \cdot\na v | \p_t
v\bigr)_{L^2}+(f| v_t)_{L^2} .
$$
As~$\rho$ lies between~$1/2$ and~$2$, we get
$$
\frac 3 4 \|\sqrt{\r } v_t\|_{L^2}^2+\f12\f{d}{dt}\|\na
v(t)\|_{L^2}^2\leq 4\|(u \cdot\na v)(t)\|_{L^2}^2+4\|f(t)\|_{L^2}^2
.
$$
It is important now to make precise the idea in the framework of
Stokes problem, one time derivative of $v$ is equivalent to two
space derivatives of $v.$ Here, as the equation is nonlinear, the
equivalent inequality will be nonlinear one.   This is described by
the following lemma.
\begin{lem}
\label {Echanget2x} {\sl Let~$(\r,v,\na\Pi_v)$ be a solution
of~(LINS2D). Then we have, for any~$p$ in the interval~$]1,\infty[$,
$$
 \|\nabla^2v\|_{L^p}+\|\na  \Pi_v \|_{L^p}  \leq
C\bigl(\|\sqrt{\r} v_t\|_{L^p} + \|u\|_{L^{2p}}^{2} \|\nabla
v\|_{L^2}+\|f\|_{L^p}\bigr).
$$
In the case when~$p$ equals to~$2$, we have the opposite inequality
$$
\|\sqrt \r v_t\|_{L^2} ^2 \leq C\bigl(\|\nabla^2v\|^2_{L^2} +
\|u\|_{L^4}^4\|\nabla v\|_{L^2}^2+\|f\|_{L^2}^2\bigr) .
$$}
\end{lem}

\begin{proof}
Observing that
\begin{equation*}
{\rm (SSE)}\quad \left\{\begin{array}{c}
\displaystyle  -\D   v +\grad   \Pi_v =f-\r \pa_t  v - \r u \cdot \nabla  v ,\\
\displaystyle  \dv \, v =0,
\end{array}\right.
\end{equation*}
we deduce from the classical estimate on Stokes operator that for
any $p$ in~$ ]1,\infty[$
\beq
\label {Echanget2xdemoeq1}
\begin{split}
 \|\nabla^2v\|_{L^p}+\|\na  \Pi_v \|_{L^p} \leq & C\bigl(\|f\|_{L^p}+\|\r  v_t\|_{L^p}+\|\r    u\cdot \nabla  v
 \|_{L^p}\bigr)
 \\
 \leq & C\bigl(\|f\|_{L^p}+\|\sqrt{\r} v_t\|_{L^p}+\|u\|_{L^{2p}} \|\nabla v\|_{L^{2p}}\bigr).
\end{split}
\eeq
By using the 2-D interpolation inequality that
 \beq
 \label {1.14a}
 \|a\|_{L^{2p}(\R^2) }\leq
C_p\|a\|_{L^2(\R^2)}^{\f12 }\|\nabla a\|_{L^p(\R^2)}^{ \f 12}, \eeq
we get
 \beno
 \|\nabla^2v\|_{L^p}+\|\nabla  \Pi_v \|_{L^p}
   \leq  \f 12  \| \nabla^2  v \|_{L^p  } +C\bigl(\|f\|_{L^p}+ \|\sqrt{\r}v_t\|_{L^p}+\|  u\|^2_{L^{2p}} \| \nabla  v \|_{L^2}\bigr).
\eeno
This proves the first inequality.  Because~$v_t$ is divergence
free,  we get, by taking the~$L^2$ scalar product of~$v_t$
with~$(SSE)$, that
\beno \|\sqrt  \r v_t\|_{L^2} ^2
& = &(f|v_t)_{L^2}+  (\D v|\p_tv)_{L^2}- \bigl(\sqrt \rho u\cdot\na v\big |\sqrt\rho v_t \bigr)_{L^2}\\
& \leq &
\Bigl(\sqrt 2 \bigl(\|f\|_{L^2}+ \|\D v\|_{L^2} \bigr)
 + C \|u\|_{L^4} \|\nabla v\|_{L^2}^{\frac 12}
\|\nabla^2 v\|^{\frac 1 2}_{L^2}\Bigr)\|\sqrt \r v_t\|_{L^2}.
\eeno
H\"older inequalities implies that
$$
\|\sqrt  \r v_t\|_{L^2} ^2  \leq \frac 12 \|\sqrt  \r v_t\|_{L^2}
^2 + C \bigl(\|f\|_{L^2}^2+ \|\nabla^2 v\|_{L^2}^2+C
 \|u\|_{L^4}^2\|\nabla v\|_{L^2}\|\nabla^2 v\|_{L^2}\bigr).
 $$
This  leads to the second inequality and the lemma is proved.
 \end{proof}

\noindent{\it Continuation of the proof of Lemma\refer
{DecayH1fond}} Applying the above lemma in the case when~$p$ equals to~$2,$ we
obtain  that
\beq
\label {1.3b}
\f{d}{dt}\|\na v(t)\|_{L^2}^2+\|\sqrt{\r
} v_t\|_{L^2}^2+\frac 1{C }\bigl(  \|\nabla^2v\|_{L^2}^2+ \|\na
\Pi_v \|_{L^2}^2\bigr)\leq C \bigl(\|f\|_{L^2}^2+ \|u\|_{L^4}^4
\|\nabla v\|_{L^2} ^2\bigr).
\eeq
Gronwall lemma implies that, for any non negative~$t_0$ and $t$ such that~$t$ is greater than or equal to~$t_0$,
$$
\longformule{ \|\na v(t)\|_{L^2}^2+\int_{t_0}^t\Bigl(\|\sqrt{\r
} v_t(t') \|_{L^2}^2+\frac 1{C }  \|\nabla^2 v(t') \|_{L^2}^2
+\frac 1{C} \|\na  \Pi_v(t') \|_{L^2}^2\Bigr) dt'
 }
 {
 {} \leq C\Bigl( \|\na v(t_0)\|_{L^2}^2+
 \int_{t_0}^t \|f(t')\|_{L^2}^2 dt' \Bigr) \exp \Bigl(C\int_{t_0}^t  \|u(t')\|_{L^4}^4 dt'\Bigr)
 }
$$
which together with \eqref{zp2} implies \eqref{LemDecayH1fondeq1}.

 Let
us prove \eqref{LemDecayH1fondeq2}. For any positive~$s$, by multiplying
$ \w{t}_{{}_T}^s$ to \eqref{1.3b}, we get \beno
 &&\f{d}{dt}\bigl(\w{t}^s_T\|\na v(t)\|_{L^2}^2\bigr)+\w{t}^s_T\bigl(\|\sqrt{\r
}\p_t v\|_{L^2}^2+ \|\nabla^2 v\|_{L^2}^2+ \|\na \Pi_v
\|_{L^2}^2\bigr)\\
&&\qquad\qquad\qquad\leq C \bigl(\w{t}^s_T\|f\|_{L^2}^2+
\|u\|_{L^4}^4 \w{t}^s_T\|\nabla v\|_{L^2} ^2+\frac 1 T \w{t}^{s-1}_T\|\na
v\|_{L^2}^2\bigr).
 \eeno Applying Gronwall's Lemma and using
\eqref{zp2} leads to \eqref{LemDecayH1fondeq2}.

Similarly from Inequality\refeq{1.3b}, we deduce that for any
non-negative~$t_0$ and $t$ such that~$t$ is greater  than or equal  to~$t_0$,
 \beno
\f{d}{dt}\bigl((t-t_0)\|\na v(t)\|_{L^2}^2\bigr)   &= & \|\na v(t)\|_{L^2}^2+ (t-t_0)\f{d}{dt}\|\na v(t)\|_{L^2}^2 \\
 & \lesssim  & \|\na v(t)\|_{L^2}^2 + (t-t_0)  \|f(t)\|_{L^2}^2+
\|u(t)\|_{L^4}^4 (t-t_0)\|\nabla v(t)\|_{L^2} ^2.
\eeno
Applying Gronwall's inequality yields
$$
{
 (t-t_0)\|\na v(t)\|_{L^2}^2 \lesssim  \int_{t_0}^t\bigl(\|\na v(t')\|_{L^2}^2+(t'-t_0)
 \|f(t')\|_{L^2}^2\bigr)\,dt'
 }
 {
 {}
 \exp\Bigl(C\int_{t_0} ^t  \|u(t')\|_{L^4}^4\,dt'\Bigr),
 }
 $$
 which together the energy estimate on~$v$, namely Inequality\refeq{Energybasicinteg}, ensures
 $$
\longformule{ (t-t_0)\|\nabla v (t)\|_{L^2}^2 \lesssim
\biggr(\|v(t_0)\|^2_{L^2} + \Bigl(\int_{t_0}^t  \|f(t')\|_{L^2} dt'
\Bigr)^2+ \int_{t_0}^t (t'-t_0)  \|f(t')\|_{L^2}^2\,dt'\biggr)
 }
 {
 {}
\times  \exp\Bigl(C\int_{t_0} ^t  \|u(t')\|_{L^4}^4\,dt'\Bigr).
 }
 $$
 Taking $t_0$ equal to~$t/2$ in the above inequality gives
 \beno t\|\nabla v (t)\|_{L^2}^2 \lesssim
\biggr(\|v(t/2)\|^2_{L^2} +  t\int_{t/2}^t
\|f(t')\|_{L^2}^2\,dt'\biggr)
  \exp\Bigl(C\int_{t/2} ^t  \|u(t')\|_{L^4}^4\,dt'\Bigr).
 \eeno
This together with \eqref{zp2}  concludes the proof of
\eqref{LemDecayH1fondeq3}.
\end{proof}

\subsection{Sharp~$L^2$  decay estimates}
The purpose of this subsection is to prove the following
proposition.
\begin{prop}
\label {propDecayL2H1}
 {\sl Let $T_0(\r_0,u_0)$ be given by Theorem \ref{ins2Ddecay} and
$
 T_1(\r_0,u_0)\eqdefa \max\bigl\{T_0,
{E_0}/{E_1}\bigr\}. $ Then under the assumptions of Theorem
\ref{ins2Ddecay}, there holds \eqref{ins2Ddecayeq10} for $T\geq
T_0.$ And for $T\geq T_1,$
 we have following decay estimate
 \ben
\label {ins2Ddecayeq2} \|\nabla u(t)\|^2_{L^2}   \lesssim  E_1
\langle t \rangle _{{}_T}^{-3}. \een }
\end{prop}

\begin{proof}
It  is based on the method introduced
by M. {Wiegner}  in\ccite  {Wiegner} in order to study the decay of
the energy of the classical Navier-Stokes system in two space
dimension (see also\ccite{Schonbek, Schonbek2}).
 The idea is to use a cut off in the frequency space adapted to  time.
 More precisely, let us consider a positive constant~$T$  (which can be understood as a scaling parameter which has the dimension of  time), and~$g$ any  positive real function  defined on~$\R^+$ such that
\beq
\label {LemmedecayWiegnerdemoeq2b}
g^2(\tau)\leq 3 \langle \tau\rangle^{-1} \with \langle \tau\rangle \eqdefa (e+\tau).
\eeq
 Let us define
\beq
\label {LemmedecayWiegnerdemoeq2a}
\begin{split}
 & S_T (t)\eqdefa \bigl\{\xi\in\R^2,\ \ \sqrt T|\xi| \leq\sqrt  2   g_T(t) \bigr\} \andf v_\flat (t) \eqdefa \cF^{-1}
\bigl({\bf 1}_{S_T(t)} \wh v(t)\bigr).
\end{split}
\eeq   Here we adapt this method to the  inhomogenenous  case through the following lemma.
\begin{lem}
\label {LemmedecayWiegnerbasic} {\sl Let~$(\rho, u, v)$  solve
~$(LINS2D)$. Then we have
$$
 \f12\f{d}{dt}\|\sqrt \r v(t)\|_{L^2}^2+\frac 1 T  g^2_T(t)\|\sqrt \r v(t)\|^2_{L^2}\\
\leq \frac 2 T   g^2_T(t)\|v_\flat(t)\|_{L^2}^2+(f(t)|v(t))_{L^2} .
$$
}
\end{lem}

\begin{proof}  As~$v(t)-v_\flat(t)$ and~$v_\flat(t)$ are orthogonal in all Sobolev spaces, we get in particular that
$$
\|\nabla v(t) \|_{L^2}^2 = \|\nabla v_\flat(t)\|_{L^2}^2+\|\nabla
(v(t)-v_\flat(t))\|_{L^2}^2.
$$
By defintion of~$S_T(t)$ and again by the orthogonality
between~$v(t)-v_\flat(t)$ and~$v_\flat(t)$, we get
\beno
\|\nabla (v(t)-v_\flat(t))\|_{L^2}^2 & \geq & \f2T g^2_T(t) \|v(t)-v_\flat(t)\|_{L^2} ^2 \\
& \geq & \f 2 T g^2_T(t) \|v(t)\|_{L^2}^2 -\f2T g^2_T(t)
\|v_\flat(t)\|_{L^2}^2. \eeno As~$\rho(t,x)$ is less than or equal
to~$2$, the energy estimate\refeq {Energybasic} implies the lemma.
\end{proof}

The interest of this lemma is that the term on the left is typically
a term that creates decay. Of course,  the control of the
term~$v_\flat$ associated with (very)  low  frequencies is the term
that tends to prevent the decay. It must be estimated in a careful
way.  Writing a general theory with external force seems too
ambitious.  We are going to restrict ourselves to two cases: the
case when~$u=v$ and $f\equiv 0,$ namely, the case of solution
of~(INS2D), and later on the case  of a family of
solution~$v^\h(\cdot,z)$ of~(INS2D) where $z$ is a real parameter
and then~$v$ represents derivatives of ~$v^\h(\cdot,z)$  with
respect to the parameter~$z$.
\begin{lem}
\label {estimuflat} {\sl Under the hypothesis of Proposition
\refer{propDecayL2H1}, we have, for any~$T$ greater than or equal
to~$T_0$,
$$
\longformule{ \|u_\flat(t)\|_{L^2}^2 \leq \frac 14 \|\sqrt \r
u(t)\|_{L^2}^2 + CE_0^2 \langle t\rangle_{{}_T}^{-2} } {{} +
Cg_{T}^6(t) \Bigl(\int_0^t \|\sqrt\r u (t')\|_{L^2} \frac {dt'}
T\Bigr)^2 +Cg_{T}^4(t) \Bigl(\int_0^t \|\sqrt\r u (t')\|^2_{L^2}
\frac {dt'} T\Bigr)^2\cdotp }
$$
}
\end{lem}

\begin{proof}
It relies on the rewriting of  the momentum  equation of~(INS2D) as
 \beq \label
{estimuflatdemoeq1}
\partial_t u -\Delta u +\nabla \Pi = - \partial_t (\varrho u) -\dive \bigl(\r   u \otimes u\bigr).
\eeq
If~$\PP$ denotes the Leray projection on divergence free vector fields on~$\R^2$, the above relation writes in term of Fourier transform
$$
\longformule{ \wh u(t,\xi) =e^{-t|\xi|^2 }\wh
u_0(\xi)-\int_0^te^{-(t-t')|\xi|^2}\p_t\cF\mathbb{P}(\varrho u
)(t',\xi)\,dt' } { {} -\int_0^t  e^{-(t-t')|\xi|^2}\cF\PP(\dive (\r
u \otimes u)\bigr)(t',\xi)\,dt'.
 }
$$
By integration by parts in time, we get that
$$
\longformule{ \int_0^te^{-(t-t')|\xi|^2} \p_t\cF\mathbb{P}(\varrho u
)(t',\xi)dt'  = \cF\mathbb{P}(\varrho u)(t,\xi) } { {} -
e^{-t|\xi|^2}\cF\mathbb{P}(\varrho_0 u_0)(\xi) -\int_0^t
e^{-(t-t')|\xi|^2}|\xi|^2\cF\mathbb{P}(\varrho u)(t',\xi)\,dt'.
 }
$$
Let us notice that in the integral term, we exchange one time
derivative for two space derivatives. This gives the following key
formula \beq \label  {estimuflatdemoeq2}
\begin{split}
&\wh u(t,\xi) =e^{-t|\xi|^2 }\cF\PP(\r_0 u_0)(\xi)-\cF\mathbb{P}(\varrho u)(t,\xi) \\
+\int_0^t  e^{-(t-t')|\xi|^2}|\xi|^2&\cF\mathbb{P}(\varrho u
)(t',\xi)\,dt'-\int_0^t  e^{-(t-t')|\xi|^2}\cF\PP(\dive (\r u
\otimes u )\bigr)(t',\xi)\,dt'.
\end{split}
\eeq
Because~$\PP$ decreases the modulus of the Fourier transform, we get for any~$t$ and~$\xi$,
\beq
\label  {LemmedecayWiegnerdemoeq4}
\begin{split}
&|\wh u(t,\xi)|^2\leq 2e^{-2t|\xi|^2}\bigl|\cF(\r_0
u_0)(\xi)\bigr|^2 +2\bigl|\cF (\varrho u)(t,\xi)\bigr|^2
\\
&\qquad\quad {} +2|\xi |^4\Bigl(\int_0^t\bigl|\cF\bigl(\varrho u
)\bigr)(t',\xi)\bigr|\,dt'\Bigr)^2
+2|\xi|^2\Bigl(\int_0^t\bigl|\cF(\r   u \otimes
u)(t',\xi)\bigr|\,dt'\Bigr)^2.
\end{split}
\eeq
To estimate~$\|u_\flat(t)\|_{L^2},$ we have to integrate the
above inequality over~$S_T(t)$. In order to do it, we make pointwise estimates in the Fourier variable.

First, let us observe that~$u_0$ and thus~$\rho_0 u_0$ belongs
to~$L^1(\R^2,|x|dx)$. Because of the fact that ~$\rho_0 u_0$ is mean
free, we infer that
$$
|\cF(\rho_0u_0) (\xi) |   \leq |\xi| \
\|D_\xi \cF (\rho_0 u_0)\|_{L^\infty}\leq  |\xi|  \int_{\R^2}
\rho_0(x) |u_0(x)|\, |x|\, dx.
$$
 By integration on~$S_T$, this gives, because~$T$ is greater than~$T_0$
and~$g^2(\tau)\leq 3\langle \tau\rangle^{-1}$, \beq \label
{LemmedecayWiegnerdemoeq5} \int_{S_T(t)}
e^{-2t|\xi|^2}\bigl|\cF(\r_0 u_0)(\xi)\bigr|^2 d\xi \lesssim E_0^2
\langle t\rangle_{{}_T}^{-2} \cdotp \eeq Let us  observe that,
thanks to\refeq{Transportbasic}, we get, for~$T$ greater than or
equal to~$T_0$
\beno
\bigl |\cF(\varrho u)(t',\xi)\bigr|  & \leq & \|\varrho(t')\|_{L^2} \|u(t')\|_{L^2}\nonumber\\
 & \leq &  \|\varrho_0\|_{L^2} \|u(t')\|_{L^2}\leq   T^{\frac
12}_0\|u(t')\|_{L^2} . \eeno From this, we infer that, \beq \label
{LemmedecayWiegnerdemoeq6} \int_{S_T(t)} |\xi
|^4\Bigl(\int_0^t\bigl|\cF(\varrho u)(t',\xi)\bigr|\,dt'\Bigr)^2
d\xi \lesssim  g^6_T(t) \Bigl(\int_0^t\|\sqrt \r u (t') \|_{L^2}
\frac {dt'} T\Bigr)^2\cdotp
 \eeq
Along the same lines, we get that
$$
\bigl|\cF(\r   u \otimes u)(t',\xi)\bigr| \lesssim
\|\sqrt{\r}u(t')\|_{L^2}^2.
$$
Thus we get \beq \label  {LemmedecayWiegnerdemoeq7} \int_{S_T(t)}
|\xi|^2\Bigl(\int_0^t\bigl|\cF(\r   u \otimes
u)(t',\xi)\bigr|\,dt'\Bigr)^2d\xi \lesssim g^4_T(t)
\Bigl(\int_0^t\|\sqrt \r u (t') \|_{L^2}^2 \frac {dt'}
T\Bigr)^2\cdotp
\eeq
Because of the hypothesis on~$\rho_0$,  we obviously have
\beno
 \int_{S_T(t)} 2\bigl|\cF
(\varrho u)(t,\xi)\bigr|^2\,d\xi
 & \leq &  4 (2\pi)^2 \|\varrho\|^2_{L^\infty} \|\sqrt \r u(t)\|^2_{L^2}\\
 & \leq & 4 (2\pi)^2 \|\varrho_0\|^2_{L^\infty} \|\sqrt \r
  u(t)\|^2_{L^2}
\leq  \frac 14  (2\pi)^2 \|\sqrt \r u(t)\|^2_{L^2}.
 \eeno
 Together with estimates\refeq  {LemmedecayWiegnerdemoeq5}--(\ref{LemmedecayWiegnerdemoeq7}), we achieve the proof of the lemma.
\end{proof}

\noindent
{\it Continuation of  the proof of Proposition \ref {propDecayL2H1}. }
 The above lemmas give immediately that
 \beq \label
{LemmedecayWiegnerdemoeq3}
 \begin{split}
\f{d}{dt}\|\sqrt \r &u(t)\|_{L^2}^2+\frac 1 T g^2 _T(t)\|\sqrt \r u(t)\|^2_{L^2}
\lesssim \frac 1 T   E_0^2 \langle t\rangle_{{}_T}^{-3} \\
&{}
+
\frac 1 T  g_T^8(t) \Bigl(\int_0^t \|\sqrt\r u (t')\|_{L^2} \frac {dt'} T\Bigr)^2
+\frac 1 T g^6_T(t) \Bigl(\int_0^t \|\sqrt\r u (t')\|^2_{L^2} \frac {dt'} T\Bigr)^2\cdotp
\end{split}
\eeq Let us define \beq \label {LemmedecayWiegnerdemoeq33} G (\tau )
\eqdefa \exp  \Bigl( \int_0^\tau  g^2(\tau') d\tau'\Bigr). \eeq The
above formula  writes after integration \beq \label
{LemmedecayWiegnerdemoeq3b}
 \begin{split}
\|\sqrt \r &u(t)\|_{L^2}^2G_T(t) -E_0
\lesssim   E_0^2 \int_0^t \langle t'\rangle_{{}_T}^{-3}G _T(t') \frac {dt'} T  \\
&{}
+
\int_0^t g_T^8(t') G _T(t')\Bigl(\int_0^{t'} \|\sqrt\r u (t'')\|_{L^2} \frac {dt''} T\Bigr)^2 \frac {dt'} T \\
&+\int_0^t g_T^6(t') G _T(t') \Bigl(\int_0^{t'} \|\sqrt\r u (t'')\|^2_{L^2} \frac {dt''} T\Bigr)^2 \frac {dt'} T \,\cdotp
\end{split}
\eeq

Now we iterate  this inequality several times to get the final decay
estimates of $u$ given by Proposition \refer{propDecayL2H1}. Let us
first choose the function~$g$ as
$$
g^2(\tau ) =  3 \bigl(\langle \tau\rangle  \log\langle \tau\rangle\bigr)^{-1} \quad \hbox{which gives}\quad  G(\tau) = \log ^3\langle \tau\rangle.
$$
Using that~$\|\sqrt\r u(t)\|_{L^2}^2$ is less than or equal to the
initial energy~$E_0$, Inequality\refeq{LemmedecayWiegnerdemoeq3b}
writes \beno \|\sqrt \rho u (t)\|_{L^2}^2 \log^3 \langle t\rangle
_{{}_T}
  - E_0  & \lesssim  &
E_0(1+E_0) \int_0^t  \langle t'\rangle _{{}_T}
 ^{-1}  \frac {dt'} T\\
& \lesssim & E_0(1+E_0) \log  \langle t\rangle _{{}_T} .
\eeno
We deduce that
\beq
\label {smalldecay}
\|\sqrt \r u(t)\|_{L^2}^2 \log^2  \langle t\rangle _{{}_T}
 \lesssim E_0 (   1+ E_0).
\eeq Now let us plug this estimate into
Inequality\refeq{LemmedecayWiegnerdemoeq3b} with the choice $
g^2(\tau) =\w{\tau}^{-1},$ which gives $ G(\tau) = \w{\tau}.$ This
leads to
$$
\longformule{
\|\sqrt \r u(t)\|_{L^2}^2\langle t\rangle _{{}_T}
 -E_0
\lesssim   E_0^2 \int_0^t \langle t'\rangle _{{}_T}
^{-2} \frac {dt'} T
+
\int_0^t  \langle t'\rangle _{{}_T}
^{-3} \Bigl(\int_0^{t'} \|\sqrt\r u (t'')\|_{L^2} \frac {dt''} T\Bigr)^2 \frac {dt'} T
}
{
{}
+\int_0^t \langle t'\rangle _{{}_T}
^{-2} \Bigl(\int_0^{t'} \|\sqrt\r u (t'')\|^2_{L^2} \frac {dt''} T\Bigr)^2
\frac {dt'} T \cdotp
}
$$
Let us define~$V(t)\eqdefa \ds \sup_{t'\leq t}  \bigl(\|\sqrt \r
u(t')\|_{L^2}^2\langle t'\rangle _{{}_T}\bigr).
 $ We get
\beno V(t) -E_0
 & \lesssim  &   E_0^2 \int_0^t \langle t'\rangle _{{}_T}
^{-2} \frac {dt'} T
+
(1+E_0)^2\int_0^t  \langle t'\rangle _{{}_T}
^{-2}  H_T(t') V(t') \frac {dt'} T  \with\\
H(\tau) & \eqdefa & 1 +\Bigl( \int_0^{\tau} \langle \tau'\rangle
 ^{-\frac 12} \log ^{-1}    \langle \tau'\rangle
d\tau'\Bigr)^2.
\eeno
As we have that~$\ds H(\tau) \leq C\bigl (1+ \langle \tau\rangle _{{}_T}
  \log ^{-2}    \langle \tau\rangle _{{}_T}\bigr),$
 the function~$\langle t'\rangle _{{}_T}
^{-2}  H_T(t')$ is integrable and then  Gronwall lemma gives
 $$
 \|\sqrt \r u(t)\|_{L^2}^2\langle t\rangle _{{}_T}
 \leq \cC(E_0) E_0.
 $$
Let us plug this estimate into\refeq{LemmedecayWiegnerdemoeq3b} and
choose~$ g^2(\tau) =2\w{\tau}^{-1}$ which gives~$ G(\tau) =
\w{\tau}^2.$ We infer that \beno \|\sqrt \r u(t)\|_{L^2}^2\langle
t\rangle _{{}_T} ^2-E_0
 & \leq  & E_0^2 \int_0^t \langle t'\rangle _{{}_T}
^{-1} \frac {dt'} T + \cC(E_0)E_0\int_0^t  \langle t'\rangle _{{}_T}
^{-1} \log^2  \langle t'\rangle _{{}_T}
 \frac {dt'} T \\
& \leq  & \cC(E_0)E_0\log^3  \langle t\rangle _{{}_T} . \eeno
Finally resuming the above estimate
into\refeq{LemmedecayWiegnerdemoeq3b} once again with the
choice~$g^2(\tau)=\al\w{\tau}^{-1},$ for $\al\in ]2,3[$ gives
$G(\tau)=\w{\tau}^\al$ and
$$
\longformule{
\|\sqrt \r u(t)\|_{L^2}^2\langle t\rangle _{{}_T} ^\al-E_0
  \lesssim  E_0^2 \int_0^t \langle t'\rangle _{{}_T}
^{\al-3} \frac {dt'} T + \cC(E_0)E_0\int_0^t  \langle t'\rangle
_{{}_T} ^{\al-4} \log^5  \langle t'\rangle _{{}_T}
 \frac {dt'} T}{{}+ \cC(E_0)E_0\int_0^t  \langle t'\rangle _{{}_T}
^{\al-5} \log^6  \langle t'\rangle _{{}_T}
 \frac {dt'} T\, \virgp}
  $$
  which implies the Estimate\refeq{ins2Ddecayeq10}.

Let us prove \refeq {ins2Ddecayeq2}. Applying
\eqref{LemDecayH1fondeq3}  with~$v=u$ and~$f\equiv 0$, and
Inequality\refeq{ins2Ddecayeq10}, we get \beno t\|\nabla u
(t)\|_{L^2}^2 \lesssim \|u(t/2)\|_{L^2}^2 \lesssim E_0\langle
t\rangle_{{}_T}^{-2}, \eeno which write, in the case when~$T$ is
greater than~$T_1$,
$$
\|\nabla u (t)\|_{L^2}^2 \lesssim \frac {E_0} T \f{T}t\langle
t\rangle_{{}_T}^{-2} \lesssim E_1\f{T}t \langle
t\rangle_{{}_T}^{-2}.
$$
Moreover,  Inequality \eqref{LemDecayH1fondeq1} implies that~$ \|\nabla u(t)
\|_{L^2}\lesssim \|\nabla u_0\|_{L^2}^2$ which proves
\eqref{ins2Ddecayeq2}.
\end{proof}

\subsection {Decay estimates for the second and third derivatives of $u$}

The main idea which seems to be the simplest one at the first glance
consists in the differentiation of the momentum equation of (INS2D)
with respect to the space variables and then trying to apply result
of the previous subsections. Nevertheless for this particular
system, this quite natural idea fails. The reason is due to the fact
that term of the type~$\na_x\rho$ will appear in this process. Their
control demands a control of the norm which is ~$L^1$ in time with
value in~$Lip$ in space for the vector field~$u$. This control
cannot be assumed and has to be proved. The main idea to overcome
this difficulty consists in differentiating the momentum  equation
of (INS2D) with respect to the time variable. As shown by
Lemma\refer {Echanget2x}, this represents the estimate of the second
space derivatives of $u.$

All the results of this subsection relies on the following lemma.

\begin{lem}
\label {estimenergyH2Linearise} {\sl Let~$(\rho,u,v)$ solve
~(LINS2D). Then we have, for any positive constant~$\cT$,
$$
\displaylines {\frac 12\frac d {dt} \|\sqrt \r v_t(t)\|^2_{L^2}
+\frac 34 \|\nabla v_t\|_{L^2}^2 \leq (f_t|v_t)_{L^2} +CF_{1,
\cT}(u(t)) \|\sqrt \r v_t(t)\|_{L^2} ^2 +CF_{2,\cT}(u(t),v(t))\cr {}
+C \|\nabla v_t(t)\|_{L^2} ^{\frac 12}\|\nabla  u_t(t)\|_{L^2}
^{\frac 12}  \|\nabla v(t) \|_{L^2}
 \|\sqrt \r v_t(t)\|_{L^2} ^{\frac 12}\|\sqrt \r u_t(t)\|_{L^2} ^{\frac 12}.
 }
$$
with
\beno
F_{1,\cT} (u) & \eqdefa  &   \|u\|_{L^4}^4 + \frac 1 {\cT} \|u\|_{L^2}^2 \|\nabla ^2u\|_{L^2}^2 \andf \\
F_{2,\cT}(u,v)& \eqdefa & \|\nabla u\|_{L^2}\|\nabla ^2u\|_{L^2}
\|\nabla v\|_{L^2}\|\nabla ^2v\|_{L^2} +C \|u\|_{L^2} \|\nabla
v\|_{L^2} ^2 \|\nabla^2 u\|^2_{L^2}+\cT \|\nabla^2v\|_{L^2}^2.
\eeno}
\end{lem}

Before applying and then proving  this lemma,  let us make some
comments about it. First of all, the parameter~$\cT$ is a scaling
parameter the role of which will appear in a while. Inequalities\refeq {zp2} and\refeq{LemDecayH1fondeq1} implies that, for any positive~$t $,
 \beq
\label {estimenergyH2Lineariseeq1}
\int_t^\infty F_{1,\cT} (u(t)) dt
\leq \| u(t) \|_{L^2}^4 + \frac {\|\nabla u(t)\|_{L^2}^2 } \cT\,
\cC(E_0) .
\eeq
In the same spirit, it can easily be infered that, for
any positive~$t $, \beno
\int_t^\infty F_{2,\cT} (u(t'),v(t'))  dt'  & \leq &  \|\nabla u\|_{L^\infty([t,\infty[;L^2)}\|\nabla^2 u\|_{L^2([t,\infty[\times\R^2)}\\
 &&\qquad\qquad{}\times\|\nabla v\|_{L^\infty([t,\infty[;L^2)}\|\nabla ^2v\|_{L^2([t,\infty[\times\R^2)} \\
&&\!\!\!\!\!\!\!\!\!\!\!\!\!\!\!\!\!\!\!\!\!\!\!\!\!\!\!\!\!\!\!\!\!\!\!\!\!\!\!\!\!\!\!\!\!\!\!\!\!
{}+C \|u\|_{L^\infty([t,\infty[;L^2)}  \|\nabla
v\|_{L^\infty([t,\infty[;L^2)} ^2 \|\nabla^2
u\|^2_{L^2([t,\infty[\times\R^2)}
 +\cT \|\nabla^2v\|_{L^2([t,\infty[\times\R^2)}^2.
\eeno
Inequality \eqref{LemDecayH1fondeq1} implies that  for any positive~$t $,
\beq \label {estimenergyH2Lineariseeq2}
\begin{split}
&\int_t^\infty F_{2,\cT} (u(t'),v(t'))  dt'
\lesssim\bigl(  \|\nabla v(t)\|_{L^2}^2+\int_t^\infty\|f(t')\|_{L^2}^2\,dt'\bigr)\\
&\qquad\qquad\qquad\qquad\qquad\qquad\qquad{}\times\bigl( \|\nabla
u(t)\|_{L^2}^2 + \|u(t) \|_{L^2} \|\nabla u(t)\|^2_{L^2} +\cT\bigr).
\end{split}
\eeq

Let us establish now the following corollary.
\begin{col}
\label {estimH2u} {\sl Let~$(\r,u,\na\Pi)$ be  a smooth enough
solution of~(INS2D), then we have for any positive~$t_0$ and any~$t$
greater than or equal to~$t_0$,
 $$
  \|\sqrt \r u _t(t)\|^2_{L^2} +\|\nabla^2 u (t)\|_{L^2} ^2 +\|\nabla \Pi(t)\|_{L^2}^2  +\int_{t_0}^t  \|\nabla u _t(t')\|^2_{L^2}  dt'
  \lesssim E_2(t_0).
 $$}
\end{col}
\begin{proof}
Let us apply Lemma\refer {estimenergyH2Linearise} with~Ê$u=v$ and~$f\equiv 0$.  This gives
$$
\displaylines {\frac 12 \frac d {dt} \|\sqrt \r u_t(t)\|^2_{L^2}
+\frac 34 \|\nabla u_t(t)\|_{L^2}^2 \leq  CF_{1, \cT}(u(t)) \|\sqrt
\r u_t(t)\|_{L^2} ^2 +CF_{2,\cT}(u(t),u(t))\cr {} +C \|\nabla
u_t(t)\|_{L^2} \|\nabla u(t) \|_{L^2}  \|\sqrt \r u_t(t)\|_{L^2}.
 }
$$
Using the convexity inequality, this gives \beq \label
{estimH2udemoeq1} \frac d {dt} \|\sqrt \r u_t(t)\|^2_{L^2} +\|\nabla
u_t\|_{L^2}^2 \lesssim \wt F_{1, \cT}(u(t)) \|\sqrt \r
u_t(t)\|_{L^2} ^2 +F_{2,\cT}(u(t),u(t)). \eeq with~$\wt F_{1, \cT}
(w) \eqdefa F_{1,\cT} (w) +\|\nabla w\|_{L^2}^2$.  Then Gronwall
lemma implies that
 for any positive~$t_0$ and any~$t$ greater than or equal to~$t_0$,
  $$
\longformule{
  \|\sqrt \r u _t(t)\|^2_{L^2}  +\int_{t_0}^t  \|\nabla u _t(t')\|^2_{L^2}  dt'
  \lesssim \Bigl(  \|\sqrt \r u _t(t_0)\|^2_{L^2} +\int_{t_0}^\infty F_{2,\cT} (u(t'),u(t'))  dt' \Bigr)
  }
  {
  {}\times \exp \Bigl(\int_{t_0}^\infty \wt F_{1,\cT} (u(t'))  dt '\Bigr).
  }
$$
Inequalities \eqref{zp1}, \refeq {estimenergyH2Lineariseeq1} and
\refeq{estimenergyH2Lineariseeq2} applied with~$u=v$ gives
 \beno
 && \|\sqrt \r u _t(t)\|^2_{L^2}  +\int_{t_0}^t  \|\nabla u _t(t')\|^2_{L^2}  dt'\\
 && \qquad \qquad  \qquad{}\lesssim \Bigl(  \|\sqrt \r u _t(t_0)\|^2_{L^2} +(1+\|u(t_0)\|_{L^2})  \|\nabla u(t_0)\|_{L^2}^2 (\|\nabla u(t_0)\|_{L^2}^2+\cT) \Bigr)
\\
&&  \qquad \qquad \qquad \qquad \quad \qquad{}\times \exp \Bigl(C
\Bigl(\|\sqrt \r u(t_0) \|_{L^2}^2+ \| u(t_0) \|_{L^2}^4 + \frac
{\|\nabla u(t_0)\|_{L^2}^2 } \cT \cC(E_0)\Bigr)\Bigr). \eeno
Choosing~$\cT=  \|\nabla u(t_0)\|_{L^2}^2$ ensures \beno \|\sqrt \r
u _t(t)\|^2_{L^2} +\int_{t_0}^t  \|\nabla u _t(t')\|^2_{L^2}
dt'\lesssim E_2(t_0), \eeno which together with the first inequality
of Lemma \ref{Echanget2x} implies \beno \|\na^2
u(t)\|_{L^2}+\|\na\Pi(t)\|_{L^2}\lesssim
\|u_t(t)\|_{L^2}+\|u(t)\|_{L^2}\|\na u(t)\|_{L^2}^2\lesssim
E_2^{\f12}(t_0). \eeno This finishes the proof of the corollary.
\end{proof}

\noindent{\it Proof of Lemma\refer {estimenergyH2Linearise} }
By applying~$\p_t$ to the  the momentum part of~(LINS2D), we obtain
 \beq
\label {estimenergyH2Linearisedemoeq0} \begin{split}
 \r v_t+\r   u \cdot\na v_t-\D v_t+\nabla \p_t\Pi_v & = \wt f_t  \with
\wt f_t  \eqdefa   -\r_t v_t- \r_t   u \cdot\na v -\rho u_t \cdot
\nabla v+f_t .
 \end{split}
 \eeq
 Now let us observe  that  as~$\rho$ is transported by the flow of~$u$, we have~$
 \r_t=-\dive (\rho u).$ Thus the new external force  becomes
$$
\wt f_t =\dive (\rho u) v_t+\dive (\rho u)  u \cdot\na v -\rho u_t
\cdot \nabla v+f_t.
$$
Applying the basic energy estimate\refeq{Energybasic}, we get
$$
\f12\f{d}{dt}\|\sqrt{\r }v_t\|_{L^2}^2+\|\na v_t\|_{L^2}^2= (\wt
f_t|v_t)_{L^2}.
$$
The key point  consists in  estimating the term~$ (\wt
f_t|v_t)_{L^2}.$ It follows by integration by parts~that
\ben
(\wt f_t|v_t)_{L^2}
& =&  (f_t |v_t)_{L^2} +\sum_{i=1}^5 \cE_i(t)\with \nonumber\\
\cE_1(t) &\eqdefa &
-2( \r   u \cdot\na v_t| v_t )_{L^2}  \nonumber \\
\cE_2(t) &\eqdefa &
-\bigl(\r  (u\cdot\nabla u)\cdot \nabla v\big | v_t\bigr)_{L^2}  \nonumber\\
 \label{1.4}
 \cE_3(t) &\eqdefa &
-\bigl ( \r  (u\otimes u):\na^2 v |\ v_t)_{L^2}
\\
\cE_4(t) &\eqdefa &
-( \r  u\cdot \nabla v|u\cdot\nabla v_t)_{L^2}\andf  \nonumber \\
\cE_5(t) &\eqdefa & -(\r u _t\cdot\nabla v |v_t)_{L^2}\,.  \nonumber
\een By using the 2-D interpolation inequality \refeq{1.14a},  we
get
 \beq
 \label  {estimenergyH2Linearisedemoeq1}
  \begin{split}
|\cE_1(t)|   \lesssim &
 \|u\|_{L^4 }\|\na v_t \|_{L^2}\|v_t\|_{L^4} \\
\lesssim  & \|\na v_t \|^{\frac 32 } _{L^2} \|u\|_{L^4 }
\|v_t\|^{\frac 12}_{L^2} \leq \epsilon  \|\na v_t \|^{2 } _{L^2} + C_\epsilon
\|u\|^4_{L^4 } \|v_t\|^{2}_{L^2}.
\end{split}
 \eeq
Using again  2-D interpolation inequality\refeq {1.14a} yields
 \beno
|\cE_2(t)|  & \leq &
 \|u\|_{L^4 }\|\nabla u\|_{L^4} \|\nabla v\|_{L^4} \|v_t\|_{L^4} \\
&\lesssim  & \|\nabla u\|^{\frac 12} _{L^2}  \|\nabla^2 u\|^{\frac12} _{L^2} \|\nabla v\|^{\frac 12} _{L^2}  \|\nabla^2 v \|^{\frac12} _{L^2}  \|u\|_{L^4} \|v_t\|_{L^2} ^{\frac 12} \|\nabla v_t\|_{L^2}^{\frac 12}.
 \eeno
 H\"older inequality implies that
 \beq
\label  {estimenergyH2Linearisedemoeq2} \cE_2(t) \leq \epsilon \|\na
v_t \|^{2 } _{L^2}  + C_\epsilon \|u\|^4_{L^4 } \|v_t\|^{2}_{L^2} +C_\epsilon \|\nabla
u\| _{L^2}  \|\nabla^2 u\| _{L^2} \|\nabla v\|_{L^2} \|\nabla^2 v \|
_{L^2} .
 \eeq
Using   the  2-D interpolation inequality\beq \label {1.14b}
\|a\|_{L^\infty(\R^2)}\lesssim \|a\|_{L^2(\R^2)}^{\f12}\|\nabla^2
a\|_{L^2(\R^2)}^{\f12}, \eeq we get
 \beq \label  {estimenergyH2Linearisedemoeq3} \begin{split}
|\cE_3(t)|   \leq &
 \|u\|_{L^\infty }^2\|\na^2 v \|_{L^2}\|v_t\|_{L^2} \\
\lesssim  & \|\nabla^2 v\|_{L^2}  \|u\|_{L^2}\|\nabla^2
u\|_{L^2}\|v_t\|_{L^2}\lesssim  \cT   \|\na^2 v \|^{2 } _{L^2} +
\frac 1 {\cT}\|u\|_{L^2}^2  \|\na^2 u \|^{2 } _{L^2}
\|v_t\|^{2}_{L^2}. \end{split}
 \eeq
Similarly using  again 2-D interpolation inequality\refeq{1.14b} ,
we get
 \beq \label  {estimenergyH2Linearisedemoeq4} \begin{split}
|\cE_4(t)|   \leq &
 \|u\|_{L^\infty }^2\|\na v \|_{L^2}\|\nabla v_t\|_{L^2}\\
 \leq  & \epsilon \|\na
v_t\|_{L^2}^2 +C_\epsilon \|u\|^2_{L^2} \|\nabla^2 u\|^2_{L^2} \|\nabla v
\|^2_{L^2} .
\end{split}
 \eeq
 In order to estimate~$\cE_5$,  we use\refeq {1.14a} which gives
 $$
 |\cE_5(t) |\lesssim  \|\nabla v_t(t)\|_{L^2} ^{\frac 12}\|\nabla  u_t(t)\|_{L^2} ^{\frac 12}  \|\nabla v(t) \|_{L^2}
 \|\sqrt \r v_t(t)\|_{L^2} ^{\frac 12}\|\sqrt \r u_t(t)\|_{L^2} ^{\frac 12}.
 $$
 Together with Inequalities\refeq{estimenergyH2Linearisedemoeq1}--(\ref{estimenergyH2Linearisedemoeq4}), this yields Lemma\refer{estimenergyH2Linearise}.
\qed

\medbreak Now let us  investigate the decay properties of the second
order space derivatives or of one time derivative of $u.$ They are
describe by the following proposition.
\begin{prop}
\label {decayD2}
{\sl For any~$T$ greater than or equal to~$T_2(\r_0,u_0)\eqdefa
\max\bigl\{T_1, E_2/E_1\bigr\}$ and under the assumptions of
Theorem\refer{ins2Ddecay}, Inequality\refeq  {ins2Ddecayeq3} holds. }
\end{prop}

\begin{proof}
We follow the same lines as the proof of Inequality\refeq {LemDecayH1fondeq3}
with the following computations. Using Relation\refeq
{estimH2udemoeq1}, we get that for any positive~${t_0} $ and~$t$
such that~$t$ is greater than or equal to ~${t_0} $,
$$
\longformule{
\frac d {dt} \bigl( (t-{t_0}
) \|\sqrt \r u_t(t)\|^2_{L^2}\bigr) + (t-{t_0}
)\|\nabla u_t\|_{L^2}^2
 \lesssim \|\sqrt \r u_t(t)\|^2_{L^2}
 }
 {{}
+\wt F_{1, \cT}(u(t)) (t-{t_0} ) \|\sqrt \r u_t(t)\|_{L^2} ^2 +
(t-{t_0} )F_{2,\cT}(u(t),u(t)). }
$$
Gronwall lemma along with \eqref{zp1} and
\eqref{estimenergyH2Lineariseeq1} implies that
 \beq
\begin{split} (t-t_0) \|\sqrt \r u_t(t)\|^2_{L^2}
 \lesssim&\biggl(  \int_{{t_0}
} ^t\|\sqrt \r u_t(t')\|^2_{L^2}dt' +\int_{{t_0}}^t  (t'-{t_0}
)F_{2,\cT}(u(t'),u(t'))dt'\biggr)
 \\
 &\qquad\qquad\quad \times \exp \biggl( \|\sqrt \r u({t_0}
) \|_{L^2}^2+ \|u({t_0} ) \|_{L^2}^4 + \frac {\|\nabla u({t_0}
)\|_{L^2}^2 } \cT \biggr) \cdotp \end{split} \label{zp3}
\eeq
It follows from \eqref{estimenergyH2Lineariseeq2} that
 \beno
 \int_{{t_0}}^t
(t'-{t_0} )F_{2,\cT}(u(t'),u(t'))dt'
& \leq &
t \int_{{t_0}}^t F_{2,\cT}(u(t'),u(t'))dt'\\
&\lesssim &
t \bigl( \|\nabla u(t_0)\|^2_{L^2}+\cT\bigr)   \|\nabla u({t_0})\|^2_{L^2} .
\eeno
 Resuming the above estimates into
\eqref{zp3} and choosing~$\cT=  \ds   \|\nabla u(t_0)\|^2_{L^2}$
yields
\beno
(t-t_0)\|\sqrt{\r}u_t(t)\|_{L^2}^2\lesssim \|\na
u(t_0)\|_{L^2}^2+t\|\na u(t_0)\|_{L^2}^4.
 \eeno
 Taking $t_0$ equals to~$\f{t}2$
in the above inequality, then Inequality\refeq {ins2Ddecayeq2} of
Proposition\refer {propDecayL2H1} ensures~that
\beno
  \|\sqrt \r u_t(t)\|^2_{L^2}   & \lesssim & E_1 \frac T t  \frac 1T  \langle t\rangle_{{}_T}^{-3}
+ E_1^2 \langle t\rangle_{{}_T}^{-6} .
  \eeno
Using Corollary\refer {estimH2u}, we infer that for $t\geq T_2$
\beq\label{zp4}
  \|\sqrt \r u_t(t)\|^2_{L^2} \lesssim \bigl(  E_2 +
  E_1^2
  +E_1/T\bigr)  \langle
  t\rangle_{{}_T}^{-4} \lesssim E_2 \langle
  t\rangle_{{}_T}^{-4}.
\eeq While it follows from Lemma \ref{Echanget2x} that \beno
\|\na^2u(t)\|_{L^2}+\|\na\Pi(t)\|_{L^2}\lesssim
\|u_t(t)\|_{L^2}+\|u(t)\|_{L^2}\|\na u(t)\|_{L^2}^2, \eeno
 which together with Inequalities \eqref{ins2Ddecayeq10}, \refeq {ins2Ddecayeq2} and \eqref{zp4} leads to
 Inequality\refeq {ins2Ddecayeq3}. The proposition is proved.
\end{proof}

 \begin{col}
 \label {S7col1}
 {\sl
 Under the assumptions of Theorem\refer  {ins2Ddecay}, we have
 \ben
 \label  {S7col1eq0}
 \|u(t)\|_{L^\infty}  & \lesssim  & E_0^{\frac 14} E_2^{\frac 14}  \langle t\rangle_{{}_T}^{-\frac 32} \andf\\
\label  {S7col1eq1} \int_{t}^\infty \|\nabla u(t')\|_{L^\infty}  dt'
& \lesssim & \bigl(\sqrt{E_2}T\bigr)^{\f34} \langle
t\rangle_{{}_T}^{-1}+\bigl(\sqrt{E_2}T\bigr)^{\f14} \w{t}_{{}_T}^{-2}.
 \een
 Moreover, we have the following estimates on the density. For any~$ p$ in~$ [2,\infty]$, we have
\ben
\label {1.19b}
\|\nabla \rho(t)\|_{L^p}  & \lesssim &  \|\nabla \rho_0\|_{L^p},\\
\label {S7col1eq3}
\|\rho_t (t)\|_{L^p} & \lesssim &  \|\nabla \rho_0\|_{L^p}E_0^{\frac 14}E_2^{\frac 14} \langle  t\rangle_{{}_T}^{-\frac  3  2},\\
\label {S7col1eq4}
 \|\nabla^2 \rho(t)\|_{L^2}  & \lesssim &  \|\nabla^2  \rho_0\|_{L^2} +\|\nabla \rho_0\|_{L^\infty} E_2^{\frac 12} T  \andf\\
 \label {S7col1eq5}
  \|\nabla\rho_t(t)\|_{L^2}  & \lesssim & \bigl( E_2^{\f14}\|\nabla^2  \rho_0\|_{L^2} +\|\nabla \rho_0\|_{L^\infty}(
   E_2^{\frac 34} T+E_1^{\frac12})\bigr) \langle  t\rangle_{{}_T}^{-\frac  3  2}.
\een
}
\end{col}
\begin{proof}
Inequality\refeq{S7col1eq0} follows directly from
Inequalities \refeq{ins2Ddecayeq10} and \eqref{ins2Ddecayeq3}  and from the
 interpolation inequality\refeq{1.14b}. By using 2D interpolation inequality, we get,
by applying Lemma\refer{Echanget2x} with~$p$ equal to~$4$,  that
\beno
\|\nabla u(t) \|_{L^\infty}
& \leq &
C \|\nabla u (t)\|^{\frac 12}_{L^4}\|\nabla ^2 u(t)\|^{\frac 12}_{L^4}\\
& \leq &
C \|\nabla u (t)\|^{\frac 12}_{L^4} \bigl ( \|u_t\|_{L^4} ^{\frac 12} + C\|u\|_{L^8}\|\nabla u(t)\|_{L^2}^{\frac 12}\bigr).
\eeno
Using that~$\|u(t)\|_{L^8} \leq C \|u(t)\|_{L^2}^{\frac 14}\|\nabla u(t)\|_{L^2} ^{\frac 34}$,   we infer that
$$
\|\nabla u(t) \|_{L^\infty}  \leq C \|\nabla u(t)\|_{L^2}^{\frac 14}
\|\nabla^2 u (t)\|_{L^2}^{\frac 14} \|\sqrt \r u_t\|_{L^2}^{\frac
14} \|\nabla u_t\|_{L^2}^{\frac 14} +C \|\nabla u (t)\|_{L^2}^{\frac
32}  \|u(t)\|_{L^2}^{\frac 14}\|\nabla ^2 u(t)\|_{L^2}^{\frac 14}.
$$
Applying H\"older inequality with respectively
$\left(\f18,\f34,\f18\right)$ and $\left(\f34, \f14\right)$ gives
$$
\longformule{\int_t^\infty\|\na u(t')\|_{L^\infty}dt'\lesssim
\Bigl(\int_t^\infty\|\na
u(t')\|_{L^2}^2dt'\Bigr)^{\f34}\Bigl(\int_t^\infty\|u(t')\|_{L^2}\|\na^2
u(t')\|_{L^2}dt'\Bigr)^{\f14} }{{} +\Bigl(\int_t^\infty\|\na
u(t')\|_{L^2}^2dt'\Bigr)^{\f18}\Bigl(\int_t^\infty
\|\na^2u(t')\|_{L^2}^{\f13}\|\sqrt{\r}u_t(t')\|_{L^2}^{\f13}dt'\Bigr)^{\f34}\Bigl(\int_t^\infty\|\na
u_t(t')\|_{L^2}^2dt'\Bigr)^{\f18},} $$ which together with
\eqref{zp1}, \refeq{ins2Ddecayeq10}-\eqref{ins2Ddecayeq3} and
Corollary \ref{estimH2u} ensures \eqref {S7col1eq1}.

Inequality\refeq {1.19b} comes simply from the density equation
after differentiation which~is \beq\label{zp5}
\partial_t \nabla \rho +u\cdot \nabla \nabla \rho = -\nabla u\cdot \nabla \rho.
\eeq Gronwall lemma  and \eqref{S7col1eq1} allows to conclude the
$L^p$ estimate for~$\nabla \rho$. For the inequality on~$\rho_t$,  let us observe that, thanks to Inequalities\refeq{S7col1eq0} and \eqref{1.19b}, the transport equation implies that,
$$
\|\rho_t(t)  \|_{L^p}\leq \|u(t)\|_{L^\infty} \|\nabla
\rho(t)\|_{L^p} \lesssim  \|\nabla \rho_0\|_{L^p} E_0^{\frac 14}
E_2^{\frac 14}  \langle t\rangle_{{}_T}^{-\frac  3  2}.
$$
which is exactly the required inequality. In order to prove
Inequality\refeq  {S7col1eq4}, let us differentiate twice the
transport equation which gives \beq\label{zp5tr}
\partial_t \partial_j\partial_k \rho +u\cdot\nabla \partial_j\partial_k \rho = -\partial_k u\cdot\nabla \partial_j\rho
 -\partial_j u\cdot\nabla \partial_k\rho- \partial_j\partial_k u\cdot\nabla\rho.
\eeq Let us observe that \beno
\|\partial_k u(t)\cdot\nabla \partial_j\rho(t)\|_{L^2} & \leq &  \|\nabla u(t)\|_{L^\infty} \|\nabla^2\rho(t)\|_{L^2}\andf\\
\|\partial_j\partial_k u(t)\cdot\nabla\rho(t)\|_{L^2} & \leq &
\|\nabla^2 u(t)\|_{L^2} \|\nabla \rho(t)\|_{L^\infty}, \eeno so that
we obtain \beq \label {S7col1eq9} \f{d}{dt}\|\na^2\r(t)\|_{L^2}\leq
2\|\na
u(t)\|_{L^\infty}\|\na^2\r(t)\|_{L^2}+\|\na^2u(t)\|_{L^2}\|\na\r(t)\|_{L^\infty}.
\eeq Gronwall lemma along with Proposition \ref{decayD2} gives
Inequality\refeq {S7col1eq4}.
Finally it follows from Inequality~\eqref{zp5} that
\beno
\|\na\r_t(t)\|_{L^2}
\leq
\|u(t)\|_{L^\infty}\|\na^2\r(t)\|_{L^2}+\|\na u(t)\|_{L^2}\|\na\r(t)\|_{L^\infty}.
\eeno
Then Inequality~\eqref{S7col1eq5}
follows from Inequalities~\eqref{S7col1eq0}, \eqref{S7col1eq4} and
\eqref{ins2Ddecayeq2}.
\end{proof}

Let us remark that before Inequality\refeq {1.19b}, we never use any
regularity property for the density~$\rho$. From now on, we shall do
it in order to estimate the third derivatives of the velocity field.

\begin{prop}
\label {decayD3} {\sl For any~$T$ greater than or equal
to~$T_3(\r_0,u_0)\eqdefa \max\{T_2, E_2/E_3\},$ we have under the
assumptions of Theorem\refer{ins2Ddecay}, \beq\label{propdecayD3eq1}
\|\nabla  u_t (t)\|^2 _{L^2} +\int_t^\infty \Bigl(\|u_{tt}  (t')\|^2
_{L^2} + \|\nabla^2 u_t(t')\|^2_{L^2}+\|\nabla\partial_t \Pi
(t')\|_{L^2}^2\Bigr) dt' \lesssim E_3 \langle t \rangle_{{}_T}^{-5},
\eeq and for any nonnegative $t_0$ and any $t\geq t_0,$
\beq
\label{propdecayD3eq2}
\int_{t_0}^t\langle t'
\rangle_{{}_T}^{5_-} \Bigl(\|u_{tt}  (t')\|^2 _{L^2} + \|\nabla^2
u_t(t')\|^2_{L^2}+\|\nabla\partial_t \Pi (t')\|_{L^2}^2\Bigr) dt'
\lesssim (1+T) \bigl(1+\|\na
 \r_0\|_{L^\infty}^2\bigr)E_3. \eeq
 Here and in all that follows, $\frak{a}_-$
denotes any number strictly less than $\frak{a}.$}
\end{prop}

\begin{proof}
Relation\refeq{estimenergyH2Linearisedemoeq0} applies with~$v=u$
and~$f\equiv 0$ claims exactly that~$u_t$ is a solution of~(LINS2D)
with the external force
$$
\wt f\eqdefa -\rho_tu_t  - \rho_t u\cdot \nabla u -\rho u_t\cdot\nabla u.
$$
In order to apply Lemma\refer  {DecayH1fond}, we have to
estimate~$\|\wt f(t)\|_{L^2}$. H\"older inequality  and
interpolation inequality allow to write \beno
 \|\wt f \|_{L^2}^2
 & \leq & 2\|\r_t \|^2_{L^\infty} \|u_t \|_{L^2}^2+2\|\r_t
\|_{L^\infty}^2 \|u \|^2_{L^\infty}\|\na u\|_{L^2}^2
+4\|u_t \|_{L^4}^2\|\nabla u \|_{L^4}^2\\
& \lesssim  &
 \|\r_t \|^2_{L^\infty} \|u_t \|_{L^2}^2+\|\r_t
\|_{L^\infty}^2 \|u \|^2_{L^\infty}\|\na u\|_{L^2}^2 + \|\nabla u_t
\|_{L^2}\|u_t \|_{L^2}\|\nabla^2 u\|_{L^2} \|\nabla u \|_{L^2}.
\eeno
 Using Corollary\refer {S7col1} and Proposition\refer {decayD2}, we get
 \beq\label{propdecayD3eq3}
\begin{split}
\|\wt f(t)\|^2_{L^2}  \lesssim \|\nabla& \r_0\|_{L^\infty} ^2
E_2^{\frac 12}   \langle t \rangle_{{}_T}^{-3} \|u_t(t)\|_{L^2}^2
\\
&+ \|\nabla \r_0\|_{L^\infty} ^2 E_2 \langle t \rangle_{{}_T}^{-6}
\|\nabla u(t)\|_{L^2}^2 + E_2 \langle t \rangle_{{}_T}^{-4} \|\nabla
u_t (t)\|_{L^2}\|\nabla u(t) \|_{L^2}. \end{split} \eeq
 Cauchy-Schwarz inequality gives
\beno
\int_{t_0}^t \|\wt f(t')\|^2_{L^2} dt'  & \lesssim & \|\nabla
\r_0\|_{L^\infty} ^2 E_2^{\frac 12}   \langle t_0
\rangle_{{}_T}^{-3}\int_{t_0}^t \|u_t(t')\|_{L^2}^2dt'
\\
&&{}+ \|\nabla \r_0\|_{L^\infty} ^2
E_2   \langle t_0 \rangle_{{}_T}^{-6} \int_{t_0}^t  \|\nabla u(t')\|_{L^2}^2dt'\\
&&{} + E_2   \langle t_0 \rangle_{{}_T}^{-4} \Bigl( \int_{t_0}^t
\|\nabla u_t(t') \|^2_{L^2} dt' \Bigr)^{\frac 12} \Bigl(
\int_{t_0}^t \|\nabla u(t') \|_{L^2}^2 dt' \Bigr)^{\frac 12}.
\eeno
 Applying \eqref{LemDecayH1fondeq1} and Corollary\refer
{estimH2u} leads to
\begin{equation}\label{zp6}
\int_{t_0}^t \|\wt f(t')\|^2_{L^2} dt'  \lesssim \bigl(\|\nabla
\r_0\|_{L^\infty} ^2 E_2^{\frac 12} (E_1 +E_2^{\frac 12})+
E_2^{\frac 32}\bigr)   \langle t_0 \rangle_{{}_T}^{-6} \lesssim
E_3\w{t_0}^{-6}_T.
\end{equation}
Applying again \eqref{LemDecayH1fondeq1} gives \beq \label{zp7}
\begin{split}
& \|\nabla  u_t
(t)\|^2 _{L^2} +\int_{t_0}^\infty \bigl(\|u_{tt}  (t')\|^2 _{L^2} +
\|\nabla^2 u_t(t')\|^2_{L^2}+\|\nabla\partial_t \Pi
(t')\|_{L^2}^2\bigr) dt'\\
&\qquad\qquad\qquad\qquad\qquad
\qquad\qquad\qquad\qquad\quad\lesssim \|\na
u_t(t_0)\|_{L^2}^2+\int_{t_0}^\infty\|\wt{f}(t')\|_{L^2}^2\,dt'.
\end{split}
 \eeq
 In particular, \eqref{zp6} and \eqref{zp7} ensure
that \beq \label{zp8} \|\na u_t(t)\|_{L^2}^2\lesssim E_3. \eeq While
it follows from \eqref{LemDecayH1fondeq3}  and \eqref{zp6} that
\beno t\|\na u_t(t)\|_{L^2}^2&\lesssim &
\|u_t(t/2)\|_{L^2}^2+t\int_{t/2}^\infty
\|\wt{f}(t')\|_{L^2}^2\,dt'\\
&\lesssim& E_2 \w{t}_{{}_T}^{-4}+E_3 t \w{t}_{{}_T}^{-6},
 \eeno
 which together with \eqref{zp8} implies that if $T$ greater than or equal to~$E_2/E_3$
\beno \|\na u_t(t)\|_{L^2}^2\lesssim \max\bigl(E_3,{E_2}/T
\bigr) \w{t}_{{}_T}^{-5}+E_3 \w{t}_{{}_T}^{-6}\lesssim E_3 \w{t}_{{}_T}^{-5}. \eeno
Resuming the above estimate and \eqref{zp6} into \eqref{zp7} gives
\beno \int_{t}^\infty \bigl(\|u_{tt}  (t')\|^2 _{L^2} + \|\nabla^2
u_t(t')\|^2_{L^2}+\|\nabla\partial_t \Pi (t')\|_{L^2}^2\bigr)
dt'\lesssim E_3 \w{t}_{{}_T}^{-5}.
 \eeno
 This proves Inequality \eqref{propdecayD3eq1}.  \eqref{propdecayD3eq1} together with
 \eqref{propdecayD3eq3} implies that
 \beno
 \|\wt{f}(t)\|_{L^2}^2\lesssim
 \bigl(1+\|\na
 \r_0\|_{L^\infty}^2\bigr)E_3 \w{t}_{{}_T}^{-7}. \eeno

In order to prove \refeq{propdecayD3eq2}, we  apply
\eqref{LemDecayH1fondeq2} for  any $s$ less than~$5$ and $v$ equal
to~$u_t$ to get
$$
\longformule{ \langle t \rangle_{{}_T}^{5_-}\|\nabla u_t (t)\|^2
_{L^2} +\int_{t_0}^t\langle t' \rangle_{{}_T}^{5_-} \Bigl(\|u_{tt}
(t')\|^2 _{L^2} + \|\nabla^2 u_t(t')\|^2_{L^2}+\|\nabla\partial_t
\Pi (t')\|_{L^2}^2\Bigr) dt' } { {} \lesssim \langle t_0
\rangle_{{}_T}^{5_-}\|\nabla u_t (t_0)\|^2
_{L^2}+\int_{t_0}^t\w{t'}_T^{5_-}\|\wt{f}(t')\|_{L^2}^2dt'+\int_{t_0}^t\w{t'}^{4_-}_T\|\na
u_t(t')\|_{L^2}^2\f{dt'}T, }
$$
 which together with the fact that
\beno \int_{t_0}^t\w{t'}^{5_-}\|\wt{f}(t')\|_{L^2}^2dt' & \lesssim &
 \bigl(1+\|\na
 \r_0\|_{L^\infty}^2\bigr) E_3 T\andf \\
\int_{t_0}^t\w{t'}_T^{4_-}\|\na u_t(t')\|_{L^2}^2\f{dt'}T &\lesssim
& E_3\int_{t_0}^t\w{t'}_T^{-1_+}\f{dt'}T\lesssim E_3 \eeno
 leads to Estimate\refeq{propdecayD3eq2}. This completes the proof of the proposition.
\end{proof}

Now let us translate the control of~$\|\nabla u_t(t)\|_{L^2} $ in
term of control of~$\|\nabla ^3 u(t)\|_{L^2}$.
\begin{prop}
 \label{DecayD4} {\sl
Under the hypothesis of Theorem\refer  {ins2Ddecay}, for any $T\geq
T_3(\r_0,u_0),$ we have \ben &&\label{PropS7eq6eq1} \|\nabla^3
u(t)\|^2_{L^2}+\|\nabla^2 \Pi (t)\|^2_{L^2} \lesssim
\bigl(1+\|\na\r_0\|_{L^2}\bigr) E_3 \langle t\rangle _{{}_T}^{-
5}\log^2 \langle t\rangle_{{}_T} \andf \\
&&\label{PropS7eq6eq2} \qquad\qquad\|\nabla u(t)\|_{L^\infty}
\lesssim \bigl(1+\|\na\r_0\|_{L^2}\bigr)^{\f14} E_1^{\f14}E_3^{\f14}
\langle t\rangle _{{}_T}^{- 2}\log^{\f12} \langle t\rangle_{{}_T}.
\een}
\end{prop}
\begin{proof}
By differentiation of the momentum equation of (INS2D) with respect
to the space variables, we get, by using Leibnitz formula, that
$$
\Delta \partial_j u-\partial_j\nabla \Pi = \rho\, \partial_j u_t  +
\partial_j \r\, u_t +\partial_j \r \,u\cdot\nabla u + \r\, \partial_j
u\cdot\nabla u + \r u\cdot\nabla \partial_j u.
$$
Applying  Lemma\refer {Echanget2x} with~$v=\partial_j u$ and~$f= \partial_j \r\, u_t  + \partial_j \r\, u\cdot\nabla u+ \r\, \partial_j u\cdot\nabla u$ gives
$$
\|\Delta \partial_j u\|_{L^2} +\|\partial_j\nabla \Pi\|_{L^2} \leq
C\bigl(\|\nabla u_t\|_{L^2}+\|f\|_{L^2}+ \|u\|_{L^{4}}^{2} \|\nabla
^2 u\|_{L^2}\bigr).
$$
In view of Propositions\refer{propDecayL2H1}
 and\refer  {decayD3}, we infer that
\ben\nonumber
 \|\Delta \partial_j u(t)\|_{L^2} +\|\partial_j\nabla \Pi(t)\|_{L^2}
 &\lesssim&
 E_3^{\frac 12} \langle t\rangle_{{}_T}^{-\frac  5  2} +E_1^{\f12}E_2^{\f12} \w{t}_{{}_T}^{-\frac  9  2}+\|f(t)\|_{L^2}\\
 &\lesssim&
 E_3^{\frac 12} \langle t\rangle_{{}_T}^{-\frac  5  2} +\|f(t)\|_{L^2} \label{S7eq6demoeq1}.
\een Let us estimate~$\|f(t)\|_{L^2}$. We write that \beno
\| \partial_j \r\, u\cdot\nabla u\|_{L^2} & \leq & \|\nabla \r\|_{L^\infty} \|u\|_{L^\infty} \|\nabla u\|_{L^2}\andf\\
\| \r\, \partial_j u\cdot\nabla u\|_{L^2} & \leq & \|\r\|_{L^\infty}
\|\nabla u\|_{L^2}\|\nabla^2 u\|_{L^2}. \eeno Applying
Propositions\refer{propDecayL2H1}
 and\refer  {decayD3}, we infer that
 \beq \begin{split}
 \| \partial_j \r\, u\cdot\nabla u\|_{L^2} +\| \r\, \partial_j u\cdot\nabla u\|_{L^2}
 &\lesssim \bigl(\|\na\r_0\|_{L^\infty}E_1^{\f12}E_2^{\f14}+ E_1^{\f12}E_2^{\f12}\bigr) \langle t\rangle_{{}_T}^{-3}\\
 &\lesssim E_3^{\f12}\langle t\rangle_{{}_T}^{-3}.
 \label{S7eq6demoeq3} \end{split}
 \eeq
 The linear term~$\partial_j\r\, u_t$ is more delicate to estimate.
   Let us write that, using H\"older inequality and interpolation inequality, for any~$\epsilon$ in~$]0,1[$,
 \beno
 \|\partial_j\r u_t\|_{L^2} &  \leq &  \|\nabla \r\|_{L^{\frac 2 {1-\epsilon}}} \|u_t\|_{L^{\frac 2 \epsilon}}\\
 &\leq & \frac C \epsilon  \|\nabla \r\|_{L^{\frac 2 {1-\epsilon}}} \|u_t\|_{L^{2}}^\epsilon
 \|\nabla u_t\|_{L^2} ^{1-\epsilon}.
\eeno Using Propositions\refer {decayD2} and\refer {decayD3}, and
Estimates\refeq {1.19b}, we get that, for any~$\epsilon$ in~$]0,1[$,
$$
 \|\partial_j\r \,u_t\|_{L^2}   \leq (\|\nabla \rho_0\|^2_{L^\infty} E_2)^{\frac \epsilon 2 }  (\|\nabla \rho_0\|_{L^2}^2 E_3)^{\frac {1-\epsilon} 2}
\langle t \rangle_{{}_T}^{-\frac 52}\times \frac 1 \epsilon \langle
t \rangle_{{}_T}^{\frac \epsilon 2}.
$$
By convexity inequality we get that for any~$\epsilon$ in~$Ê]0,1[$,
$$
 (\|\nabla \rho_0\|^2_{L^\infty} E_2)^{\frac \epsilon 2 }  (\|\nabla \rho_0\|_{L^2}^2 E_3)^{\frac {1-\epsilon} 2}
 \leq
  \|\nabla \rho_0\|_{L^\infty} E_2^{\frac 12}+ \|\nabla \rho_0\|_{L^2} E_3^{\frac 12}
  \leq (1+\|\nabla \rho_0\|_{L^2})E_3^{\frac 12} .
$$
Thus, for any $\epsilon$ in~$]0,1[$,
\beq
\label{S7eq6demoeq4}
 \|\partial_j\r u_t\|_{L^2}   \lesssim (1+\|\nabla \rho_0\|_{L^2})E_3^{\frac 12}
\langle t \rangle_{{}_T}^{-\frac 52}\times \frac 1 \epsilon \langle
t \rangle_{{}_T}^{\frac \epsilon 2}.
\eeq
Choosing~$\epsilon$ equal to~$\log^{-1} \langle t \rangle_{{}_T}$ in \eqref{S7eq6demoeq4}, and
then substituting the resulting inequality and Inequality\refeq{S7eq6demoeq3}
into \eqref{S7eq6demoeq1} leads to \eqref{PropS7eq6eq1}.

Finally \eqref{PropS7eq6eq2} follows from   interpolation inequality
\eqref{1.14b}, and \eqref{ins2Ddecayeq2}, \eqref{PropS7eq6eq1}. This
finishes the proof of Proposition \ref{DecayD4}.
\end{proof}

By summarizing Propositions \ref{propDecayL2H1}, \ref{decayD2},
\ref{decayD3} and \ref{DecayD4}, we conclude the proof of Theorem
\ref{ins2Ddecay}.

\subsection {Decay of solutions to \eqref{INS2dParameterb}} Applying
 Theorem \ref{ins2Ddecay} to the System\refeq{INS2dParameterb} leads
 to the following theorem:

\begin{thm}
\label {S3thm2}
{\sl Let $(\r^\h, v^\h, \nh \Pi^\h)$ be  the smooth
 solution of \eqref{INS2dParameterb}. Then  under the
assumptions of  Theorem \ref{insslowvar}, we have
\beq
\label{S3thm2eq1}
\begin{split} &
\w{t}\|
v^\h(t,\cdot,z)\|_{L^2_\h}+\w{t}^{\f32}\bigl(\|\nh
v^\h(t,\cdot,z)\|_{L^2_\h}+\|v^\h(t,\cdot,z)\|_{L^\infty_\h}\bigr)\\
&+ \w{t}^{2}\bigl(\|v^\h_t(t,\cdot,z)\|_{L^2_\h}+ \|\nh^2
v^\h(t,\cdot,z)\|_{L^2_\h}+\|\nh
\Pi^\h(t,\cdot,z)\|_{L^2_\h}\bigr)\\
&+\w{t}^{\f52}\log^{-1}\w{t}\bigl(\|\nh^3v^\h(t,\cdot,z)\|_{L^2_\h}+\|\nh^2
\Pi^\h(t,\cdot,z)\|_{L^2_\h}\bigr)\\
&+ \w{t}^{\f52} \|\nh
v^\h_t(t,\cdot,z)\|_{L^2_\h}+\w{t}^{2}\log^{-\f12}\w{t}\|\nh
v^\h(t,\cdot,z)\|_{L^\infty_\h} \biggr)\leq \cC_0h(z),
\end{split}
\eeq and \beq\label{S3thm2eq8} \int_{t_0}^t\langle t'
\rangle_{{}_T}^{5_-} \Bigl(\|\p_{t}^2v^\h  (t')\|^2 _{L^2_\h} +
\|\nabla_\h^2 v^\h_t(t')\|^2_{L^2_\h}+\|\nabla_\h\partial_t \Pi^\h
(t')\|_{L^2_\h}^2\Bigr) dt' \lesssim \cC_0 h^2(z). \eeq We also have
 \ben
 \label{S3thm2eq2} &&\qquad\quad\|\nh^4
v^\h(\cdot,\cdot,z)\|_{L^1_t(L^2_\h)}\leq \cC_0h(z)\andf\\
\label{S3thm2eq3}
&&\|\varrho^h(t,\cdot,z)\|_{H^4_\h}+\w{t}^{\f32}\|\r_t(t,\cdot,z)\|_{H^3_\h}\leq
\cC_0\eta  h(z). \een Here and in what follows, we always denote
$\varrho^\h\eqdefa \r^h-1$ and $h(z)$ to be a generic positive
function which belongs to $L^2_{\rm v}\cap L^\infty_{\rm v}.$
 }
 \end{thm}

 \begin{proof} \eqref{S3thm2eq1} and \eqref{S3thm2eq8} follows directly from Theorem
\ref{ins2Ddecay} and \eqref{propdecayD3eq2}. In order to prove
\eqref{S3thm2eq2},   we get by applying \eqref{S3thm2eq1} and
Theorem 3.14 of \cite{BCD} that \beq \label{S3thm2eq6}
\|\varrho(t)\|_{L^\infty_t(H^3_\h)}\lesssim \|\varrho_0
\|_{H^3_\h}\exp\Bigl(C\int_0^t\|\nh v^\h(t')\|_{H^2_\h}dt'\Bigr)\leq
\cC_0\eta  h(z). \eeq Whereas we deduce from the momentum equation
of \eqref{INS2dParameterb} and the classical estimates on Stokes
operator that \beno \|\nh^4
v^\h(t)\|_{L^2_\h}+\|\nh^3\Pi^\h(t)\|_{L^2_\h}\lesssim
\|\nh^2(\r^\h v_t^\h)(t)\|_{L^2_\h}+\|\nh^2(\r^\h v^\h\cdot\nh
v^\h)(t)\|_{L^2_\h}, \eeno which together with \eqref{S3thm2eq1} and
\eqref{S3thm2eq6} ensures that \beq \label{S3thm2eq5} \|\nh^4
v^\h(t)\|_{L^2_\h}+\|\nh^3\Pi^\h(t)\|_{L^2_\h}\leq
\cC_0\bigl(h(z)\w{t}^{-\frac 5 2}+\|\nh^2\r^\h v_t^\h(t)\|_{L^2_\h}+\|\nh^2v_t^\h(t)\|_{L^2_\h}\bigr).
\eeq Note that  for any~$\epsilon$ in~$]0,1[$,  we have \beno
\|\nh^2\r^\h v_t^\h\|_{L^2_\h}&\lesssim&\|\nh^2\r^\h\|_{L^{\f2{1-\epsilon}}_\h}\|v_t^\h\|_{L^{\f2\epsilon}_\h}\\
&\lesssim&
\f1{\epsilon}\|\nh^3\r^\h\|_{L^2_\h}^{\epsilon}\|\nh^2\r^\h\|_{L^2_\h}^{1-\epsilon}\|\p_tv^\h\|_{L^2_\h}^{\epsilon}\|\nh\p_tv^\h\|_{L^2_\h}^{1-\epsilon}.
\eeno Applying \eqref{S3thm2eq1}, \eqref{S3thm2eq6} and using a
similar derivation of \eqref{S7eq6demoeq4} gives
\beno
\|\nh^2\r^\h v_t^\h(t)\|_{L^2_\h}\leq
\cC_0h(z)\w{t}^{-\frac 5 2}\log\w{t}.
\eeno
Resuming the above estimate into \eqref{S3thm2eq5} and using \eqref{S3thm2eq8},
we infer \eqref{S3thm2eq2}.

With  \eqref{S3thm2eq1} and  \eqref{S3thm2eq2}, we deduce from
Theorem 3.14 of \cite{BCD} that \beq \label{S3thm2eq4}
\|\varrho^h\|_{L^\infty_t(H^4_\h)}\leq \|\varrho_0
\|_{H^4_\h}\exp\Bigl(C\int_0^t\|\nh v^\h(t')\|_{H^3_\h}dt'\Bigr)\leq
\cC_0\eta  h(z). \eeq Then by taking one more horizontal derivatives
to \eqref{zp5tr} and using \eqref{S3thm2eq1} and \eqref{S3thm2eq4},
we get
\beno
\|\nh^3\p_t\r^\h(t)\|_{L^2_\h}
 &\leq&
\|v^\h(t)\|_{L^\infty_\h}\|\nh^4\r^\h(t)\|_{L^2_\h}+3\|\nh
v^\h(t)\|_{L^\infty_\h}\|\nh^3\r^\h(t)\|_{L^2_\h}\\
&&{}
+3\|\nh^2v^\h(t)\|_{L^2_\h}\|\nh^2\r^\h(t)\|_{L^\infty_\h}+\|\nh^3v^\h(t)\|_{L^2_\h}\|\nh\r^\h(t)\|_{L^\infty_\h}\\
&\leq&
\cC_0\eta  h(z)\w{t}^{-\frac 3 2}.
\eeno
This together with \eqref{S7col1eq3} and \eqref{S3thm2eq4} ensures \eqref{S3thm2eq3}. This finishes the
 proof of Theorem\refer {S3thm2}.
\end{proof}

\setcounter{equation}{0}

\section{Decay estimates of $\p_z v^\h$}
\label {decayL2partialz}

In this section, we consider the family of solutions of
\eqref{INS2dParameterb},  which depends smoothly on a real parameter
denoted by~$z$. We skip the proof of the fact that, if the initial
data depends  smoothly on the parameter~$z$, so does the solution.
We concentrate  on decay estimates for the family~Ê$\partial_z v^\h$
which we denote by~$v^\h_z$ and also~$\partial_z\rho^\h $
by~$\rho^\h_z$. The purpose of this section is to prove that, under
suitable hypothesis  on the derivative of the family of initial
data, ~$v^\h_z$ has the same decay property as~$v^\h$.

\begin{thm}
\label {SAprop1}
{\sl Under the assumptions of Theorem
\ref{insslowvar},
 One has \beq
 \label{thmSAprop1eq1}
 \|\r^\h_z(t,\cdot,z)\|_{H^3_\h}+\w{t}^{\f32}\|\p_t\r^\h_z(t,\cdot,z)\|_{H^2_\h}\leq
 \cC_0\eta h(z).
 \eeq Moreover, $v^\h_z$ shares the same decay estimates of Inequalities \eqref{S3thm2eq1} and \eqref{S3thm2eq8} of Theorem\refer
 {S3thm2}.
 }
\end{thm}

In view of  Inequality\eqref{INS2dParameterb},  the quantity~$(\rho_z^\h,
v_z^\h, \Pi_z^\h)\eqdefa (\partial_z \rho^\h,
\partial_z v^\h,\partial_z \Pi^\h)$ satisfies the system in~$\R^+\times\R^2$
$$
\displaylines{ {\rm (D1INS2D)}\quad \left\{\begin{array}{c}
\displaystyle
 \p_t\r_z^\h+v^\h\cdot\nh \r_z^\h=-v^\h_z\cdot \nh \rho^\h,  \\
\displaystyle \r^\h\pa_t v^\h_z+ \r^\h v^\h\cdot\nh v^\h_z  -\D v^\h_z + \nh \Pi_z^\h=f_1 +L(t) v^\h_z,\\
\displaystyle \dive_\h v^\h_z=0, \\
\displaystyle  (\r_z, v^\h_z)|_{t=0}=(\eta\partial_z \varsigma_0,
\partial_z v_0^\h).
\end{array}\right.
\with\cr f_1= -\rho_z^\h v^\h_t -\rho_z^\h  v^\h\cdot\nh v^\h \andf
L(t) w \eqdefa - \rho^\h w\cdot  \nh v^\h. }
$$

\medbreak

Let us make some preliminary remarks about this system.   First of
all, we have that \beq\label{Lestimate} \|L(t)\|_{\cL(L^2_\h)} \leq
2\|\nh v^\h(t)\|_{L^\infty}. \eeq Thus thanks to\refeq{S3thm2eq1},
~$\|L(t)\|_{\cL(L^2_\h)}$  is integrable on $\R^+.$ According  to
the remark at the beginning of Section\refer{GlobalEnergy}, we
introduce
$$
\wt v^\h_z (t) \eqdefa v^\h_z(t) \exp \Bigl(-\int_0^t
\|L(t')\|_{L^\infty} dt'\Bigr).
$$
Then the energy estimate \eqref{Energybasic} gives
$$
\f12\frac d {dt} \|\sqrt \r^\h \wt
v^\h_z(t)\|_{L^2_\h}^2 + \|\nh \wt v^\h_z \|_{L^2_\h} ^2 \leq
(f_1|\wt v^\h_z)_{L^2_\h}.
$$

The external force~$f_1$ contains term with~$\rho_z$. We want of
course global estimate. But the control of~$L^p$ norm of~$\rho_z^\h$
demands the control of~$v^\h_z$ in~$L^1(\R^+;L^\infty)$ which will
be proved at the end. Thus we argue with a continuation argument.
More precisely, all the inequalities that follows are valid for~Ê$t$
less than~$T^\star_1$ defined by
\beq
\label {Tstardef1}
T^\star_1\eqdefa \sup \bigl\{ t\,/\,
\|\rho_z^\h\|_{L^\infty([0,t];L^2_\h\cap L^\infty_\h)} \leq 1
\bigr\}.
 \eeq

The first step of the study is the proof  that~$v^\h_z$ has the same
decay property as~$v^\h$ for the~$L^2_\h$ norm.  Then we apply
energy estimates of Section\refer {GlobalEnergy} to get the decay of
higher order derivatives.  The decay of~$v^\h_z$  in $L^2_\h$ norm
is given by the following proposition.

\begin{prop}
\label {decaydzL2} {\sl Under the assumptions of  Theorem
\ref{insslowvar}, for $t\leq T_1^\star,$ we have \ben\label{zp9}
\| v^\h_z(t)\|_{L^2_\h}^2  & \leq & \cC_0 h^2(z)  \langle t\rangle _{{}}^{-2} \andf\\
 \label{decaydzL2eq17}
 \| \nh v^\h_z(t)\|_{L^2_\h}^2  & \leq &  \cC_0
h^2(z) \langle t\rangle _{{}}^{-3}. \een}
\end{prop}

\begin{proof}
By the definition of~$f_1$ given by (D1INS2D),  and using
Inequality \eqref{S3thm2eq1} of Theorem\refer {S3thm2}, we get, for any~$t$ less than or equal to~$T^\star_1,$
\beq
\label {estimfondfD}
\|f_1(t)\|_{L^2_\h}  \leq  \| v^\h_t(t) \|_{L^2_\h} +
\|v^\h(t)\|_{L^\infty} \|\nh v^\h(t)\|_{L^2_\h}  \leq \cC_0  h(z)
\langle t\rangle_{{}}^{-2}.
 \eeq
 Then it
follows from Inequality~\eqref{Energybasicinteg} that for any
nonnegative $t_0$ and any time $t$ greater than or equal to~$t_0$
\ben \f12\|\sqrt{\r^\h}\wt
v^\h_z\|_{L^\infty([t_0,t];L^2_\h)}^2+\int_{t_0}^t\|\nh\wt v^\h_z
(t')\|_{L^2_\h}^2\,dt'&\leq&
\|\sqrt{\r}\wt v^\h_z (t_0)\|_{L^2_\h}^2+2\Bigl(\int_{t_0}^t\|f_1(t')\|_{L^2_\h}dt'\Bigr)^2 \nonumber\\
&\leq & \|\sqrt{\r_0}\wt
v^\h_z(t_0)\|_{L^2_\h}^2+\cC_0h^2(z)\w{t_0}^{-2}. \label{zp10} \een
This gives the bound of $\|\sqrt{\r^\h}\wt v^\h_z (t)\|_{L^2_\h}.$
In order to derive the decay of the~$L^2_\h$ norm of $v^\h_z,$ we
use Lemma\refer {LemmedecayWiegnerbasic} for $T=1,$ which gives that
\beno
 \f12\f{d}{dt}\|\sqrt \r^\h \wt v^\h_z(t)\|_{L^2}^2+ g^2(t)\|\sqrt \r^\h \wt v^\h_z(t)\|^2_{L^2}
\leq 2  g^2(t)\|\wt v^\h_{z,\flat}(t)\|_{L^2_\h}^2+(f_1(t)|\wt
v^\h_z(t))_{L^2_\h} , \eeno which  yields
 \beq \label {decaydzL2demoeq0}
 \f12\f{d}{dt}\|\sqrt \r^\h \wt v^\h_z(t)\|_{L^2_\h}^2+\frac 7{8 }  g^2(t)\|\sqrt \r^\h \wt v^\h_z(t)\|^2_{L^2_\h}
\leq 2    g^2(t)\|\wt
v^\h_{z,\flat}(t)\|_{L^2_\h}^2+2g^{-2}(t)\Bigl\|\f{f_1}{\sqrt{\r^\h}}(t)\Bigr\|_{L^2_\h}^2
. \eeq Then
  the main point is the estimate of~$\|\wt v^\h_{z,\flat}(t)\|_{L^2_\h}$.
We want to prove that, if we have that~Ê$g(\tau)\sim \langle
\tau\rangle^{-1}$ and $g(\tau) \leq \al \langle \tau\rangle^{-1}$
for $\al$ in~$ ]1,3[,$
\beq
\label {Estimateuzflat}
\begin{split}
&\|\wt v^\h_{z,\flat}(t)\|_{L^2_\h}^2    \leq \cC_0 h^2(z)\langle
t\rangle_{{}}^{-2} +\cC_0\langle t\rangle_{{}}^{-3}
\Bigl(\int_0^t\|\sqrt \r^\h \wt v^\h_z (t') \|_{L^2_\h} {dt'}
\Bigr)^2
\\
 &{}\qquad\qquad+2\|\varrho_0  \|_{L^\infty}^22\pi\|\sqrt{\r^\h}\wt v^\h_z(t)\|_{L^2_\h}^2+\cC_0 \langle t\rangle_{{}}^{-2}
\biggl (\int_0^t \langle t'\rangle^{-1}
\|\sqrt \r^\h \wt v^\h_z(t')\|_{L^2_\h}dt'\biggr)^2\cdotp
\end{split}
\eeq

In order to prove this inequality,  let us differentiate
Identity\refeq {estimuflatdemoeq2} with respect to the parameter~$z$
and then multiplying the resulting inequality by $\ds \exp
\Bigl(-\int_0^t \|L(t')\|_{L^\infty} dt'\Bigr).$ This gives \beq
\label {S7thm1demoeq1} |\wt v^\h_z(t,\xi_\h)| \leq e^{-t|\xi_\h|^2}
\bigl|\cF\PP\bigl(\partial_z(\r_0 v^\h_0)\bigr)(\xi_\h)\bigr| +\bigl
|\cF\mathbb{P} (\varrho^\h  \wt v^\h_z )(t,\xi_\h)\bigr|
+\sum_{\ell=1}^{5} |V_\ell(t,\xi_\h)|, \eeq with \beno
 V_1(t,\xi_\h) &\eqdefa & \cF\mathbb{P}(\rho_z^\h v^\h)(t,\xi_\h),
 \\
 V_2(t,\xi_\h) & \eqdefa& \int_0^t  e^{-(t-t')|\xi_\h|^2}|\xi_\h|^2\cF\mathbb{P}\bigl(\varrho^\h  \wt v^\h_z \bigr)(t',\xi_\h)\,dt',
 \\
 V_3(t,\xi_\h) & \eqdefa& \int_0^t  e^{-(t-t')|\xi_\h|^2}|\xi_\h|^2\cF\mathbb{P}(\r_z^\h  v^\h )(t',\xi_\h)\,dt',
 \\
  V_4(t,\xi_\h) & \eqdefa&\int_0^t  e^{-(t-t')|\xi_\h|^2}\cF\PP\bigl(\dive (\r_z^\h v^\h \otimes v^\h )\bigr)(t',\xi_\h)\,dt'
  \andf
  \\
    V_5(t,\xi_\h) & \eqdefa & \int_0^t  e^{-(t-t')|\xi_\h|^2}\cF\PP\bigl(\dive (\r^\h (\wt v^\h_z \otimes v^\h+v^\h \otimes \wt v^\h_z ) )\bigr)(t',\xi_\h)\,dt'.
\eeno
 During all the estimates of the above terms, we shall use that
the Leray projection operator on the divergence free vector
fields~$\PP$ decreases the pointwise values of the Fourier
transform. Following the lines for proving Inequality\refeq
 {LemmedecayWiegnerdemoeq5},  we use that
 $$
 \forall z\in \R\,,\ \int_{\R^2} \rho_0 v^\h_0(x_\h,z) dx_\h=0\quad\hbox{and thus}\quad
\forall z\in\R\,, \  \int_{\R^2} \partial_z (\rho_0 v^\h_0) (x_\h,z) dx_\h=0.
 $$
 This gives
 \beno
 e^{-t|\xi_\h|^2} \bigl|\cF\PP\bigl(\partial_z(\r_0 v^\h_0)\bigr)(\xi_\h)\bigr|
 & \leq &
  |\xi_\h| \,\|D_{\xi_\h} \cF(\partial_z(\r_0 v^\h_0))\|_{L^\infty_\h}\\
 & \leq &
   |\xi_\h| \,\| |x_\h| \partial_z(\r_0 v^\h_0)\|_{L^1(\R^2_\h)}.
 \eeno
 Using  Leibnitz formula in~$z$, we infer that
  $$
 \bigl| e^{-t|\xi_\h|^2 }\cF\PP\bigl(\partial_z(\r_0 v^\h_0)\bigr)(\xi_\h)\bigr| \leq U_1
 |\xi_\h|,
 $$  where $ U_1\eqdefa \|v^\h_0\|_{L^1(\R^2;|x|dx)}+
\|\partial_z v^\h_0\|_{L^1(\R^2;|x|dx)}.$
 And thus  by integration on~$S_1(t)$ given by Inequality\refeq{LemmedecayWiegnerdemoeq2a}, we~get
 \ben
\int_{S_1(t)} \bigl| \cF\PP\bigl(\partial_z(\r_0
v^\h_0)\bigr)(\xi_\h)\bigr|^2 d\xi_\h & \lesssim &   {U_1^2} \langle
t\rangle_{{}}^{-2} \,\cdotp
 \een
 It is obvious that
 \beq
\label {S7thm1demoeq2}
\int_{S_1(t)}\bigl |\cF\mathbb{P}\bigl(\varrho^\h \wt v^\h_z \bigr)(t,\xi_\h)\bigr|^2 d\xi_\h
 \leq
\|\varrho_0\|_{L^\infty}^2 (2\pi)^2 \|\sqrt \r^\h\wt v^\h_z(t)\|_{L^2_\h}^2.
 \eeq
Because~$t$ is less than~$T^\star_1$, we get, by using
Cauchy-Schwartz inequality, that \beno |V_1(t,\xi_\h)|\leq
\|\rho_z^\h v^\h(t)\|_{L^1_\h} \leq \|\rho_z^\h\|_{L^2_\h}
\|v^\h(t)\|_{L^2_\h}\leq
 \|v^\h(t)\|_{L^2_\h}. \eeno This gives, using
Inequality\refeq{S3thm2eq1} of Theorem\refer {S3thm2} and Inequality\refeq
{LemmedecayWiegnerdemoeq2b} about~$g$,
  \beq
 \label {S7thm1demoeq3} \int_{S_1(t)} |V_1(t,\xi_\h)|^2d\xi_\h \leq  \cC_0 h^2(z)\langle t\rangle_{{}}^{-3} .
\eeq Exactly like for Inequality\refeq {LemmedecayWiegnerdemoeq6},
we get \beq \label {S7thm1demoeq4} \int_{S_1(t)}
|V_2(t,\xi_\h)|^2d\xi_\h \lesssim  \langle t\rangle_{{}}^{-3}
\Bigl(\int_0^t\|\sqrt \r^\h\wt v^\h_z (t')
\|_{L^2_\h}{dt'}\Bigr)^2\cdotp \eeq Using that $\|\r_z(t)\|_{L^2_\h}
$ is less than or equal to~$1 $ and Inequality \refeq{S3thm2eq1}, we
get
$$
\bigl| \cF\mathbb {P}(\r_z^\h  v^\h )(t,\xi_\h)\bigr|\lesssim \|
\r_z^\h(t)\|_{L^2_\h} \|v^\h(t)\|_{L^2_\h}\lesssim \cC_0 h(z)\langle
t\rangle_{{}}^{-1} .
$$
We infer that \beq \label {S7thm1demoeq5} \int_{S_1(t)}
|V_3(t,\xi_\h)|^2d\xi_\h\leq \cC_0h^2(z)\int_{S_1(t)}|\xi_\h|^4
\log^2\langle t\rangle_{{}}d\xi_\h \leq \cC_0h^2(z)\langle
t\rangle_{{}}^{-3} \log^2\langle t\rangle_{{}}. \eeq Again as~$t$ is
less than~$T^\star_1$, we get, by using again Inequality
\refeq{S3thm2eq1}, that
$$
\bigl | \cF\PP\bigl(\dive (\r_z^\h v^\h \otimes v^\h
)\bigr)(t,\xi_\h)\bigr | \lesssim |\xi_\h|\|v^\h(t)\|^2_{L^2_\h}
\leq \cC_0h^2(z)|\xi_\h|\langle t\rangle^{-2}.
$$
Then we infer that
 \beq \label {S7thm1demoeq6} \int_{S_1(t)}
|V_4(t,\xi_\h)|^2d\xi_\h \leq \cC_0h^2(z)\int_{S_1(t)}
|\xi_\h|^2\Bigl( \int_0^t \langle t'\rangle^{-2}
 {dt'}\Bigr)^2d\xi_\h \leq \cC_0h^2(z)\langle
t\rangle_{{}}^{-2}\,\cdotp \eeq We get, by using again Inequality
\refeq{S3thm2eq1}, that \beno \bigl | \cF\PP\bigl(\dive (\r^\h (\wt
v^\h_z \otimes v^\h+v^\h \otimes \wt v^\h_z ) )\bigr)(t,\xi_\h) | &
\lesssim &
 |\xi_\h|\|v^\h(t)\|_{L^2_\h} \|\sqrt \r^\h\wt v^\h_z(t)\|_{L^2_\h} \\
& \leq & \cC_0\langle t \rangle _{{}}^{-1} |\xi_\h| \|\sqrt \r^\h
\wt v^\h_z(t')\|_{L^2_\h},
\eeno
 which implies
 \beq
 \label {S7thm1demoeq7}
  \int_{S_1(t)} |V_5(t,\xi_\h)|^2d\xi_\h \leq \cC_0
\langle t\rangle_{{}}^{-2} \biggl (\int_0^t \langle t'\rangle ^{-1}
\|\sqrt \r^\h \wt v^\h_z(t')\|_{L^2_\h}{dt'} \biggr)^2\cdotp \eeq
By summing up Inequalities \refeq {S7thm1demoeq1}--(\ref
{S7thm1demoeq7}), we get Inequality\refeq {Estimateuzflat}.

Then we infer from Inequalities\refeq  {decaydzL2demoeq0} and\refeq {Estimateuzflat} that
 \beq
 \label {S7thm1demoeq20}
 \begin{split} \f{d}{dt}\|\sqrt \r^\h &\wt v^\h_z(t)\|_{L^2_\h}^2+
g^2(t)\|\sqrt \r^\h \wt v^\h_z(t)\|^2_{L^2_\h}
  \leq \cC_0 h^2(z)\langle t\rangle_{{}}^{-3}
\\
& +\cC_0\langle t\rangle_{{}}^{-4} \Bigl(\int_0^t\|\sqrt \r^\h \wt
v^\h_z (t') \|_{L^2_\h}  {dt'} \Bigr)^2
 +\cC_0\w{t}^{-3}\Bigl (\int_0^t \langle t'\rangle^{-1}
\|\sqrt \r^\h \wt v^\h_z(t')\|_{L^2_\h}{dt'}\Bigr)^2\cdotp
\end{split}
\eeq
 Let $G
(\tau ) $ be given by \eqref{LemmedecayWiegnerdemoeq33}.  The above
formula writes after integration \ben \label{decaydzL2eq15} \|\sqrt
\r^\h \wt v^\h_z(t)\|_{L^2_\h}^2G(t) & \leq &  \|\sqrt \r^\h \wt
v^\h_z(0)\|_{L^2_\h}^2+ \cC_0  h^2(z) \int_0^t \langle
t'\rangle_{{}}^{-3}G(t'){dt'}
 \nonumber \\
&&\ + \cC_0\int_0^t\w{t'}^{-4} G(t')\Bigl(\int_0^{t'} \|\sqrt\r^\h
\wt v^\h_z
(t'')\|_{L^2_\h} {dt''} \Bigr)^2 {dt'}\\
&&\quad+\cC_0\int_0^t\w{t'}^{-3} G(t') \Bigl(\int_0^{t'}
\w{t''}^{-1}\|\sqrt\r^\h \wt v^\h_z (t'')\|_{L^2_\h}
{dt''}\Bigr)^2dt'. \nonumber \een Taking
$g^2(\tau)=\al\w{\tau}^{-1},$ for $\al$ in~$ ]1,2[,$ gives $G (\tau
)=\w{\tau}^\al. $ Then we get, by applying \eqref{zp10}, that \beno
\|\sqrt \r^\h \wt v^\h_z(t)\|_{L^2_\h}^2\w{t}^\al \leq \|\sqrt \r_0
v^\h_z(0)\|_{L^2_\h}^2 + \cC_0 h^2(z)\w{t}^{\al-1}, \eeno which
immediately implies \beno \|\sqrt \r^\h \wt v^\h_z(t)\|_{L^2_\h}^2
\leq \cC_0 h^2(z)\w{t}^{-1}. \eeno And thus we have \beno
&&\int_0^{t'} \|\sqrt\r^\h\wt v^\h_z (t'')\|_{L^2_\h} {dt''} \leq
\cC_0 h(z)\w{t'}^{\f12} \andf \int_0^{t'} \w{t''}^{-1}\|\sqrt\r^\h
\wt v^\h_z (t'')\|_{L^2_\h} {dt''} \leq \cC_0 h(z). \eeno Then
taking
 $g^2(\tau)=\al\w{\tau}^{-1},$ for $\al$ in~$]2,3[,$ in
 \eqref{decaydzL2eq15} gives rise to
\beno \|\sqrt \r^\h \wt v^\h_z(t)\|_{L^2_\h}^2\w{t}^\al \leq \|\sqrt
\r_0 v^\h_z(0)\|_{L^2_\h}^2 + \cC_0 h^2(z)\w{t}^{\al-2}, \eeno which
leads to
 \eqref{zp9}.

On the other hand,  for any $t_0\geq 0$ and any time $t\geq t_0,$ we
get, by applying Lemma \ref{DecayH1fond} and using
\eqref{estimfondfD}, that
\begin{equation}
\label{decaydzL2eq16}
\begin{split}
 &\|\nh \wt
v^\h_z(t)\|_{L^2_\h}^2+\int_{t_0}^t\Bigl(\|\sqrt{\r^\h}\p_t\wt
v^\h_z(t')\|_{L^2_\h}^2+\|\nh^2\wt
v^\h_z(t')\|_{L^2_\h}^2+\|\nh\wt \Pi_z^\h(t')\|_{L^2_\h}^2\Bigr)dt'\\
&\qquad\qquad\qquad\qquad\qquad\qquad\qquad\qquad\qquad\qquad\leq
C\|\na \wt v^\h_z(t_0)\|_{L^2}^2+\cC_0 h^2(z)\w{t_0}^{-3},
\end{split}
\end{equation}
and \beq \label{decaydzL2eq16ag} t \|\nh \wt
v^\h_z(t)\|_{L^2_\h}^2\leq C\| \wt
v^\h_z(t/2)\|_{L^2_\h}^2+\cC_0th^2(z)\w{t}^{-3}. \eeq Hence by
virtue of \eqref{zp9} and \eqref{decaydzL2eq16},
\eqref{decaydzL2eq16ag}, we achieve \eqref{decaydzL2eq17}. This
completes the proof of the proposition.
\end{proof}

\begin{prop}\label{SAlem2}
{\sl
Under the assumptions of  Theorem
\ref{insslowvar}, for $t\leq T_1^\star,$ we have
\ben
\label{SAeq2}
&&
 \|\p_tv^\h_z(t)\|_{L^2_\h}+ \|\nh^2
v^\h_z(t)\|_{L^2_\h}+\|\nh\Pi^\h_z(t)\|_{L^2_\h}\leq \cC_0h(z)\w{t}^{-2}\andf\\
&& \label{SAeq3} \qquad\qquad\|v^\h_z(t)\|_{L^\infty_\h} \leq
\cC_0h(z)\w{t}^{-\frac 3 2}. \een }
\end{prop}
\begin{proof}
In order to study the decay of the second space derivatives of
$v^\h_z,$ we need to use Lemma \ref{estimenergyH2Linearise}, which
gives
\ben
&&\f12\f{d}{dt}\|\sqrt{\r^\h}\p_tv^\h_z(t)\|_{L^2_\h}^2+\f34\|\nh\p_tv^\h_z\|_{L^2_\h}^2\leq
(\p_tf_1 | \p_tv^\h_z)_{L^2_\h}-\bigl((\r^\h v^\h_z\cdot\nh v^\h)_t
|
\p_tv^\h_z\bigr)_{L^2_\h} \nonumber\\
&&\qquad\qquad+CF_{1,\cT}(v^\h(t))\|\sqrt{\r^\h}\p_tv^\h_z(t)\|_{L^2_\h}^2+CF_{2,\cT}(v^\h(t),v^\h_z(t))
\label{Prop5.2eq1} \\
&&\qquad\qquad+C\|\nh\p_tv^\h_z(t)\|_{L^2_\h}^{\f12}\|\nh
v^\h_t(t)\|_{L^2_\h}^{\f12}\|\nh
v^\h_z(t)\|_{L^2_\h}\|\sqrt{\r^\h}\p_tv^\h_z(t)\|_{L^2_\h}^{\f12}\|\sqrt{\r^\h}v^\h_t(t)\|_{L^2_\h}^{\f12}.
\nonumber
\een

\no$\bullet$ \underline{The estimate of  $(\p_tf_1 |
\p_tv^\h_z)_{L^2_\h}.$}

 Note that \beno
\p_tf_1=-\p_t\r^\h_zv^\h_t-\r_z^\h\p_t^2v^\h-\p_t\r_z^\h
v^\h\cdot\nh v^\h-\r_z^\h v^\h_t\cdot\nh v^\h-\r_z^\h v^\h\cdot\nh
v^\h_t. \eeno By virtue of the transport equation of
\eqref{INS2dParameterb}, we get, by using integration by parts, that
\beno \bigl(\p_t\r_z^\h v^\h_t |
\p_tv^\h_z\bigr)_{L^2_\h}&=&-\bigl(\p_z\dive_\h(\r^\h v^\h)v^\h_t
| \p_tv^\h_z\bigr)_{L^2_\h}\\
&=&\bigl(\p_z(\r^\h v^\h)\cdot\nh v^\h_t |
\p_tv^\h_z\bigr)_{L^2_\h}+\bigl(\p_z(\r^\h v^\h)\otimes v^\h_t |
\nh\p_tv^\h_z\bigr)_{L^2_\h}.
\eeno
Then for $t$ less than or equal to~$T^\star_1,$
applying H\"older's inequality yields \beno \bigl|\bigl(\p_t\r_z^\h
v^\h_t | \p_tv^\h_z\bigr)_{L^2_\h}\bigr|&\leq
&\|\r_z^\h\|_{L^\infty}\|v^\h\|_{L^\infty_\h}\bigl(\|\nh
v^\h_t\|_{L^2_\h}\|\p_tv^\h_z\|_{L^2_\h}+\|\nh\p_tv^\h_z\|_{L^2_\h}\|v^\h_t\|_{L^2_\h}\bigr)\\
&&+\|\r^\h\|_{L^\infty}\|v^\h_z\|_{L^4_\h}\bigl(\|\nh
v^\h_t\|_{L^2_\h}\|\p_tv^\h_z\|_{L^4_\h}+\|\nh\p_tv^\h_z\|_{L^2_\h}\|v^\h_t\|_{L^4_\h}\bigr),
\eeno which together with \eqref{1.14a} and Young's inequality
ensures \beno \bigl|\bigl(\p_t\r_z^\h v^\h_t |
\p_tv^\h_z\bigr)_{L^2_\h}\bigr|&\leq&
\e\|\nh\p_tv^\h_z\|_{L^2_\h}^2+C_\e\Bigl((\|v^\h\|_{L^\infty_\h}+\|\nh
v^\h_z\|_{L^2_\h}^2)\|\p_tv^\h_z\|_{L^2_\h}^2+\|v^\h\|_{L^\infty_\h}^2\|v^\h_t\|_{L^2_\h}^2\nonumber\\
&&+(\|v^\h\|_{L^\infty_\h}+\|v^\h_z\|_{L^2_\h})\|\nh
v^\h_t\|_{L^2_\h}^2+\|v^\h_z\|_{L^2_\h}\|\nh
v^\h_z\|_{L^2_\h}\|v^\h_t\|_{L^2_\h}\|\nh v^\h_t\|_{L^2_\h}\Bigr).
\eeno Because of the induction hypothesis we get, for any
positive~$\epsilon>0,$ we have \beno \bigl|(\r_z^\h\p_t^2v^\h |
\p_tv^\h_z)_{L^2_\h}\bigr| &\leq &
\|\r_z^\h\|_{L^{\f2{1-\epsilon}}_\h}\|\p_t^2v^\h\|_{L^2_\h}\|\p_tv^\h_z\|_{L^{\f2{\epsilon}}_\h}\\
&\leq&
\f{C}{\epsilon} \|\nh\p_tv^\h_z\|_{L^2_\h}^{1-\epsilon} \|\p_tv^\h_z\|_{L^2_\h}^{\epsilon}
\|\p_t^2v^\h\|_{L^2_\h} \\
&\leq & C  \|\nh\p_tv^\h_z\|_{L^2_\h}^{1-\epsilon}  \langle t
\rangle^{-\epsilon} \|\p_tv^\h_z\|_{L^2_\h}^{\epsilon} \frac
{\langle t \rangle^{\epsilon} } \epsilon \|\p_t^2v^\h\|_{L^2_\h} .
\eeno
H\"older inequality implies that for any couple~$(\epsilon,\e)$ of positive real numbers, we have
$$
 \bigl|(\r_z^\h\p_t^2v^\h | \p_tv^\h_z)_{L^2_\h}\bigr|\leq \e\|\nh\p_tv^\h_z\|_{L^2_\h}^2+
 C_\e\Bigl(\w{t}^{-2}\|\p_tv^\h_z\|_{L^2_\h}^2+\frac {\langle t \rangle^{2\epsilon} }{ \epsilon^2} \|\p_t^2v^\h\|^2_{L^2_\h} \Bigr).
$$
Taking $2\epsilon=\log^{-1}\w{t}$ in the above inequality
gives rise to
\beq \label{Prop5.2eq1a}
 \bigl|(\r_z^\h\p_t^2v^\h | \p_tv^\h_z)_{L^2_\h}\bigr|\leq \e\|\nh\p_tv^\h_z\|_{L^2_\h}^2+
 C\Bigl(\w{t}^{-2}\|\p_tv^\h_z\|_{L^2_\h}^2+\log^2\w{t}\|\p_t^2v^\h\|_{L^2_\h}^2\Bigr).
 \eeq
 Using again the transport equation of \eqref{INS2dParameterb}  and integration by parts,
 one has
 $$\longformule{
 \bigl(\p_t\r_z^\h v^\h\cdot\nh v^\h | \p_tv^\h_z\bigr)_{L^2_\h}= \bigl(\p_z(\r^\h v^\h)\cdot\nh v^\h\cdot\nh v^\h |
 \p_tv^\h_z\bigr)_{L^2_\h}}{{}
+\bigl(\p_z(\r^\h v^\h)\cdot(v^\h\cdot\nh)\nh v^\h |
\p_tv^\h_z\bigr)_{L^2_\h} +\bigl(\p_z(\r^\h v^\h)\otimes
(v^\h\cdot\nh) v^\h | \nh\p_tv^\h_z\bigr)_{L^2_\h}.} $$
It is easy observe that  because of induction hypothesis\refeq {Tstardef1}, we have
\beno
&&\bigl|\bigl(\p_z(\r^\h v^\h)\cdot\nh
v^\h\cdot\nh v^\h |
 \p_tv^\h_z\bigr)_{L^2_\h}\bigr|
+\bigl|\bigl(\p_z(\r^\h v^\h)\cdot(v^\h\cdot\nh)\nh v^\h |
\p_tv^\h_z\bigr)_{L^2_\h}\bigr|\\
&&\quad{} \lesssim \Bigl(\|v^\h\|_{L^\infty_\h}\bigl(\|\nh
v^\h\|_{L^\infty_\h}\|\nh
v^\h\|_{L^2_\h}+\|v^\h\|_{L^\infty_\h}\|\nh^2v^\h\|_{L^2_\h}\bigr)+\|\nh
v^\h\|_{L^\infty_\h}^2\|v^\h_z\|_{L^2_\h}\Bigr)\|\p_tv^\h_z\|_{L^2_\h}\\
&&\qquad\qquad\qquad\qquad\qquad\qquad\qquad\qquad\qquad\qquad\qquad
+\|v^\h\|_{L^\infty_\h}\|v^\h_z\|_{L^4_\h}\|\nh^2v^\h\|_{L^2_\h}\|\p_tv^\h_z\|_{L^4_\h},
\eeno and
$$
\longformule{
 \bigl|\bigl(\p_z(\r^\h v^\h)\otimes
(v^\h\cdot\nh) v^\h| \nh\p_tv^\h_z\bigr)_{L^2_\h}\bigr|
}
{{}
\lesssim
\bigl(\|v^\h\|_{L^\infty_\h}^2\|\nh
v^\h\|_{L^2_\h}+\|v^\h_z\|_{L^2_\h}\|v^\h\|_{L^\infty_\h}\|\nh
v^\h\|_{L^\infty_\h}\bigr)\|\nh\p_tv^\h_z\|_{L^2_\h}.
}
$$
Hence by applying Young's inequality, one has
 \beno
 &&\bigl| \bigl(\p_t\r_z^\h v^\h\cdot\nh v^\h | \p_tv^\h_z\bigr)_{L^2_\h}\bigr|\leq
 \e
 \|\nh\p_tv^\h_z\|_{L^2_\h}^2+C\bigl(\|v^\h\|_{L^\infty_\h}+\|v^\h_z\|_{L^2_\h}^2+\|\nh
 v^\h\|_{L^\infty_\h}\bigr)\|\p_tv^\h_z\|_{L^2}^2\nonumber\\
 &&\qquad\quad\qquad{}
 +C\|v^\h\|_{L^\infty_\h}\bigl(\|\nh
 v^\h\|_{L^\infty_\h}^2\|\nh
 v^\h\|_{L^2_\h}^2+\|v^\h\|_{L^\infty_\h}^2\|\nh^2v^\h\|_{L^2_\h}^2\bigr)
 +C\|\nh v^\h\|_{L^\infty_\h}^3\|v^\h_z\|_{L^2_\h}^2\nonumber\\
&&\qquad\qquad\qquad{}+\|v^\h\|_{L^\infty_\h}^2\bigl(\|\nh
v^\h_z\|_{L^2_\h}\|\nh^2v^\h\|_{L^2_\h}^2+\|v^\h\|_{L^\infty_\h}^2\|\nh
v^\h\|_{L^2_\h}^2+\|\nh
v^\h\|_{L^\infty_\h}^2\|v^\h_z\|_{L^2_\h}^2\bigr). \eeno Similarly,
one has \beno \bigl|\bigl(\r_z^\h v^\h_t\cdot\nh v^\h |
\p_tv^\h_z\bigr)_{L^2_\h}\bigr|&\leq&
\|\r_z^\h\|_{L^\infty_\h}\|v^\h_t\|_{L^2_\h}\|\nh
v^\h\|_{L^4_\h}\|\p_tv^\h_z\|_{L^4_\h}\\
&\leq& \e \|\nh\p_tv^\h_z\|_{L^2_\h}^2+C\bigl(\|\nh
v^\h\|_{L^2_\h}^2\|\p_tv^\h_z\|_{L^2_\h}^2+\|v^\h_t\|_{L^2_\h}^2\|\nh^2v^\h\|_{L^2_\h}\bigr),
\eeno and \beno \bigl|\bigl(\r_z^\h v^\h\cdot\nh v^\h_t |
\p_tv^\h_z\bigr)_{L^2_\h}\bigr|&\leq&
\|\r_z^\h\|_{L^\infty_\h}\|v^\h\|_{L^\infty_\h}\|\nh
v^\h_t\|_{L^2_\h}\|\p_tv^\h_z\|_{L^2_\h}\\
&\leq& C\bigl(\|
v^\h\|_{L^\infty_\h}\|\p_tv^\h_z\|_{L^2_\h}^2+\|v^\h\|_{L^\infty_\h}\|\nh
v^\h_t\|_{L^2_\h}^2\bigr). \eeno This  together with Theorem
\ref{S3thm2} and Proposition \ref{decaydzL2} ensures that \beq
\label{Prop5.2eq4} \begin{split} \bigl|(\p_tf_1 |
\p_tv^\h_z)_{L^2_\h}\bigr|\ \leq & \ 3\e
 \|\nh\p_tv^\h_z\|_{L^2_\h}^2\\
 &\ +\cC_0\Bigl(\w{t}^{-\frac 3 2}\|\p_tv^\h_z\|_{L^2_\h}^2+h^2(z)\w{t}^{-6}+\log^2\w{t}\|\p_t^2v^\h_z\|_{L^2_\h}^2\Bigr).
 \end{split}
 \eeq

\no$\bullet$ \underline{The estimate of $\bigl((\r^\h v^\h_z\cdot\nh
v^\h)_t | \p_tv^\h_z\bigr)_{L^2_\h}.$ }

It is easy to observe that \beno \bigl|\bigl(\r_t^\h v^\h_z\cdot\nh
v^\h |
\p_tv^\h_z\bigr)_{L^2_\h}\bigr|&\leq&\|\r_t^\h\|_{L^\infty_\h}\|v^\h_z\|_{L^4_\h}\|\nh
v^\h\|_{L^4_\h}\|\p_tv^\h_z\|_{L^2_\h}\\
&\leq & C\|\nh
v^\h\|_{L^2_\h}\|\p_tv^\h_z\|_{L^2_\h}^2+C\|\r_t^\h\|_{L^\infty_\h}^2\|v^\h_z\|_{L^2_\h}\|\nh
v^\h_z\|_{L^2_\h}\|\nh^2 v^\h\|_{L^2_\h}, \eeno and
 \beno \bigl|\bigl(\r^\h
\p_tv^\h_z\cdot\nh v^\h | \p_tv^\h_z\bigr)_{L^2_\h}\bigr|\leq \|\nh
v^\h\|_{L^\infty_\h}\|\p_tv^\h_z\|_{L^2_\h}^2, \eeno and
 \beno \bigl|\bigl(\r^\h
v^\h_z\cdot\nh v^\h_t |
\p_tv^\h_z\bigr)_{L^2_\h}\bigr|&\leq&\|\r^\h\|_{L^\infty_\h}\|v^\h_z\|_{L^4_\h}\|\nh
v^\h_t\|_{L^2_\h}\|\p_tv^\h_z\|_{L^4_\h}\\
&\leq
&\e\|\nh\p_tv^\h_z\|_{L^2_\h}^2+C\|v^\h_z\|_{L^2_\h}^2\|\p_tv^\h_z\|_{L^2_\h}^2+C\|\nh
v^\h_z\|_{L^2_\h}\|\nh v^\h_t\|_{L^2_\h}^2. \eeno Then we infer from
Theorem \ref{S3thm2} and Proposition \ref{decaydzL2}  that
\beq
 \bigl|\bigl((\r^\h
v^\h_z\cdot\nh v^\h)_t | \p_tv^\h_z\bigr)_{L^2_\h}\bigr|\leq
\e\|\nh\p_tv^\h_z\|_{L^2_\h}^2+\cC_0\w{t}^{-\f32}\|\p_tv^\h_z\|_{L^2_\h}^2+\cC_0h^2(z)\w{t}^{-\bigl(\f{13}2\bigr)}.
\label{Prop5.2eq5} \eeq

\no$\bullet$ \underline{The closure of the energy estimate}

Applying Young's inequality ensures that
\beq
\begin{split} &C\|\nh\p_tv^\h_z\|_{L^2_\h}^{\f12}\|\nh
v^\h_t\|_{L^2_\h}^{\f12}\|\nh
v^\h_z\|_{L^2_\h}\|\sqrt{\r^\h}\p_tv^\h_z\|_{L^2_\h}^{\f12}\|\sqrt{\r^\h}v^\h_t\|_{L^2_\h}^{\f12}\\
&\qquad\quad\leq \e\|\nh\p_tv^\h_z\|_{L^2_\h}^2+C\|\nh
v^\h_z\|_{L^2_\h}^2\|\sqrt{\r^\h}\p_tv^\h_z\|_{L^2_\h}^2+C\|\nh
v^\h_t\|_{L^2_\h}\|\nh v^\h_z\|_{L^2_\h}\|v^\h_t\|_{L^2_\h}.
\end{split} \label{Prop5.2eq6}
 \eeq Resuming the Estimates
(\ref{Prop5.2eq4}-\ref{Prop5.2eq6}) into \eqref{Prop5.2eq1} leads to
\beq
\begin{split}
&\f{d}{dt}\|\sqrt{\r^\h}\p_tv^\h_z(t)\|_{L^2_\h}^2+\|\nh\p_tv^\h_z\|_{L^2_\h}^2\leq
\cC_0\bigl(\w{t}^{-\f32}+F_{1,\cT}(v^\h(t))\bigr)\|\sqrt{\r^\h}\p_tv^\h_z(t)\|_{L^2_\h}^2\\
&\qquad\qquad+CF_{2,\cT}(v^\h(t),v^\h_z(t))+\cC_0h^2(z)\w{t}^{-6}+C\log^2\w{t}\|\p_t^2v^\h(t)\|_{L^2_\h}^2.
\end{split}
\label{Prop5.2eq7}
 \eeq
Applying \eqref{estimenergyH2Lineariseeq1},
\eqref{estimenergyH2Lineariseeq2} for $u=v^\h,$ $v=v^\h_z$ and
$f(t)=f_1(t)$ and using \eqref{estimfondfD} gives \beno
\int_{t_0}^t F_{1,\cT}(v^\h(t'))dt'&\leq & \cC_0\left(
\|v^\h(t_0)\|_{L^2_\h}^4+{\|\nh
v^\h(t)\|_{L^2_\h}^2}\bigl/{\cT}\right)
\andf\\
\int_{t_0}^tF_{2,\cT}(v^\h(t),v^\h_z(t'))dt'&\leq& \cC_0\bigl(\|\nh
v^\h_z(t_0)\|_{L^2_\h}^2+h^2(z)\w{t_0}^{-3}\bigr)\\
&&\quad\times \left(\bigl(1+\|v^\h(t_0)\|_{L^2}\bigr)\|\nh
v^\h(t_0)\|_{L^2_\h}^2+\cT\right),\eeno from which, Theorem
\ref{S3thm2},  and taking $\cT=\|\na v^\h(t_0)\|_{L^2_\h}$ in the
 above inequality, we infer \beq \label{Prop5.2eq8} \begin{split} \int_{t_0}^t
F_{1,\cT}(v^\h(t'))dt'\leq & \cC_0 \andf\\
\int_{t_0}^tF_{2,\cT}(v^\h(t),v^\h_z(t'))dt'\leq & \cC_0
h^2(z)\w{t_0}^{-3}\|\nh v^\h(t_0)\|_{L^2_\h}^2. \end{split} \eeq
Thanks to \eqref{Prop5.2eq8}, we get, by applying Gronwall's Lemma
to \eqref{Prop5.2eq7}, that for any nonnegative~$t_0$ and any $t$
greater than or equal to~$ t_0,$ \beq \begin{split}
&\|\sqrt{\r^\h}\p_tv^\h_z(t)\|_{L^2_\h}^2+\int_{t_0}^t\|\nh\p_tv^\h_z(t')\|_{L^2_\h}^2dt'\leq
\cC_0h^2(z)\int_{t_0}^t\w{t'}^{-6}dt'\\
&\qquad\qquad\qquad\qquad+C\int_{t_0}^t\log^2\w{t'}\|\p_t^2v^\h(t')\|_{L^2_\h}^2dt'+\cC_0h^2(z)\w{t_0}^{-3}\|\nh
v^\h(t_0)\|_{L^2_\h}^2. \end{split} \label{Prop5.2eq9} \eeq

On the other hand, by multiplying Inequality~\eqref{Prop5.2eq7} by $t-t_0$ and
then applying Gronwall's Lemma, we obtain
$$
\longformule{ (t-t_0)\|\sqrt{\r^\h}\p_tv^\h_z(t)\|_{L^2_\h}^2\leq
\int_{t_0}^t\|\sqrt{\r^\h}\p_tv^\h_z(t')\|_{L^2_\h}^2dt'+\cC_0h^2(z)(t-t_0)\w{t_0}^{-3}\|\nh
v^\h(t_0)\|_{L^2_\h}^2 } { {} +
\cC_0h^2(z)\int_{t_0}^t(t'-t_0)\w{t'}^{-6}dt'+C\int_{t_0}^t\log^2\w{t'}(t'-t_0)\|\p_t^2v^\h(t')\|_{L^2_\h}^2dt'.
}
$$
Taking $t_0$ equal to~$ t/2$ in the above inequality and using again
\eqref{decaydzL2eq16}, we obtain \beno
t\|\sqrt{\r^\h}\p_tv^\h_z(t)\|_{L^2_\h}^2\leq \cC_0\Bigl(
h^2(z)\w{t}^{-3}+\w{t}^{-4_-}\int_{t/2}^t\w{t'}^{5_-}\|\p_t^2v^\h(t')\|_{L^2_\h}^2dt'\Bigr),
\eeno from which, \eqref{S3thm2eq8} and \eqref{Prop5.2eq9}, we
conclude \beq \label{Prop5.2eq11} \|\p_tv^\h_z(t)\|_{L^2_\h}\leq
\cC_0 h(z)\w{t}^{-2}. \eeq

On the other hand, it follows from Lemma \ref{Echanget2x} that $$
\longformule{
\|\nh^2v^\h_z(t)\|_{L^2_\h}+\|\nh\Pi_z^\h(t)\|_{L^2_\h}
\lesssim
\|\p_tv^\h_z(t)\|_{L^2_\h}+\|v^\h(t)\|_{L^4_\h}^2\|\nh
v^\h_z(t)\|_{L^2_\h}
}
{{}
+\|f_1(t)\|_{L^2_\h}+\|\nh
v^\h(t)\|_{L^\infty_\h}\|\nh v^\h_z(t)\|_{L^2_\h},
}
$$
from which,
\eqref{estimfondfD} and \eqref{S3thm2eq1}, we deduce \beno
\|\nh^2v^\h_z(t)\|_{L^2_\h}+\|\nh\Pi_z^\h(t)\|_{L^2_\h}\leq \cC_0
h(z)\w{t}^{-2}. \eeno Together with \eqref{Prop5.2eq11}, this proves
\eqref{SAeq2}.

Finally \eqref{SAeq3} follows from the interpolation Inequality
\eqref{1.14b}, \eqref{zp9} and \eqref{SAeq2}. This completes the
proof of the proposition.
\end{proof}

\begin{col}
\label{SAcol1} {\sl The time~$T^\star_1$ given by
Relation~\eqref{Tstardef1} equals $+\infty,$ and there hold
\beq\label{2.21qwe} \|\nh\r_z^\h(t)\|_{L^\infty_\h\cap
L^2_\h}\leq\cC_0\eta h(z) \andf \|\p_t\r^\h_z(t)\|_{L^\infty_\h\cap
L^2_\h}\leq\cC_0\eta h(z)\w{t}^{-\frac 3 2}.
 \eeq }
\end{col}
\begin{proof}
It follows from the transport equation of (D1INS2D),
\eqref{S3thm2eq3} and \eqref{SAeq3} that  \beno
\|\r_z^\h(t)\|_{L^2_\h\cap L^\infty_\h} \leq
\|\p_z\r_0(\cdot,z)\|_{L^2_\h\cap
L^\infty_\h}+\|\nh\r^\h\|_{L^\infty_t(L^2_\h\cap
L^\infty_\h)}\|v^\h_z\|_{L^1_t(L^\infty_\h)} \leq \cC_0\eta.
 \eeno Therefore as long as $\eta$ is
sufficiently small, we have \beno \|\r_z^\h(t)\|_{L^\infty_{\rm
v}(L^2_{\rm h})\cap L^\infty}\leq \f12\quad\mbox{for}\quad t\leq
T^\star_1. \eeno This in turn shows that $T^\star_1$ given by
Relation \eqref{Tstardef1}  equals $+\infty.$ Moreover, there hold
\eqref{zp9}, \eqref{decaydzL2eq17} and\refeq{SAeq2} and\refeq{SAeq3}
for any $t<\infty.$ Then we deduce from Lemma  \ref{Echanget2x} and
(D1INS2D) that \beno \|\nh^2v^\h_z\|_{L^4_\h}+\|\nh
\Pi_z^\h\|_{L^4_\h} & \leq &
C\bigl\|\bigl(\r^\h\p_tv^\h_z+\r^\h v^\h\cdot\nh v^\h_z+\r_z^\h v^\h_t+\p_z(\r^\h v^\h)\cdot\nh v^\h\bigr)\bigr\|_{L^4_\h}\\
&\leq &
C\Bigl(\|\p_t v^\h_z\|_{L^4_\h}+\|v^\h_t\|_{L^4_\h}+\|v^\h\|_{L^\infty_\h}\bigl(\|\nh v^\h_z\|_{L^4_\h}\\
&&\qquad\qquad\qquad\quad +\|\nh v^\h\|_{L^4_\h}\bigr)+
\|v^\h_z\|_{L^\infty_\h}\|\nh v^\h\|_{L^4_\h}\Bigr), \eeno from
which \eqref{1.14a},  \eqref{S3thm2eq1} and Propositions
\ref{decaydzL2} and \ref{SAlem2}, we infer \beno \|\nh^2
v^\h_z\|_{L^4_\h}+\|\nh \Pi_z^\h\|_{L^4_\h}\leq
\cC_0\bigl(h(z)\w{t}^{-\frac  9 4}+h^{\f12}(z)\w{t}^{-1}\|\nh\p_tv^\h_z\|_{L^2_\h}^{\f12}\bigr),
\eeno so that in view of  \eqref{decaydzL2eq17} and
\eqref{Prop5.2eq9}, we conclude \beq \begin{split}
 \|\nh v^\h_z\|_{L^1_t(L^\infty)}\leq &
C\int_0^t\|\nh v^\h_z(t')\|_{L^2}^{\f13}\|\nh^2 v^\h_z(t')\|_{L^4}^{\f23}dt'\\
\leq
&\cC_0\Bigl(h(z)+h^{\f23}(z)\int_0^t\w{t'}^{-\bigl(\f76\bigr)}\|\nh\p_t
v^\h_z(t')\|_{L^2}^{\f13}\,dt'\Bigr)\leq \cC_0h(z). \end{split}
\label{2.19af}\eeq

While by taking $\nh$ to the transport equation of (D1INS2D), one
has \beq \label{2.19mnb} \p_t\nh\r_z^\h+v^\h\cdot\nh\nh\r_z^\h=-\nh
v^\h\cdot\nh\r_z^\h-\nh v_z^\h\cdot\nh\r^\h-v^\h_z\cdot\nh\nh\r^h,
\eeq from which, we infer
$$
\longformule{
\|\nh\r_z^\h(t)\|_{L^2_\h\cap
L^\infty_\h}\leq
\Bigl(\|\p_z\nh\r_0(z)\|_{L^2_\h\cap L^\infty_\h}+\|\nh v^\h_z(t)\|_{L^1_t(L^\infty_\h)}\|\nh\r^\h\|_{L^\infty_t(L^2_\h\cap L^\infty_\h)}
}
{ {}
+\|v^\h_z\|_{L^1_t(L^\infty_\h)}\|\nh^2\r^\h\|_{L^\infty_t(L^2_\h\cap
L^\infty_\h)}\Bigr)\exp\bigl(\|\nh
v^\h\|_{L^1_t(L^\infty_\h)}\bigr),
}
$$ which together with \eqref{S3thm2eq2} and
\eqref{SAeq3}  ensures the first inequality of \eqref{2.21qwe}.

Finally it follows from the transport equation of (D1INS2D),
\eqref{S3thm2eq1} and \eqref{SAeq3} that \beno
\|\p_t\r_z^\h(t)\|_{L^2_\h\cap L^\infty_\h}&\leq&
\|\nh\r_z^\h(t)\|_{L^2_\h\cap
L^\infty_\h}\|v^\h(t)\|_{L^\infty_\h}+\|v^\h_z(t)\|_{L^\infty_\h}\|\nh\r^\h(t)\|_{L^2_\h\cap
L^\infty_\h}\\
&\leq& \cC_0 \eta h(z)\w{t}^{-\frac 3 2}. \eeno This
completes the proof of the Corollary.
\end{proof}

\begin{prop}\label{SAlem3} {\sl Under the assumptions of  Theorem
\ref{insslowvar}, one has  \ben &&\label{2.22} \|\nh\p_t
v^\h_z(t)\|_{L^2_\h}^2 +\int_{t}^\infty \bigl(\|\p_t^2v^\h_{z}\|^2
_{L^2_\h} + \|\nh^2\p_t
v^\h_z\|^2_{L^2_\h}+\|\nh\partial_t\Pi_z^\h\|_{L^2_\h}^2\bigr)
dt'\leq \cC_0
h^2(z)\w{t}^{-5},\\
&&\label{PropSAlem3eq2}\int_0^t\w{t'}^{5_-}\bigl(\|\p_t^2v^\h_{z}\|^2
_{L^2_\h} + \|\nh^2\p_t
v^\h_z\|^2_{L^2_\h}+\|\nh\partial_t\Pi^\h_z\|_{L^2_\h}^2\bigr) dt'\leq \cC_0 h^2(z),\\
&&\label{2.23} \|\nh^3
v^\h_z(t)\|_{L^2_\h}+\|\nh^2\Pi_z^\h(t)\|_{L^2_\h}\leq
\cC_0h(z)\w{t}^{-\f52}\log\w{t}. \een}
\end{prop}

\begin{proof} In view of (D1INS2D), $\p_tv^\h_z$ is a solution of
(LINS2D) with external force $\wt f_1(t)$ given~by
\beno
\wt f_1(t)=-(\r^\h v^\h)_t\cdot\nh v^\h_z-(\r_z^\h v^\h_t)_t-(\r_z^\h
v^\h\cdot\nh v^\h)_t-(\r^\h v^\h_z\cdot\nh v^\h)_t, \eeno so that
\beno \|\wt f_1(t)\|_{L^2_\h}&\leq
&\|\r^\h\|_{L^\infty_\h}\bigl(\|\nh
v^\h\|_{L^\infty_\h}\|\p_tv^\h_z\|_{L^2_\h}+\|
v^\h_z\|_{L^\infty_\h}\|\nh v^\h_t\|_{L^2_\h}\bigr)\\
&&+\|\r_t^\h\|_{L^\infty_\h}\bigl(\| v^\h\|_{L^\infty_\h}\|\nh
v^\h_z\|_{L^2_\h}+\|
v^\h_z\|_{L^\infty_\h}\|\nh v^\h\|_{L^2_\h}\bigr)\\
&&+\|\r_z^\h\|_{L^\infty_\h}\bigl(\|\p_t^2v^\h\|_{L^2_\h}+\| \nh
v^\h\|_{L^\infty_\h}\| v^\h_t\|_{L^2_\h}+\|
v^\h\|_{L^\infty_\h}\|\nh v^\h_t\|_{L^2_\h}\bigr)\\
&&+\|\p_t\r_z^\h\|_{L^\infty_\h}\bigl(\| v^\h_t\|_{L^2_\h}+\|
v^\h\|_{L^\infty_\h}\|\nh v^\h\|_{L^2_\h}\bigr). \eeno Then from the
estimates in the previous sections, we deduce that for any
nonnegative $t_0$ and any time $t\geq t_0$ \beno \int_{t_0}^t\|\wt
f_1(t')\|_{L^2_\h}^2dt'\leq \cC_0
h^2(z)\w{t_0}^{-6}+\int_{t_0}^t\|\p_t^2v^\h(t')\|_{L^2_\h}^2dt'\leq
\cC_0 h^2(z)\w{t_0}^{-5}, \eeno where we used \eqref{propdecayD3eq1}
 in the last step. Then applying Lemma \ref{DecayH1fond} gives
 \beq
  \label{PropSAlem3eq1}
  \begin{split}
  &
 \|\nh
\p_tv^\h_z (t)\|^2 _{L^2_\h} +\int_{t_0}^\infty
\bigl(\|\p_t^2v^\h_{z}  (t')\|^2 _{L^2_\h} + \|\nh^2\p_t
v^\h_z(t')\|^2_{L^2_\h}+\|\nh\partial_t\Pi_z^\h
(t')\|_{L^2_\h}^2\bigr) dt'\\
&\qquad\qquad\qquad\qquad\qquad
\qquad\qquad\qquad\qquad\quad\lesssim \|\nh\p_t
v^\h_z(t_0)\|_{L^2_\h}^2+\cC_0h^2(z)\w{t_0}^{-5},
\end{split}
\eeq
and
\beno t\|\nh\p_tv^\h_z(t)\|_{L^2_\h}^2\lesssim
\|\p_tv^\h_z(t/2)\|_{L^2_\h}^2+t\int_{t/2}^\infty\|\wt
f_1(t')\|_{L^2_\h}^2dt'\leq \cC_0 h^2(z)\w{t}^{-4}.
 \eeno
 This
yields \beno \|\nh\p_tv^\h_z(t)\|_{L^2_\h}^2\leq \cC_0
h^2(z)\w{t}^{-5}. \eeno Taking $t_0$ equal to~$t$ in
Inequality~\eqref{PropSAlem3eq1} and resuming the above estimate
into the resulting inequality leads to Estimate~\eqref{2.22}. With
this inequality and \eqref{S3thm2eq8}, we conclude the proof of
\eqref{PropSAlem3eq2} by applying  \eqref{LemDecayH1fondeq2} for
$s=5_-$ and $v=\p_tv^\h_z.$

In order to prove \eqref{2.23}, we get, by differentiate (D1INS2D),
that
$$\longformule{ \D_\h\p_jv^\h_z-\p_j\nh\Pi_z^\h=-\p_j(\r^\h\p_tv^\h_z)-\p_j(\r^\h
v^\h\cdot\nh v^\h_z)}{{}-\p_j(\r_z^\h v^\h_t)-\p_j(\r_z^\h
v^\h\cdot\nh v^\h)-\p_j(\r^\h v^\h_z\cdot\nh v^\h).} $$ Applying
Lemma \ref{Echanget2x} gives
\beno
&&\|\nh^3v^\h_z\|_{L^2_\h}+\|\nh^2\Pi_z^\h\|_{L^2_\h}\lesssim
\|\nh\r^\h\p_tv^\h_z\|_{L^2_\h}+\|\nh\r_z^\h v^\h_t\|_{L^2_\h}+\|\r^\h\|_{L^\infty}\bigl(\|\nh\p_tv^\h_z\|_{L^2_\h}\\
&&\quad +\|\nh v^\h\|_{L^\infty_\h}\|\nh
v^\h_z\|_{L^2_\h}+\|v^\h\|_{L^\infty_\h}\|\nh^2v^\h_z\|_{L^2_\h}+\|v^\h_z\|_{L^\infty_\h}\|\nh^2v^\h\|_{L^2_\h}\bigr)\\
&&\quad+\|\nh\r^\h\|_{L^\infty_\h}\bigl(\|v^\h\|_{L^\infty_\h}\|\nh
v^\h_z\|_{L^2_\h}+\|v^\h_z\|_{L^\infty_\h}\|\nh
v^\h\|_{L^2_\h}\bigr)+\|\nh
\r_z^\h\|_{L^\infty_\h}\|v^\h\|_{L^\infty_\h}\|\nh
v^\h\|_{L^2_\h}\\
&&\quad +\|\r_z^\h\|_{L^\infty_\h}\bigl(\|\nh v^\h_t\|_{L^2_\h}+\|\nh
v^\h\|_{L^\infty_\h}\|\nh
v^\h\|_{L^2_\h}+\|v^\h\|_{L^\infty_\h}\|\nh^2v^\h\|_{L^2_\h}\bigr),
\eeno
from which and the previous decay estimates, we infer
$$
\|\nh^3v^\h_z\|_{L^2_\h}+\|\nh^2\Pi_z^\h\|_{L^2_\h}\lesssim
\|\nh\r^\h\p_tv^\h_z\|_{L^2_\h}+\|\nh\r_z^\h
v^\h_t\|_{L^2_\h}+\cC_0h(z)\w{t}^{-\frac 5 2}.
$$ Yet
along the same line to the proof of \eqref{S7eq6demoeq4} and using
\eqref{S3thm2eq3}, \eqref{2.21qwe}, we have \beno
\|\nh\r^\h\p_tv^\h_z(t)\|_{L^2_\h}+\|\nh\r_z^\h
v^\h_t(t)\|_{L^2_\h}\leq \cC_0 h(z)\w{t}^{-\f52}\log\w{t}. \eeno
This proves Inequality~\eqref{2.23} and finishes the proof of
Proposition \ref{SAlem3}.
\end{proof}

Let us now turn to the proof of Theorem \ref{SAprop1}.

\begin{proof}[Proof of Theorem \ref{SAprop1}] By summing up Propositions \ref{decaydzL2}, \ref{SAlem2}  and
\ref{SAlem3}, we conclude that $v^\h_z$ shares the same decay
estimates of Inequalities \eqref{S3thm2eq1} and \eqref{S3thm2eq8}.
It remains to prove \eqref{thmSAprop1eq1}.
 Indeed by virtue of \eqref{S3thm2eq2}, Theorem 3.14 of\ccite{BCD} and the transport equation of (D1INS2D), we~get
 \beq
 \label{thmSAprop1eq2}
 \begin{split}
 \|\r^\h_z(t)\|_{H^3_\h}\leq &
 C\Bigl(\|\p_z\r_0\|_{H^3_\h}+\int_0^t\bigl(\|\v_z^\h(t')\|_{L^\infty}\|\nh\r^\h(t')\|_{H^3_\h}\\
 &\quad+\|\nh
 v^h_z(t')\|_{H^2_\h}\|\nh\r^\h(t')\|_{L^\infty}\bigr)dt'\Bigr)\exp\Bigl(C\|\nh
 v^h\|_{L^1_t(H^2_\h)}\Bigr)\leq
 \cC_0\eta h(z). \end{split}
 \eeq
Then
 by taking one more
horizontal derivative to  \eqref{2.19mnb} and using
\eqref{thmSAprop1eq2}, we infer
\beno
\|\nh^2\p_t\r^\h_z(t)\|_{L^2_\h}
&\leq &
\|v^\h(t)\|_{L^\infty}\|\nh^3\r^\h_z(t)\|_{L^2_\h}+2\|\nh
v^\h(t)\|_{L^\infty}\|\nh^2\r^\h_z(t)\|_{L^2_\h}\\
&&{}+\|\nh^2v^\h(t)\|_{L^2_\h}\|\nh\r_z^\h(t)\|_{L^\infty}+\|\nh^2v^\h_z(t)\|_{L^2_\h}\|\nh\r^\h(t)\|_{L^\infty_\h}\\
&&{}+2\|\nh
v^\h_z(t)\|_{L^\infty}\|\nh^2\r^\h(t)\|_{L^2_\h}+\|v^\h_z(t)\|_{L^\infty}\|\nh^3\r^\h(t)\|_{L^2_\h}\\
&\leq & \cC_0\eta h(z)\w{t}^{-\frac 3 2}.
 \eeno
 This together with Inequality~\eqref{2.21qwe} and \eqref{thmSAprop1eq2} proves  Estimate (\refer{thmSAprop1eq1}).
\end{proof}

\setcounter{equation}{0}

\section{Decay estimates of $\p_{zz} v^\h$}
\label {decayL2partialz2}

Similar to Section \ref{decayL2partialz}, let us denote $\p_z^2v^\h$
by $v^\h_{zz}$ and also $\p_z^2\r^\h$ by $\r_{zz}^\h$ and
$\p_z^2\Pi^\h$ by $\Pi_{zz}^\h.$ Then $(\r_{zz}^\h,v_{zz}^\h,
\Pi_{zz}^\h)$ solves the system in $\R^+\times\R^2$
$$
\displaylines{ {\rm (D2INS2D)}\quad \left\{\begin{array}{c}
\displaystyle
 \p_t\r^\h_{zz}+v^\h\cdot\nh \r^\h_{zz}=-2v^\h_z\cdot \nh \rho^\h_z-v^\h_{zz}\cdot\nh\r^\h,  \\
\displaystyle \r^\h\pa_t v^\h_{zz}+ \r^\h v^\h\cdot\nh v^\h_{zz}  -\D_\h v^\h_{zz} + \nh \Pi_{zz}=f_2 +L(t) v^\h_{zz},\\
\displaystyle \dive_\h v^\h_{zz}=0, \\
\displaystyle  (\r^\h_{zz}, v^\h_{zz})|_{t=0}=(\eta\partial_z^2
\varsigma_0,
\partial_z^2 v^\h_0).
\end{array}\right.
\with\cr f_2=-\r_z^\h\p_tv^\h_z -\p_z(\rho_z^\h v^\h_t)
-\p_z(\rho_z^\h v^\h\cdot\nh v^\h)-\r_z^\h v^\h_z\cdot\nh v^\h-\r^\h
v^\h_z\cdot\nh v^\h_z }
$$ and the linear operator $L(t)$ is given by (D1INS2D).

In this section, we basically follow the same line as that in
Section \ref{decayL2partialz} to derive the required decay estimates
for the second derivatives of the family of solutions to \eqref
{INS2dParameterb} with respect to the parameter $z.$ The main result
can be listed as follows:

\begin{thm}\label{SBprop1} {\sl Under the assumptions of Theorem
\ref{insslowvar},
 $v^\h_{zz}$ shares the same decay estimates of Inequalities \eqref{S3thm2eq1} and \eqref{S3thm2eq8}. Moreover, there holds
 \beq \label{thSBprop1eq1}
\|\r^\h_{zz}(t)\|_{H^2_\h}+\w{t}^{\f32}\|\p_t\r^\h_{zz}(t)\|_{H^1_\h}\leq
\cC_0\eta h(z). \eeq}
\end{thm}

Again due to \eqref{Lestimate},
 as  in Section \ref{decayL2partialz}, we introduce
$$
\wt v^\h_{zz} (t) \eqdefa v^\h_{zz}(t) \exp \Bigl(-\int_0^t
\|L(t')\|_{L^\infty} dt'\Bigr).
$$
Then the energy estimate \eqref{Energybasic} gives
$$
\f12\frac d {dt} \|\sqrt \r^\h \wt
v^\h_{zz}(t)\|_{L^2_\h}^2 + \|\nh \wt v^\h_{zz} \|_{L^2_\h} ^2 \leq
(f_2|\wt v^\h_{zz})_{L^2_\h}.
$$
The external force~$f_2$ contains term with~$\rho_{zz}^\h$. The
global estimate $\rho_{zz}^\h$ demands the control of~$v^\h_{zz}$
in~$L^1(\R^+;L^\infty)$ which will be achieved hereafter. Thus we
again argue with a continuation argument. More precisely, all the
inequalities that follows are valid  for~$t$ less than~$T^\star_2$
defined by
\beq
\label{Tstardef2} T^\star_2\eqdefa \sup \bigl\{ t\,/
\, \|\rho^\h_{zz}\|_{L^\infty([0,t];L^2_\h\cap L^\infty_\h)} \leq 1
\bigr\}. \eeq

Next we decompose the proof of Theorem \ref{SBprop1} into the
following propositions:

\begin{prop}\label{Propsect6.1} {\sl Under the assumptions of Theorem
\ref{insslowvar}, for $t\leq T_2^\star,$ we have \ben
\label{Propsect6.1eq1}&& \|v^\h_{zz}(t)\|_{L^2_\h}^2\leq \cC_0
h^2(z)\w{t}^{-2}, \andf\\
&&\label{5.9} \|\nh v^\h_{zz}(t)\|_{L^2_\h}^2\leq \cC_0
h^2(z)\w{t}^{-3} . \een}
\end{prop}

\begin{proof} Note that for $f_2$ given by (D2INS2D), we have
\beno \|f_2(t)\|_{L^2}&\leq
&\|\r^\h\|_{L^\infty_\h}\|v^\h_z\|_{L^\infty_\h}\|\nh
v^\h_z\|_{L^2_\h}+\|\r_z^\h\|_{L^\infty_\h}\bigl(2\|\p_tv^\h_z\|_{L^2_\h}+\|v^\h_z\|_{L^\infty_\h}\|\nh
v^\h\|_{L^2_\h}\\
&&\qquad\quad+\|v^\h\|_{L^\infty_\h}\|\nh
v^\h_z\|_{L^2_\h}\bigr)+\|\r^\h_{zz}\|_{L^\infty_\h}\bigl(\|v^\h_t\|_{L^2_\h}+\|v^\h\|_{L^\infty_\h}\|\nh
v^\h\|_{L^2_\h}\bigr), \eeno from which, Theorem \ref{S3thm2}  and
Theorem \ref{SAprop1}, we infer  for $t\leq T_2^\star$ \beq
\label{Propsect6.1eq2} \|f_2(t)\|_{L^2_\h}\leq \cC_0 h(z)\w{t}^{-2}.
\eeq Then we deduce from \eqref{Energybasicinteg} that for any
nonnegative $t_0$ and any time $t\geq t_0$ \beq
\label{Propsect6.1eq3} \f12\|\sqrt{\r^\h}\wt
v^\h_{zz}\|_{L^\infty([t_0,t]; L^2_\h)}^2+\int_{t_0}^t\|\nh\wt
v^\h_{zz}(t')\|_{L^2_\h}^2dt'\lesssim \|\sqrt{\r^\h}\wt
v^\h_{zz}(t_0)\|_{L^2_\h}^2+\cC_0h^2(z)\w{t_0}^{-2}. \eeq This gives
the estimate of $\|\sqrt{\r^\h}\wt v^\h_{zz}(t)\|_{L^2_\h}.$ In
order to the derive the decay of the $L^2_\h$ norm of~$v^\h_{zz},$
we use Lemma \ref{LemmedecayWiegnerbasic} for $T=1$  to get \beq
\label{Propsect6.1eq4}
 \f12\f{d}{dt}\|\sqrt \r^\h \wt v^\h_{zz}(t)\|_{L^2_\h}^2
 +\frac 7{8}  g^2(t)\|\sqrt \r^\h \wt v^\h_{zz}(t)\|^2_{L^2_\h}
\leq  2    g^2(t)\|\wt
v^\h_{zz,\flat}(t)\|_{L^2_\h}^2+2g^{-2}(t)\Bigl\|\f{f_2}{\sqrt{\r^\h}}(t)\Bigr\|_{L^2_\h}^2
. \eeq As in the proof of Proposition \ref{decaydzL2}, it amounts to
derive the estimate of $\|\wt v^\h_{zz,\flat}(t)\|_{L^2_\h},$ which
we present as follows. We first get, by differentiating \eqref
{estimuflatdemoeq2} twice with respect to the parameter~$z,$ that
\beno \widehat v^\h_{zz}(t,\xi_\h)&= & e^{-t|\xi|^2}\cF\mathbb
{P}\p_z^2(\r_0 v^\h_0)(\xi_\h)-
\cF\mathbb {P}\p_z^2(\varrho^\h v^\h)(t,\xi_\h)\\
&&+\int_0^te^{-(t-t')|\xi_\h|^2}\cF\mathbb
{P}\Bigl(\D\p_z^2(\varrho^\h v^\h)-\dive\p_z^2(\r^\h v^\h\otimes
v^\h)\Bigr)(t',\xi_\h)dt',\eeno so that along the same line to the
proof of \eqref {decaydzL2demoeq0}, we have \beno &&\|\wt
v^\h_{zz,\flat}(t)\|_{L^2_\h}^2 \leq 2\|\varrho_0
\|_{L^\infty}2\pi\|\sqrt{\r^\h}\wt v^\h_{zz}\|_{L^2_\h}^2+
C\Bigl(\int_{S_1(t)}\bigl|\cF\p_z^2(\r_0
v^\h_0)\bigr|^2\,d\xi_\h\\
&&\qquad\qquad\qquad+\int_{S_1(t)}\bigl(\bigl|\cF(\r_z^\h
v^\h_z)\bigr|^2+\bigl|\cF
(\r^\h_{zz} v^\h)\bigr|^2\bigr)\,d\xi_\h\\
&&{}+g^6(t)\Bigl(\int_0^t\|\p_z^2(\varrho^\h \wt
v^\h)(t')\|_{L^1_\h}\,dt'\Bigr)^2+g^4(t)\Bigl(\int_0^t\|\p_z^2\bigl(\r^\h
v^\h\otimes\wt v^\h\bigr)(t')\|_{L^1_\h}dt'\Bigr)^2\Bigr). \eeno
Similar to the estimate of \eqref{LemmedecayWiegnerdemoeq5}, one has
\beno \int_{S_1(t)}\bigl|\cF\p_z^2(\r_0v^\h_0)\bigr|^2\,d\xi_\h\leq
Cg^4(t)\Bigl(\int_{\R^2}|x||\p_z^2(\r_0v^\h_0)|\,dx\Bigr)^2.\eeno It
follows from \eqref{S3thm2eq1}  and  Theorem \ref{SAprop1} that
$$
\longformule{
\int_{S_1(t)}\bigl(\bigl|\cF(\r_z^\h v^\h_z)\bigr|^2+\bigl|\cF
(\r^\h_{zz} v^\h)\bigr|^2\bigr)\,d\xi_\h
}
{{}
\leq
Cg^2(t)\bigl(\|\r_z^\h(t)\|_{L^2_\h}^2\|v^\h_z(t)\|_{L^2_\h}^2+\|\r^\h_{zz}(t)\|_{L^2_\h}^2\|v^\h(t)\|_{L^2_\h}^2\bigr)\leq
\cC_0 h^2(z)g^2(t)\w{t}^{-2},
}
$$
and \beno \int_0^t\|\p_z^2(\varrho^\h\wt v^\h)(t')\|_{L^1_\h}dt'
&\leq &
\int_0^t\bigl(\|\r^\h_{zz}\|_{L^2_\h}\|v^\h\|_{L^2_\h}+2\|\r_z^\h\|_{L^2_\h}\|v^\h_z\|_{L^2_\h}
+\|\varrho^\h\|_{L^2_\h}\|\wt v^\h_{zz}\|_{L^2_\h}\bigr)dt'\\
&\leq & \cC_0\Bigl(h(z)\ln\w{t}+\int_0^t\|\wt
v^\h_{zz}(t')\|_{L^2_\h}\,dt'\Bigr),
\eeno
and \
\beno
\int_0^t\|\p_z^2\bigl(\r^\h v^\h\otimes \wt
v^\h\bigr)(t')\|_{L^1_\h}dt'
&\leq&
C\int_0^t\Bigl(\|\r^\h_{zz}\|_{L^2_\h}\|v^\h\|_{L^2_\h}\|v^\h\|_{L^\infty_\h}
+\|\r_z^\h\|_{L^\infty_\h}\|v^\h\|_{L^2_\h}\|v^\h_z\|_{L^2_\h}\\
&&\qquad\quad+\|\r^\h\|_{L^\infty_\h}\bigl(\|v^\h\|_{L^2_\h}\|\wt v^\h_{zz}\|_{L^2_\h}+\|v^\h_z\|_{L^2_\h}^2\bigr)\Bigr)dt'\\
& \leq &\cC_0\Bigl(h^2(z)+\int_0^t\w{t'}^{-1}\|\wt
v^\h_{zz}(t')\|_{L^2_\h}\,dt'\Bigr).
\eeno
As a consequence, we obtain
$$\longformule{
 \|\wt
v^\h_{zz,\flat}(t)\|_{L^2_\h}^2 \leq
2\|\varrho_0  \|_{L^\infty}2\pi\|\sqrt{\r}\wt v^\h_{zz}\|_{L^2_\h}^2
}
{
{}+\cC_0g^4(t)\Bigl(h^2(z)+g^2(t)\Bigl(\int_0^t\|\wt
v^\h_{zz}(t')\|_{L^2_\h}dt'\Bigr)^2 +\Bigl(\int_0^t\w{t'}^{-1}\|\wt
v^\h_{zz}(t')\|_{L^2_\h}dt'\Bigr)^2\Bigr).
}
$$
Resuming the above estimate into \eqref{Propsect6.1eq4} and using
\eqref{Propsect6.1eq2}, we conclude
$$
\longformule{ \f{d}{dt}\|\sqrt \r^\h \wt v^\h_{zz}(t)\|_{L^2_\h}^2+
g^2(t)\|\sqrt \r^\h \wt v^\h_{zz}(t)\|^2_{L^2_\h} \leq \cC_0
h^2(z)\w{t}^{-3} } { {} +\cC_0\w{t}^{-4}\Bigl(\int_0^t\|\wt
v^\h_{zz}(t')\|_{L^2_\h}dt'\Bigr)^2
+\cC_0\w{t}^{-3}\Bigl(\int_0^t\w{t'}^{-1}\|\wt
v^\h_{zz}(t')\|_{L^2_\h}dt'\Bigr)^2, }
$$ which is exactly the same as the Inequality \eqref{S7thm1demoeq20}.
  Hence by using
\eqref{Propsect6.1eq3} and a similar derivation of \eqref{zp9}, we
achieve \eqref{Propsect6.1eq1}.

On the other hand, it follows  from \eqref{LemDecayH1fondeq1} and
\eqref{Propsect6.1eq2} that for any nonnegative $t_0$ and any time
$t\geq t_0$
 \beq
 \label{PropSBlem2eq1}
 \begin{split}
&\|\nh \wt v^\h_{zz}(t)\|_{L^2_\h}^2+\int_{t_0}^t\bigl(\|\p_t\wt v^\h_{zz}(t')\|_{L^2_\h}^2
+\|\nh^2\wt v^\h_{zz}(t')\|_{L^2_\h}^2+\|\nh \wt \Pi_{zz}^\h(t')\|_{L^2_\h}\bigr)dt'\\
&\qquad\qquad\qquad\qquad\qquad\qquad\qquad\qquad\qquad
\leq \|\nh
v^\h_{zz}(t_0)\|_{L^2_\h}^2+\cC_0h^2(z)\int_{t_0}^t\w{t'}^{-4}dt'.
\end{split}
\eeq
 It follows from \eqref{LemDecayH1fondeq3} that
$$
\f{t}2\|\nh \wt v^\h_{zz}(t)\|_{L^2_\h}^2 \lesssim \|\wt
v^\h_{zz}({t}/2)\|_{L^2_\h}^2+\cC_0h^2(z)t\int_{{t}/2}^t\w{t'}^{-4}dt',
$$
 so that thanks to \eqref{Propsect6.1eq1} and
\eqref{PropSBlem2eq1}, we obtain  \eqref{5.9}. This completes the
proof of the proposition.
\end{proof}

\begin{prop}
\label{SBlem3}
{\sl
Under the assumptions of Theorem \ref{insslowvar}, for $t\leq
T_2^\star,$  one has
\ben &&\label{5.14}
\|\p_tv^\h_{zz}(t)\|_{L^2_\h}^2+\|\nh^2v^\h_{zz}(t)\|_{L^2_\h}^2+\|\nh\Pi_{zz}^\h(t)\|_{L^2_\h}^2\leq
\cC_0h^2(z)\w{t}^{-4},\\
 &&\label{5.17}
\|v^\h_{zz}(t)\|_{L^\infty_\h}\leq
\cC_0h(z)\w{t}^{-\frac 3 2}\andf \|\nh
v^\h_{zz}\|_{L^1_t(L^\infty_\h)}\leq \cC_0h(z). \een }
\end{prop}

\begin{proof} We use again Lemma \ref{estimenergyH2Linearise} to the System (D2INS2D) to get
\ben
&&\f12\f{d}{dt}\|\sqrt{\r^\h}\p_tv^\h_{zz}(t)\|_{L^2_\h}^2+\f34\|\nh\p_tv^\h_{zz}\|_{L^2_\h}^2\leq
(\p_tf_2 | \p_tv^\h_{zz})_{L^2_\h}-\bigl(\p_t(\r^\h
v^\h_{zz}\cdot\nh v^\h) |
\p_tv^\h_{zz}\bigr)_{L^2_\h}\nonumber\\
&&\qquad\qquad{}+CF_{1,\cT}(v^\h(t))\|\sqrt{\r^\h}\p_tv^\h_{zz}(t)\|_{L^2_\h}^2+CF_{2,\cT}(v^\h(t),v^\h_{zz}(t))  \label{PropSBlem3eq1}\\
&&\qquad\qquad{}+C\|\nh\p_tv^\h_{zz}(t)\|_{L^2_\h}^{\f12}\|\nh
v^\h_t(t)\|_{L^2_\h}^{\f12}\|\nh
v^\h_{zz}(t)\|_{L^2_\h}\|\sqrt{\r^\h}\p_tv^\h_{zz}(t)\|_{L^2_\h}^{\f12}\|\sqrt{\r^\h}v^\h_t(t)\|_{L^2_\h}^{\f12}.\nonumber
\een We first deal with the estimate of $(\p_tf_2 |
\p_tv^\h_{zz})_{L^2}.$ It is easy to observe that $$\longformule{
\p_tf_2=-\p_t(\r^\h_{zz}v_t^\h)-2\p_t(\r_z^\h\p_tv^\h_z)-\p_t(\r_{zz}^\h
v^\h\cdot\nh v^\h)}{{}-2\p_t(\r^\h_z v^\h_z\cdot\nh v^\h)-\p_t(\r^\h
v^\h_z\cdot\nh v^\h_z)-\p_t(\r^\h_zv^\h\cdot\nh v^\h_z).} $$ By
using the transport equation of \eqref{INS2dParameterb} and
integration by parts, we have \beno (\p_t\r^\h_{zz}v^\h_t |\
\p_tv^\h_{zz})_{L^2_\h}=\bigl(\p_z^2(\r^\h v^\h)\cdot\nh v^\h_t |
\p_tv^\h_{zz}\bigr)_{L^2_\h}+\bigl( v^\h_t | \p_z^2(\r^\h
v^\h)\cdot\nh \p_tv^\h_{zz}\bigr)_{L^2_\h}, \eeno from which and the
known decay estimates, we infer \beno \bigl|(\p_t\r^\h_{zz}v^\h_t |
\p_tv^\h_{zz})_{L^2_\h}\bigr| &\leq&
\|\r^\h\|_{L^\infty_\h}\|v^\h_{zz}\|_{L^4_\h}\bigl(\|\nh
v^\h_t\|_{L^2_\h}\|\p_tv^\h_{zz}\|_{L^4_\h}+\|v^\h_t\|_{L^4_\h}\|\nh\p_tv^\h_{zz}\|_{L^2_\h}\bigr)\\
&&
\!\!\!\!\!\!\!\!\!\!\!\!\!\!\!\!\!\!\!\!\!\!\!\!\!\!\!\!\!\!\!\!\!\!\!\!\!\!\!\!\!\!\!\!\!\!\!\!{}+\bigl(\|\r_z^\h\|_{L^\infty_\h}\|v^\h_z\|_{L^\infty_\h}+\|\r^\h_{zz}\|_{L^\infty_\h}\|v^\h\|_{L^\infty_\h}\bigr)
\bigl(\|\nh
v^\h_t\|_{L^2_\h}\|\p_tv^\h_{zz}\|_{L^2_\h}+\|v^\h_t\|_{L^2_\h}\|\nh\p_tv^\h_{zz}\|_{L^2_\h}\bigr)\\
& \leq &
\f{\e}2\|\nh\p_tv^\h_{zz}\|_{L^2_\h}^2+\f12\w{t}^{-2}\|\p_tv^\h_{zz}\|_{L^2_\h}^2+\cC_0h^2(z)\w{t}^{-6}.
\eeno
It follows from a similar derivation of \eqref{Prop5.2eq1a} that
 \beno
 \bigl|(\r^\h_{zz}\p_t^2v^\h |
\p_tv^\h_{zz})_{L^2_\h}\bigr|\leq
\f{\e}2\|\nh\p_tv^\h_{zz}\|_{L^2_\h}^2+\f12\w{t}^{-2}\|\p_tv^\h_{zz}\|_{L^2_\h}^2+C\log^2\w{t}\|\p_t^2v^\h\|_{L^2_\h}.
\eeno This gives
$$
 \bigl|\bigl(\p_t(\r^\h_{zz}v^\h_{t}) |
\p_tv^\h_{zz}\bigr)_{L^2_\h}\bigr|\leq
{\e}\|\nh\p_tv^\h_{zz}\|_{L^2_\h}^2+\w{t}^{-2}\|\p_tv^\h_{zz}\|_{L^2_\h}^2+C\log^2\w{t}\|\p_t^2v^\h\|_{L^2_\h}+\cC_0h^2(z)\w{t}^{-6}.
$$
Similarly, thanks to \eqref{2.21qwe} and \eqref{SAeq2}, we
have
 \beno  \bigl|\bigl(\p_t(\r_{z}^\h\p_tv^\h_{z}) |
\p_tv^\h_{zz}\bigr)_{L^2_\h}\bigr|
&\leq&
\|\p_t\r_z^\h\|_{L^\infty_\h}\|\p_tv^\h_z\|_{L^2_\h}\|\p_tv^\h_{zz}\|_{L^2_\h}+\bigl|(\r_z^\h\p_t^2v^\h_z | \p_tv^\h_{zz})_{L^2_\h}\bigr|\\
&\leq&
{\e}\|\na\p_tv^\h_{zz}\|_{L^2_\h}^2+\w{t}^{-\frac 3 2}\|\p_tv^\h_{zz}\|_{L^2_\h}^2\\
&&\qquad\qquad\qquad+C\log\w{t}^2\|\p_t^2v^\h_z\|_{L^2_\h}+\cC_0h^2(z)\w{t}^{-\frac  {11} 2}.
\eeno
 In order to deal with $\bigl(\p_t(\r^\h_{zz} v^\h\cdot\nh
v^\h) | \p_tv^\h_{zz}\bigr)_{L^2_\h},$ we use the transport equation
of \eqref{INS2dParameterb}
 and  integration by parts to get
 \beno
\bigl(\p_t\r^\h_{zz}v^\h\cdot\nh v^\h |
\p_tv^\h_{zz}\bigr)_{L^2_\h} &=& \bigl(\r^\h_{zz}v^\h\cdot\nh v^\h |
v^\h\cdot\nh\p_tv^\h_{zz}\bigr)_{L^2_\h}\\
&&\!\!\!\!\!\!{}+\bigl(\r^\h_{zz}\bigl((v^\h\cdot\nh v^\h)\cdot\nh
v^\h+(v^\h\otimes v^\h):(\nh\otimes \nh)v^\h\bigr)
 | \p_tv^\h_{zz}\bigr)_{L^2_\h}\\
&&{}-\bigl((2v^\h_z\cdot\nh\r_{z}^\h+v^\h_{zz}\cdot\nh\r^\h)v^\h\cdot\nh
v^\h | \p_tv^\h_{zz}\bigr)_{L^2_\h},
\eeno
 from which, \eqref{2.21qwe} and the previous decay estimates, we deduce that for $t\leq
T_2^\star$ \beno
 \bigl|\bigl(\p_t\r^\h_{zz}v^\h\cdot\nh
v^\h |\p_tv^\h_{zz}\bigr)_{L^2_\h}\bigr|
&\leq &
C\|v^\h\|_{L^\infty_\h}\Bigl(\|\nh v^\h\|_{L^\infty_\h}\bigl(\|\nh v^\h\|_{L^2_\h}+\|v^\h_z\|_{L^2_\h}+\|v^\h_{zz}\|_{L^2_\h}\bigr)\\
&&\!\!\!\!\!\!\!\!\!\!\!\!{}+\|v^\h\|_{L^\infty_\h}\|\nh^2v^\h\|_{L^2_\h}\bigr)\|\p_tv^\h_{zz}\|_{L^2_\h}+\|v^\h\|_{L^\infty_\h}\|\nh v^\h\|_{L^2_\h}\|\nh\p_tv^\h_{zz}\|_{L^2_\h}\Bigr)\\
&\leq &\f12
\w{t}^{-2}\|\p_tv^\h_{zz}\|_{L^2_\h}^2+\e\|\nh\p_tv^\h_{zz}\|_{L^2_\h}^2+\cC_0
h^2(z)\w{t}^{-6}.
 \eeno
 It is easy observe that
 \beno
&&\bigl(\r^\h_{zz}(v^\h_t\cdot\nh v^\h+v^\h\cdot\nh v^\h_t) |
\p_tv^\h_{zz}\bigr)_{L^2_\h}\bigr|\\
&&\qquad\qquad{}\leq\|\r^\h_{zz}\|_{L^\infty_\h}\bigl(\|\nh
v^\h\|_{L^4_\h}\|v^\h_t\|_{L^4_\h}+\|v^\h\|_{L^\infty_\h}\|\nh
v^\h_t\|_{L^2_\h}\bigr)\|\p_tv^\h_{zz}\|_{L^2_\h}\\
&& \qquad\qquad{}\leq \f12
\w{t}^{-2}\|\p_tv^\h_{zz}\|_{L^2_\h}^2+\cC_0 h^2(z)\w{t}^{-6}. \eeno
As a consequence, we obtain \beno
 \bigl|\bigl(\p_t(\r_{zz}^\h v^\h\cdot\nh v^\h) |
\p_tv^\h_{zz}\bigr)_{L^2_\h}\bigr|\leq
\w{t}^{-2}\|\p_tv^\h_{zz}\|_{L^2_\h}^2+\e\|\nh\p_tv^\h_{zz}\|_{L^2_\h}^2+\cC_0
h^2(z)\w{t}^{-6}. \eeno

Using again the previous decay estimates, we have
 \beno
 \bigl|\bigl(\p_t(\r_{z}^\h v^\h_z\cdot\nh v^\h) |
\p_tv^\h_{zz}\bigr)_{L^2_\h}\bigr|
&\leq&\Bigl(\|\p_t\r_z^\h\|_{L^\infty_\h}\|v^\h_z\|_{L^\infty_\h}\|\nh
v^\h\|_{L^2_\h}+\|\r_z^\h\|_{L^\infty_\h}\bigl(\|\nh
v^\h\|_{L^4_\h}\|\p_tv^\h_z\|_{L^4_\h}\\
&& \qquad\qquad\qquad\qquad\qquad\qquad{}+\|v^\h_z\|_{L^\infty_\h}\|\nh
v^\h_t\|_{L^2_\h}\bigr)\Bigr)\|\p_tv^\h_{zz}\|_{L^2_\h}\\
&\leq & \w{t}^{-2}\|\p_tv^\h_{zz}\|_{L^2_\h}^2+\cC_0
h^2(z)\w{t}^{-6}.
\eeno
 Similarly, one has
\beno
 \bigl|\bigl(\p_t(\r_{z}^\h v^\h\cdot\nh v^\h_z) |
\p_tv^\h_{zz}\bigr)_{L^2_\h}\bigr|
&\leq&
\Bigl(\|\p_t\r_z^\h\|_{L^\infty_\h}\|v^\h\|_{L^\infty_\h}\|\nh
v^\h_z\|_{L^2_\h}+\|\r_z^\h\|_{L^\infty_\h}\bigl(\|\nh
v^\h_z\|_{L^4_\h}\|v^\h_t\|_{L^4_\h}\\
&&\qquad\qquad\qquad\qquad\qquad\quad{}+\|v^\h\|_{L^\infty_\h}\|\nh
\p_tv^\h_z\|_{L^2_\h}\bigr)\Bigr)\|\p_tv^\h_{zz}\|_{L^2_\h}\\
&\leq & \w{t}^{-2}\|\p_tv^\h_{zz}\|_{L^2_\h}^2+\cC_0
h^2(z)\w{t}^{-6},
 \eeno and
 \beno
 \bigl|\bigl(\p_t(\r^\h v^\h_z\cdot\nh v^\h_z) |
\p_tv^\h_{zz}\bigr)_{L^2_\h}\bigr|
&\leq&
\Bigl(\|\r_t^\h\|_{L^\infty_\h}\|v^\h_z\|_{L^\infty_\h}\|\nh
v^\h_z\|_{L^2_\h}+\|\r^\h\|_{L^\infty}\bigl(\|\nh
v^\h_z\|_{L^4_\h}\|\p_tv^\h_z\|_{L^4_\h}\\
&&\qquad\qquad\qquad\qquad\qquad\quad{}+\|v^\h_z\|_{L^\infty_\h}\|\nh
\p_tv^\h_z\|_{L^2_\h}\bigr)\Bigr)\|\p_tv^\h_{zz}\|_{L^2_\h}\\
&\leq & \w{t}^{-2}\|\p_tv^\h_{zz}\|_{L^2_\h}^2+\cC_0
h^2(z)\w{t}^{-6},
 \eeno
 and
 \ben
 \bigl|\bigl(\p_t(\r^\h
v^\h_{zz}\cdot\nh v^\h) |
\p_tv^\h_{zz}\bigr)_{L^2_\h}\bigr|&\leq&\Bigl(\|\r_t^\h\|_{L^\infty_\h}\|v^\h_{zz}\|_{L^4_\h}\|\nh
v^\h\|_{L^4_\h}+\|\r^\h\|_{L^\infty}\|\nh
v^\h\|_{L^\infty_\h}\|\p_tv^\h_{zz}\|_{L^2_\h}\Bigr)\nonumber\\
&&\label{PropSBlem3eq3}\qquad\qquad{}\times
\|\p_tv^\h_{zz}\|_{L^2_\h}+\|\r^\h\|_{L^\infty}\|v^\h_{zz}\|_{L^4_\h}\|\nh
v^\h_t\|_{L^2_\h}\|\p_tv^\h_{zz}\|_{L^4_\h}\\
&\leq &\e\|\nh\p_tv^\h_{zz}\|_{L^2_\h}^2+
\w{t}^{-\frac 3 2}\|\p_tv^\h_{zz}\|_{L^2_\h}^2+\cC_0
h^2(z)\w{t}^{-6}.\nonumber
 \een
Hence we achieve
 \beq
 \begin{split}
\bigl|(\p_tf_2 | \p_tv^\h_{zz})_{L^2_\h}\bigr|\leq&
3\e\|\na\p_tv^\h_{zz}\|_{L^2_\h}^2+C\w{t}^{-\frac 3 2}\|\p_tv^\h_{zz}\|_{L^2_\h}^2\\
&
+C\log^2\w{t}(\|\p_t^2v^\h\|_{L^2_\h}+\|\p_t^2v^\h_z\|_{L^2_\h})+\cC_0h^2(z)\w{t}^{-\frac  {11} 2}.
\end{split} \label{PropSBlem3eq4}
\eeq
Substituting the Estimates \eqref{PropSBlem3eq3}, \eqref{PropSBlem3eq4} and a similar version
of \eqref{Prop5.2eq6}  into \eqref{PropSBlem3eq1} leads to
\beq
\begin{split}
&\f{d}{dt}\|\sqrt{\r^\h}\p_tv^\h_{zz}(t)\|_{L^2_\h}^2+\|\nh\p_tv^\h_{zz}\|_{L^2_\h}^2\leq
C\bigl(\w{t}^{-\frac 3 2}+F_{1,\cT}(v^\h(t))\bigr)\|\sqrt{\r^\h}\p_tv^\h_{zz}(t)\|_{L^2_\h}^2\\
&{}+CF_{2,\cT}(v^\h(t),v^\h_{zz}(t))+\cC_0h^2(z)\w{t}^{-\frac  {11} 2}+C\log^2\w{t}\bigl(\|\p_t^2v^\h(t)\|_{L^2_\h}^2
+\|\p_t^2v^\h_z(t)\|_{L^2_\h}^2\bigr),
\end{split} \label{PropSBlem3eq5} \eeq which is exactly the same as the Inequality \eqref{Prop5.2eq7}. With this estimate, we only need
to repeat the estimate of \eqref{SAeq2} to show \eqref{5.14}. And
the first inequality of \eqref{5.17} follows from \eqref{1.14b},
\eqref{Propsect6.1eq1} and \eqref{5.14}.

In order to prove the second inequality of \eqref{5.17},  we get, by
applying Lemma \ref{Echanget2x} to (D2INS2D), that \beno
\|\nh^2v_{zz}^\h(t)\|_{L^4_\h}\leq
C\left(\|\sqrt{\r^\h}\p_tv^\h_{zz}(t)\|_{L^4_\h}+\|v^\h(t)\|_{L^8_\h}^2\|\nh
v^\h_{zz}(t)\|_{L^2_\h}+\|f_2(t)\|_{L^4_\h}\right). \eeno Notice
that \beno \|f_2(t)\|_{L^4_\h}&\leq&
\|\r^\h\|_{L^\infty_\h}\bigl(\|\v^\h_z\|_{L^\infty_\h}\|\nh
v^\h_z\|_{L^4_\h}+\|v^\h_{zz}\|_{L^4_\h}\|\nh v^\h\|_{L^\infty_\h}\\
&&+\|\r^\h_z\|_{L^\infty_\h}\bigl(\|\p_tv^\h_z\|_{L^4_\h}+2\|v^\h_z\|_{L^\infty_\h}\|\nh
v^h\|_{L^4_\h}+\|v^\h\|_{L^\infty_\h}\|\nh v^h_z\|_{L^4_\h}\bigr)\\
&&\qquad\qquad\qquad\qquad\qquad\qquad+\|\r^\h_{zz}\|_{L^\infty_\h}\bigl(\|\v^\h_t\|_{L^4_\h}+\|v^\h\|_{L^\infty_\h}\|\nh
v^\h\|_{L^4_\h}\bigr), \eeno which together with the obtained decay
estimates implies
 \beno
  \|f_2(t)\|_{L^4_\h}\leq \cC_0
h(z)\w{t}^{-\frac  9 4}.
\eeno
Hence thanks to the 2-D
interpolation inequality: $\|v^\h\|_{L^8_\h}\lesssim
\|v^\h\|_{L^2_\h}^{\f14}\|\nh v^\h\|_{L^2_\h}^{\f34},$ ensures that
\beno \|\nh^2 v^\h_{zz}(t)\|_{L^4_\h}\leq
\cC_0\bigl(h(z)\w{t}^{-\frac  9 4}+h^{\f12}(z)\w{t}^{-1}\|\nh\p_tv^\h_{zz}(t)\|_{L^2_\h}^{\f12}\bigr).
\eeno Then repeating the proof of \eqref{2.19af} leads to the second
inequality of \eqref{5.17}. The proof of the proposition is
complete.
\end{proof}

\begin{col}\label{Sect6Col1} {\sl Under the assumptions of Theorem
\ref{insslowvar}, $T^\star_2$ given by Relation \eqref{Tstardef2}
is infinite and there holds \eqref{thSBprop1eq1}.}
\end{col}

\begin{proof} Indeed it follows from Theorem 3.1.4 of \cite{BCD} and
the transport equation of (D2INS2D) that \beno
\|\r^\h_{zz}(t)\|_{L^\infty_t(H^2_\h)} & \leq &
C\Bigl(\|\p_z^2\r_0\|_{H^2_\h}+\int_0^t\bigl(\|v^h_z(t')\|_{L^\infty}\|\nh\r^\h_z(t')\|_{H^2_\h}\\
&&{}+\|\nh\r^\h_z(t')\|_{L^\infty}\|\nh v^\h_z(t')\|_{H^1_\h}
+\|v^h_{zz}(t')\|_{L^\infty_\h}\|\nh\r^\h(t')\|_{H^2_\h}\\
&&\qquad{}+\|\nh\r^\h(t')\|_{L^\infty_\h}\|\nh
v^\h_{zz}(t')\|_{H^1_\h}\bigr)dt'\Bigr)\exp\Bigl(C\|\nh
v^\h\|_{L^1_t(H^2_\h)}\Bigr), \eeno which together with
\eqref{S3thm2eq3} and \eqref{thmSAprop1eq1} ensures that
 \beq\label{Sect6Col1eq1} \|\r_{zz}^\h(t)\|_{L^2_\h\cap
 L^\infty_\h}\leq C
 \|\r^\h_{zz}(t)\|_{L^\infty_t(H^2_\h)}\leq
\cC_0\eta h(z)\leq \f12 \quad\mbox{for}\quad t\leq T^\star_2, \eeq
provided that $\eta$ is small enough. This in turn shows that
$T^\star_2$ given by Relation \eqref{Tstardef2} is infinite.
Moreover, it follows from the transport equation of (D2INS2D) that
$$
\longformule{ \|\p_t\r^\h_{zz}(t)\|_{L^2_\h}\leq
\|v^\h(t)\|_{L^\infty_\h}\|\nh\r^\h_{zz}(t)\|_{L^2_\h}+2\|v^\h_z(t)\|_{L^\infty_\h}\|\nh\r^\h_{z}(t)\|_{L^2_\h}
} { {} +\|v^\h_{zz}(t)\|_{L^\infty_\h}\|\nh\r^\h(t)\|_{L^2_\h}\leq
\cC_0\eta h(z)\w{t}^{-\frac 3 2}. }
$$
 Similarly, by taking one more horizontal derivative to the
transport equation of (D2INS2D), we obtain \beno
\|\nh\p_t\r^\h_{zz}(t)\|_{L^2_\h}&\leq&
\|v^\h(t)\|_{L^\infty_\h}\|\nh^2\r^\h_{zz}(t)\|_{L^2_\h}+\|\nh
v^\h(t)\|_{L^\infty_\h}\|\nh\r^\h_{zz}(t)\|_{L^2_\h} \\
&&+2\|\nh
v^\h_z(t)\|_{L^\infty_\h}\|\nh\r^\h_{z}(t)\|_{L^2_\h}+2\|v^\h_{z}(t)\|_{L^\infty_\h}\|\nh^2\r^\h_z(t)\|_{L^2_\h}
\\
&&+\|\nh
v^\h_{zz}(t)\|_{L^2_\h}\|\nh\r^\h(t)\|_{L^\infty_\h}+2\|v^\h_{zz}(t)\|_{L^\infty_\h}\|\nh^2\r^\h(t)\|_{L^2_\h}\\
&\leq& \cC_0\eta h(z)\w{t}^{-\frac 3 2}. \eeno
 This together with \eqref{Sect6Col1eq1} ensures \eqref{thSBprop1eq1}.
\end{proof}

Finally let us turn to the decay estimates of the third derivatives.

\begin{prop}\label{SBlem4} {\sl Under the assumptions of Theorem \ref{insslowvar},
  one has
  \ben \label{PropSBlem4eq1}
\|\nh\p_t v^\h_{zz}(t)\|_{L^2_\h}^2  & \leq &   \cC_0
h^2(z)\w{t}^{-5},\\
 \label{PropSBlem4eq1b}
\int_{t}^\infty \bigl(\|\p_t^2v^\h_{zz}\|^2 _{L^2_\h} + \|\nh^2\p_t
v^\h_{zz}\|^2_{L^2_\h}+\|\nh\partial_t\Pi_{zz}^\h\|_{L^2_\h}^2\bigr)
dt' &\leq & \cC_0
h^2(z)\w{t}^{-5},\\
\label{PropSBlem4eq2}\int_0^t\w{t'}^{5_-}\bigl(\|\p_t^2v^\h_{zz}\|^2
_{L^2_\h} + \|\nh^2\p_t
v^\h_{zz}\|^2_{L^2_\h}+\|\nh\partial_t\Pi_{zz}^\h\|_{L^2_\h}^2\bigr) dt' &\leq &  \cC_0 h^2(z),\\
\label{PropSBlem4eq3} \|\nh^3
v^\h_{zz}(t)\|_{L^2_\h}+\|\nh^2\Pi_{zz}^\h(t)\|_{L^2} &\leq & \cC_0
h(z)\w{t}^{-\f52}\log\w{t}. \een}
\end{prop}

\begin{proof} By virtue  of (D2INS2D), $\p_tv^\h_{zz}$ is a solution of
(LINS2D) with external force $\wt f_2(t)$ given by
\beno \wt
f_2(t)&=&-\r_t^\h\p_tv^\h_{zz}-(\r^\h v^\h)_t\cdot\nh
v^\h_{zz}-2\p_t(\r_z^\h\p_tv^\h_z)-\p_t(\r_{zz}^\h v^\h_t)\\
&&-\p_t\p_z(\r_z^\h v^\h\cdot\nh v^\h)-\p_t(\r_z^\h v^\h_z\cdot\nh
v^\h)-\p_t(\r^\h v^\h_z\cdot\nh v^\h_z)-\p_t(\r^\h v^\h_{zz}\cdot\nh
v^\h), \eeno so that \beno \|\wt f_2(t)\|_{L^2_\h} &\leq
&\|\r^\h\|_{L^\infty}\Bigl(\|v^\h_t\|_{L^4_\h}\|\nh
v^\h_{zz}\|_{L^4_\h}+\|\p_tv^\h_z\|_{L^4_\h}\|\nh
v^\h_z\|_{L^4_\h}+\|\p_tv^\h_{zz}\|_{L^4_\h}\|\nh v^\h\|_{L^4_\h}\\
&&{}
+\| v^\h_z\|_{L^\infty_\h}\|\nh\p_tv^\h_z\|_{L^2_\h}+\|
v^\h_{zz}\|_{L^\infty_\h}\|\nh
v^\h_t\|_{L^2_\h}\Bigr)+\|\r_z^\h\|_{L^\infty}\Bigl(\|\p_t^2v^\h_z\|_{L^2_\h}\\
&&{}
+\|v^\h\|_{L^\infty_\h}\|\nh\p_tv^\h_z\|_{L^2_\h}+2\|v^\h_z\|_{L^\infty_\h}\|\nh
v^\h_t\|_{L^2_\h}+2\| \nh v^\h\|_{L^4_\h}\| \p_tv^\h_z\|_{L^4_\h}\\
&&{} +\|\nh v^\h_z\|_{L^4_\h}\|
v^\h_t\|_{L^4_\h}\Bigr)+\|\r_{zz}^\h\|_{L^\infty}\Bigl(\|\p_t^2v^\h\|_{L^2_\h}
+\|v^\h\|_{L^\infty_\h}\|\nh v^\h_t\|_{L^2_\h}\\
&&{}
+\|\nh
v^\h\|_{L^4_\h}\|v^\h_t\|_{L^4_\h}\Bigr)+\|\r_t^\h\|_{L^\infty_\h}\Bigl(\|\p_tv^\h_{zz}\|_{L^2_\h}+\|
v^\h\|_{L^\infty_\h}\|\nh v^\h_{zz}\|_{L^2_\h}\\
&&{}
+\| v^\h_z\|_{L^\infty_\h}\|\nh v^\h_z\|_{L^2_\h}+\|
v^\h_{zz}\|_{L^\infty_\h}\|\nh v^\h\|_{L^2_\h}\Bigr)\\
&&{}
+\|\p_t\r_z^\h\|_{L^\infty_\h}\Bigl(\|\p_tv^\h_{z}\|_{L^2_\h}+\|
v^\h\|_{L^\infty_\h}\|\nh v^\h_{z}\|_{L^2_\h}+2\|
v^\h_z\|_{L^\infty_\h}\|\nh
v^\h\|_{L^2_\h}\Bigr)\\
&&{}
+\|\p_t\r_{zz}^\h\|_{L^4_\h}\bigl(\| v^\h_t\|_{L^4_\h}+\|
v^\h\|_{L^\infty_\h}\|\nh v^\h\|_{L^4_\h}\bigr).
\eeno
 Then according to  the estimates in the previous sections, we deduce
that \beno \|\wt f(t')\|_{L^2}^2\leq \cC_0
h^2(z)\w{t}^{-7}+C\|\p_t^2v^\h(t')\|_{L^2_\h}^2+C\|\p_t^2v^\h_z(t')\|_{L^2_\h}^2,
\eeno which together with \eqref{propdecayD3eq1} and \eqref{2.22}
 ensures that for any nonnegative $t_0$ and
any time $t\geq t_0$ \beq \label{PropSBlem4eq4} \int_{t_0}^t\|\wt
f_2(t')\|_{L^2}^2dt'\leq \cC_0 h^2(z)\w{t_0}^{-5}. \eeq With
\eqref{PropSBlem4eq4}, we repeat the proofs of \eqref{2.22} and
\eqref{PropSAlem3eq2}  to get (\ref{PropSBlem4eq1} -
\ref{PropSBlem4eq2}). We omit the details here.

Finally in order to prove \eqref{PropSBlem4eq3}, we get, by
differentiate the momentum equation of (D2INS2D) with respect to
$x_j,$ that $$\longformule{
\D_\h\p_jv^\h_{zz}-\p_j\nh\Pi_{zz}^\h=\p_j(\r^\h\p_tv^\h_{zz})+\p_j(\r^\h
v^\h\cdot\nh v^\h_{zz})+2\p_j(\r_z^\h\p_tv^\h_z)+\p_j(\r_{zz}^\h
v^\h_t)}{{} +\p_j\p_z(\r_z^\h v^\h\cdot\nh v^\h)+\p_j(\r_z^\h
v^\h_z\cdot\nh v^\h)+\p_j(\r^\h v^\h_z\cdot\nh v^\h_z)+\p_j(\r^\h
v^\h_{zz}\cdot\nh v^\h).} $$  Applying Lemma \ref{Echanget2x} gives
\beno
&&\|\nh^3v^\h_{zz}\|_{L^2_\h}+\|\nh^2\Pi_{zz}^\h\|_{L^2_\h}\lesssim\|\nh\r^\h\p_tv^\h_{zz}\|_{L^2_\h}+
\|\nh\r_z^\h\p_tv^\h_z\|_{L^2_\h}+
\|\nh\r_{zz}^\h v^\h_t\|_{L^2_\h}\\
&&\qquad+\|\r^\h\|_{L^\infty}\Bigl(\|\nh\p_tv^\h_{zz}\|_{L^2_\h}
+2\|\nh v^\h\|_{L^4_\h}\|\nh
v^\h_{zz}\|_{L^4_\h}+\|v^\h\|_{L^\infty_\h}\|\nh^2v^\h_{zz}\|_{L^2_\h}\\
&&\qquad+\|\nh
v^\h_z\|_{L^4_\h}^2+\|v^\h_z\|_{L^\infty_\h}\|\nh^2v^\h_z\|_{L^2_\h}+\|v^\h_{zz}\|_{L^\infty_\h}\|\nh^2v^\h\|_{L^2_\h}\Bigr)
+\|\r_z^\h\|_{L^\infty_\h}\Bigl(\|\nh \p_tv^\h_z\|_{L^2_\h}
\\
&&\qquad+3\|\nh v^\h_z\|_{L^4_\h}\|\nh
v^\h\|_{L^4_\h}+\|v^\h\|_{L^\infty_\h}\|\nh^2v^\h_z\|_{L^2_\h}+2\|v^\h_z\|_{L^\infty_\h}\|\nh^2v^\h\|_{L^2_\h}\Bigr)\\
&&\qquad+\|\r_{zz}^\h\|_{L^\infty_\h}\Bigl(\|\nh
v^\h_t\|_{L^2_\h}+\|\nh
v^\h\|_{L^4_\h}^2+\|v^\h\|_{L^\infty_\h}\|\nh^2v^\h\|_{L^2_\h}\Bigr)\\
&&\qquad+\|\nh\r^\h\|_{L^\infty_\h}\Bigl(\|v^\h\|_{L^\infty_\h}\|\nh
v^\h_{zz}\|_{L^2_\h}+\|v^\h_z\|_{L^\infty_\h}\|\nh
v^\h_z\|_{L^2_\h}+\|v^\h_{zz}\|_{L^\infty_\h}\|\nh v^\h\|_{L^2_\h}\Bigr)\\
&&\qquad+2\|\nh \r_z^\h\|_{L^\infty_\h}\|v^\h_z\|_{L^\infty_\h}\|\nh
v^\h\|_{L^2_\h}+\|\nh \r_{zz}^\h\|_{L^4_\h}\|v^\h\|_{L^4_\h}\|\nh
v^\h\|_{L^2_\h}, \eeno from which and the previous decay estimates
and \eqref{thSBprop1eq1}, we infer
$$
\|\nh^3v^\h_{zz}\|_{L^2_\h}+\|\nh^2\Pi_{zz}^\h\|_{L^2_\h}\lesssim
\|\nh\r^\h\p_tv^\h_{zz}\|_{L^2_\h}+\|\nh\r_z^\h\p_tv^\h_z\|_{L^2_\h}+\|\nh\r_{zz}^\h
v^\h_t\|_{L^2_\h}+\cC_0h(z)\w{t}^{-\frac 5 2}.
$$
Then using \eqref{S3thm2eq3}, \eqref{thmSAprop1eq1} and
\eqref{thSBprop1eq1}, we deduce \eqref{PropSBlem4eq3} by  repeating
the proof of \eqref{S7eq6demoeq4}. This completes the proof of the
proposition.
\end{proof}

By summing up Propositions \ref{Propsect6.1}, \ref{SBlem3} and
\ref{SBlem4}, and Corollary \ref{Sect6Col1}, we complete the proof
of Theorem \ref{SBprop1}.

\setcounter{equation}{0}
\section {Estimates of $v^\h$ in terms of anisotropic Besov  norms} \label{Besovvh}

By summing up Theorems \ref{S3thm2}, \ref{SAprop1} and
\ref{SBprop1}, we conclude the following theorem concerning the
decay estimates of solutions to \eqref{INS2dParameterb}.

\begin{thm}
\label {S7thm1} {\sl Under the assumptions of Theorem
\ref{insslowvar}, System \eqref{INS2dParameterb} has a unique global
solution $(\r^\h, v^\h,\nh \Pi^\h)$ so that there hold \beq
\label{S7eq0.2}
\begin{split} &\sum_{\ell=0}^2\biggl(\w{t}\|\p_z^\ell
v^\h(t,\cdot,z)\|_{L^2_\h}+\w{t}^{\f32}\bigl(\|\nh \p_z^\ell
v^\h(t,\cdot,z)\|_{L^2_\h}+\|\p_z^\ell v^\h(t,\cdot,z)\|_{L^\infty_\h}\bigr)\\
&\qquad+ \w{t}^{2}\bigl(\|\p_z^\ell v^\h_t(t,\cdot,z)\|_{L^2_\h}+
\|\nh^2\p_z^\ell v^\h(t,\cdot,z)\|_{L^2_\h}+\|\nh \p_z^\ell
\Pi^\h(t,\cdot,z)\|_{L^2_\h}\bigr)\\
&\qquad+\w{t}^{\f52}\log^{-1}\w{t}\bigl(\|\nh^3\p_z^\ell
v^\h(t,\cdot,z)\|_{L^2_\h}+\|\nh^2\p_z^\ell
\Pi^\h(t,\cdot,z)\|_{L^2_\h}\bigr)\\
&\qquad+ \w{t}^{\f52} \|\nh\p_z^\ell
v^\h_t(t,\cdot,z)\|_{L^2_\h}+\w{t}^{2}\log^{-\f12}\w{t}\|\nh\p_z^\ell
v^\h(t,\cdot,z)\|_{L^\infty_\h}\biggr)\leq \cC_0h(z),
\end{split} \eeq and \beq \label{S7eq0.2as} \sum_{\ell=0}^2\int_0^t\w{t'}^{5_-}\bigl(\|\p_t^2\p_z^\ell v^\h(t')\|^2
_{L^2_\h} + \|\nh^2\p_z^\ell
v^\h_{t}(t')\|^2_{L^2_\h}+\|\nh\partial_z^\ell
\p_t\Pi^\h(t')\|_{L^2_\h}^2\bigr) dt' \leq   \cC_0 h^2(z). \eeq We
also have  \beq \label {S7eq0.3}
\begin{split}
\|&\varrho^\h(t,\cdot,z)\|_{H^4_\h}+\|\r^\h_z(t,\cdot,z)\|_{H^3_\h}+\|\r^\h_{zz}(t,\cdot,z)\|_{H^2_\h}\\
&{}\quad+\w{t}^{\f32}\bigl(\|\r^\h_t(t,\cdot,z)\|_{H^3_\h}+\|\p_z\r^\h_t(t,\cdot,z)\|_{H^2_\h}+\|\p_{z}^2\r^\h_t(t,\cdot,z)\|_{H^1_\h}\bigr)\leq
\cC_0\eta h(z).
\end{split}
\eeq
}
 \end{thm}

Now we shall transform the above decay estimates of $v^\h$ to the
$L^1$ or $L^2$ in time estimate of the Besov norms to $v^\h.$
 For the
convenience of the readers, we recall the following anisotropic
Bernstein type lemma from \cite{CZ1, Pa02}:

\begin{lem}
\label{lemBern} {\sl Let $\cB_{\h}$ (resp.~$\cB_{\v}$) a ball
of~$\R^2_{\h}$ (resp.~$\R_{\v}$), and~$\cC_{\h}$ (resp.~$\cC_{\v}$)
a ring of~$\R^2_{\h}$ (resp.~$\R_{\v}$); let~$1\leq p_2\leq p_1\leq
\infty$ and ~$1\leq q_2\leq q_1\leq \infty.$ Then there holds:

\smallbreak\noindent If the support of~$\wh a$ is included
in~$2^k\cB_{\h}$, then
\[
\|\partial_{{\rm h}}^\alpha a\|_{L^{p_1}_{\rm h}(L^{q_1}_{\rm v})}
\lesssim 2^{k\left(|\al|+2\left(\f1{p_2}-\f1{p_1}\right)\right)}
\|a\|_{L^{p_2}_{\rm h}(L^{q_1}_{\rm v})}.
\]
If the support of~$\wh a$ is included in~$2^\ell\cB_{\v}$, then
\[
\|\partial_{z}^\beta a\|_{L^{p_1}_{\rm h}(L^{q_1}_{\rm v})} \lesssim
2^{\ell\left(\beta+(\f1{q_2}-\f1{q_1})\right)} \| a\|_{L^{p_1}_{\rm
h}(L^{q_2}_{\rm v})}.
\]
If the support of~$\wh a$ is included in~$2^k\cC_{\h}$, then
\[
\|a\|_{L^{p_1}_{\rm h}(L^{q_1}_{\rm v})} \lesssim
2^{-kN}\sup_{|\al|=N} \|\partial_{{\rm h}}^\al a\|_{L^{p_1}_{\rm
h}(L^{q_1}_{\rm v})}.
\]
If the support of~$\wh a$ is included in~$2^\ell\cC_{\v}$, then
\[
\|a\|_{L^{p_1}_{\rm h}(L^{q_1}_{\rm v})} \lesssim 2^{-\ell N}
\|\partial_{z}^N a\|_{L^{p_1}_{\rm h}(L^{q_1}_{\rm v})}.
\]
}
\end{lem}

In view of Definition \ref{anibesov},  as a corollary of Lemma
\ref{lemBern}, we have the following inequality, if~$1\leq p_2\leq
p_1$,
 \beq
 \label{S0eq2}
\|a\|_{\cB^{s_1-2\left(\frac 1{p_2}- \frac
1{p_1}\right),s_2-\left(\frac 1 {p_2}-\frac 1 {p_1}\right)}_{p_1}}
\lesssim \|a\|_{\cB^{s_1,s_2}_{p_2}}.
 \eeq

 To consider the
product of a distribution in the isotropic Soblev space with a
distribution in the anisotropic Besov space, we  need  the following
interpolation inequalities:

\begin{lem}
\label {S0lem3}
{\sl
We have the following interpolation inequality for the $L^\infty$ norm.
\beq
 \label  {interpolBesovanisolinfty}
 \|f\|_{L^\infty} \lesssim \|f\|_{\cB^{-\frac 12+\frac 2p,\frac 1p}_p}^{\frac 23} \|\nabla_\h f \|_{\cB^{\frac 2 p,\frac1 p}_p}^{\frac 13}.
 \eeq
 Moreover, let $1<q\leq p\leq \infty,$ $-2/q+2/p<s_1<1-\bigl(2/q-2/p\bigr)$ and
$s_2$ in~$ ]0,1[,$  one~has
 \beq
 \label{S0eq3}
\begin{split}
&\|f\|_{\cB^{s_1,s_2}_p}\leq
C_{p,q}\|f\|_{L^p_\v(L^q_\h)}^{\bigl(1-s_1-2\bigl(\f1q-\f1p\bigr)\bigr)(1-s_2)}\|\p_zf\|_{L^p_\v(L^q_\h)}^{\bigl(1-s_1-2\bigl(\f1q-\f1p\bigr)\bigr)s_2}\\
&\qquad\qquad\qquad\qquad\qquad\qquad{}\times \|\nh
f\|_{L^p_\v(L^q_\h)}^{\bigl(s_1+2\bigl(\f1q-\f1p\bigr)\bigr)(1-s_2)}\|\nh
\p_zf\|_{L^p_\v(L^q_\h)}^{\bigl(s_1+2\bigl(\f1q-\f1p\bigr)\bigr)s_2}.
\end{split}
\eeq }
\end{lem}
\begin{proof}  In order to prove the first inequality, let us  write according to Lemma \ref{lemBern} that
\beno
\|f\|_{L^\infty} &\leq  & \sum_{(k,\ell)\in\Z^2} \|\D^\h_k \D^\v_\ell f\|_{L^\infty}\\
&\lesssim  & 2^{\frac K 2} \sum_{\substack {k< K \\ \ell\in\Z}}
2^{k\bigl(-\frac 12 +\frac 2 p\bigr) } 2^{\frac \ell p}  \|\D^\h_k
\D^\v_\ell f\|_{L^p} +
2^{-K} \sum_{\substack {k\geq  K \\ \ell\in\Z}} 2^{k\frac 2 p } 2^{\frac \ell p}  \|\D^\h_k \D^\v_\ell \nabla_\h f\|_{L^p}  \\
& \lesssim & 2^{\frac K 2} \|f\|_{\cB^{-\frac 12+\frac 2p,\frac
1p}_p}+ 2^{-K}  \|\nabla_\h f \|_{\cB^{\frac 2 p,\frac1 p}_p}.
 \eeno
 The appropriate choice of~$K$ ensures \eqref {interpolBesovanisolinfty}. Let us prove the second one.
According Definition \ref{anibesov}, we have \beno
 \|f\|_{\cB^{s_1,s_2}_p}=\sum_{k,\ell\in\Z^2}2^{ks_1}2^{\ell
 s_2}\|\D_k^\h\D_\ell^\v f\|_{L^p}. \eeno
For any integers $K, L_1,$ which will be chosen late, we get, by
applying Lemma \ref{lemBern}, that \beno \sum_{{k\leq K, \ell\leq
L_1}}2^{ks_1}2^{\ell
 s_2}\|\D_k^\h\D_\ell^\v f\|_{L^p}& \lesssim & \sum_{{k\leq K, \ell\leq
 L_1}}2^{k\bigl(
 s_1+\f2q-\f2p\bigr)}2^{\ell s_2}\|\D_k^\h\D_\ell^\v f\|_{L^p_\v(L^q_\h)}\\
& \lesssim & 2^{K\bigl(
 s_1+\f2q-\f2p\bigr)}2^{L_1
 s_2}\|f\|_{L^p_\v(L^q_\h)},
 \eeno
 by using the fact that $s_1$ is greater than~$ -2/q+2/p$ and $s_2$ is positive.

 Similarly since $s_1$ is greater than~$-2/q+2/p$ and $s_2$ is less than~$1,$ one has
\beno \sum_{{k\leq K, \ell> L_1}}2^{ks_1}2^{\ell
 s_2}\|\D_k^\h\D_\ell^\v f\|_{L^p}& \lesssim & \sum_{{k\leq K,
 \ell>
 L_1}}2^{k\bigl(
 s_1+\f2q-\f2p\bigr)}2^{-\ell(1-
 s_2)}\|\D_k^\h\D_\ell^\v \p_zf\|_{L^p_\v(L^q_\h)}\\
& \lesssim & 2^{K\bigl(
 s_1+\f2q-\f2p\bigr)}2^{-L_1(1-s_2)}\|\p_zf\|_{L^p_\v(L^q_\h)}.
  \eeno

Along the same line,  since $s_1$ is less than~$\ds
1-\bigl(2/q-2/p\bigr)$ and $ s_2$ is in~$ ]0,1[,$  for some
integer~$L_2$ to be chosen hereafter, we write \beno \sum_{{k> K,
\ell\leq L_2}}2^{ks_1}2^{\ell
 s_2}\|\D_k^\h\D_\ell^\v f\|_{L^p}
 & \lesssim &
 \sum_{{k> K, \ell\leq
 L_2}}2^{-k\bigl(1-s_1-\f2q+\f2p\bigr)}2^{\ell
 s_2}\|\D_k^\h\D_\ell^\v \nh f\|_{L^p_\v(L^q_\h)}\\
 &\lesssim &
  2^{-K\bigl(1-s_1-\f2q+\f2p\bigr)}2^{L_2
 s_2}\|\nh f\|_{L^p_\v(L^q_\h)},
  \eeno
and \beno \sum_{{k> K, \ell> L_2}}2^{ks_1}2^{\ell
 s_2}\|\D_k^\h\D_\ell^\v f\|_{L^p}
 & \lesssim &
 \sum_{{k> K, \ell>
 L_2}}2^{-k\bigl(1-s_1-\f2q+\f2p\bigr)}2^{-\ell(1-
 s_2)}\|\D_k^\h\D_\ell^\v \nh\p_zf\|_{L^p_\v(L^q_\h)}\\
 &\lesssim &
 2^{-K\bigl(1-s_1-\f2q+\f2p\bigr)}2^{-L_2(1-
 s_2)}\|\nh\p_zf\|_{L^p_\v(L^q_\h)}.
  \eeno

 As a consequence, we obtain
 \beno
\|f\|_{\cB^{s_1,s_2}_p} &\lesssim &
2^{K\bigl(s_1+\f2q-\f2p\bigr)}2^{L_1
 s_2}\bigl(\|f\|_{L^p_\v(L^q_\h)}+2^{-L_1}\|\p_zf\|_{L^p_\v(L^q_\h)}\bigr)\\
&&{}+2^{-K\bigl(1-s_1-\f2q+\f2p\bigr)}2^{L_2
 s_2}\bigl(\|\nh f\|_{L^p_\v(L^q_\h)}+2^{-L_2}\|\nh\p_zf\|_{L^p_\v(L^q_\h)}\bigr).
 \eeno
 Taking $L_1, L_2$ in the above inequality so that
$$
 2^{L_1}\sim
 \f{\|\p_zf\|_{L^p_\v(L^q_\h)}}{\|f\|_{L^p_\v(L^q_\h)}}\andf
 2^{L_2}\sim \f{\|\nh\p_zf\|_{L^p_\v(L^q_\h)}}{\|\nh
 f\|_{L^p_\v(L^q_\h)}}\,\virgp
 $$
 we get
 \beno
\|f\|_{\cB^{s_1,s_2}_p}
&\lesssim &
2^{K\bigl(s_1+\f2q-\f2p\bigr)}\|f\|_{L^p_\v(L^q_\h)}^{1-s_2}\|\p_zf\|_{L^p_\v(L^q_\h)}^{s_2}\\
&&{}+2^{-K\bigl(1-s_1-\f2q+\f2p\bigr)}\|\nh
f\|_{L^p_\v(L^q_\h)}^{1-s_2}\|\nh\p_zf\|_{L^p_\v(L^q_\h)}^{s_2}.
 \eeno
Taking $K$ in the above inequality so that
$$
 2^K\sim \f{\|\nh
f\|_{L^p_\v(L^q_\h)}^{1-s_2}
\|\nh\p_zf\|_{L^p_\v(L^q_\h)}^{s_2}}{\|f\|_{L^p_\v(L^q_\h)}^{1-s_2}\|\p_zf\|_{L^p_\v(L^q_\h)}^{s_2}}
$$
gives rise to \eqref{S0eq3}. This finishes the proof of Lemma
\ref{S0lem3}.
\end{proof}

\begin{lem}\label{S7lem4.5}
{\sl  Let $p$ be in~$ ]2,\infty[$,   $s_1$ and~$s_2$ in~$ ]0,1[$ and  $s'$ in~$\left]2/p-1,2/p\right[.$ Let
 $(\r^\h,v^\h,\nh\Pi^\h)$ be the global unique solution of \eqref{INS2dParameterb}.
Then under the assumptions of  Theorem \ref{insslowvar}, we have
\ben &&\label{S7eq12qw}
\|\varrho^\h\|_{L^\infty(\R^+;\cB^{s_1,s_2}_2\cap
\cB^{2+s_1,s_2}_2)}+\|\r^\h_z\|_{L^\infty(\R^+;\cB^{s_1,s_2}_2\cap
\cB^{1+s_1,s_2}_2)}\leq
\cC_0\eta,\\
 &&\label{S7eq12qwp}
\|\r^\h_t(t)\|_{\cB^{s_1,s_2}_2\cap
\cB^{1+s_1,s_2}_2}+\|\p_z\r^\h_t(t)\|_{\cB^{s_1,s_2}_2}\leq
\cC_0\eta\w{t}^{-\frac 3 2},\\
&&\label{S7eq11} \|v^\h_t(t)\|_{\cB^{s_1,s_2}_2}+\|\p_z
v^\h_t(t)\|_{\cB^{s_1,s_2}_2}\leq
\cC_0\w{t}^{-\bigl(2+\f{s_1}2\bigr)},\\
 &&\label{S7eq8}
\|v^\h(t)\|_{\cB^{s_1,s_2}_p}+\|\p_zv^\h(t)\|_{\cB^{s_1,s_2}_p}\leq
\cC_0 \w{t}^{-\bigl(\f32+\f{s_1}2-\f1p\bigr)},\\
&&\label{S7eq9} \| v^\h(t)\|_{\cB^{1+s_1,s_2}_p}+\|\p_z
v^\h(t)\|_{\cB^{1+s_1,s_2}_p}\leq
\cC_0\w{t}^{-\bigl(2+\f{s_1}2-\f1p\bigr)}\log^{\bigl(1-\f2p\bigr)s_1}\w{t} \andf\\
&& \label{S7eq10}
\sum_{\ell=0}^1\Bigl(\|\p_z^\ell\uh(t)\|_{\cB^{2+s',s_2}_p}+\|\nh\p_z^\ell
\Pi^\h(t)\|_{\cB^{s',s_2}_p}\Bigr)\leq
\cC_0\w{t}^{-\bigl(\f52+\f{s'}2-\f1p\bigr)}\log^{\bigl(1+s'-\f2p\bigr)}\w{t}.
\een }
\end{lem}

\begin{proof}
The  Inequalities (\ref{S7eq12qw}-\ref{S7eq11}) follow directly from
Lemma  \ref{S0lem3} and from Inequalities \eqref{S7eq0.3} and
\eqref{S7eq0.2}. Whereas note that in two space dimension, there
holds \beq \label{4.9} \forall p\in ]2,\infty[\,,\
\|f\|_{L^p_\h}\lesssim \|f\|_{L^2_\h}^{\f2p}\|\nh
f\|_{L^2_\h}^{1-\f2p}. \eeq Then in view of \eqref{S7eq0.2}, we
infer for $\ell$ in~$ \{0,1,2\},$ \beno \|\p_z^\ell\uh(t)\|_{L^p} &
\lesssim &
\|\p_z^\ell\uh(t)\|_{L^p_\v(L^2_\h)}^{\f2p}\|\nh\p_z^\ell\uh(t)\|_{L^p_\v(L^2_\h)}^{1-\f2p}\leq
\cC_0\w{t}^{-\bigl(\f32-\f1p\bigr)},\\
\|\nh\p_z^\ell\uh(t)\|_{L^p}
& \lesssim&
\|\nh\p_z^\ell\uh(t)\|_{L^p_\v(L^2_\h)}^{\f2p}\|\nh^2\p_z^\ell\uh(t)\|_{L^p_\v(L^2_\h)}^{1-\f2p}\leq
\cC_0\w{t}^{-\bigl(2-\f1p\bigr)} \andf\\
 \|\nh^2\p_z^\ell\uh(t)\|_{L^p}
 &\lesssim&
\|\nh^2\p_z^\ell\uh(t)\|_{L^p_\v(L^2_\h)}^{\f2p}\|\nh^3\p_z^\ell\uh(t)\|_{L^p_\v(L^2_\h)}^{1-\f2p}\leq
\cC_0\w{t}^{-\bigl(\f52-\f1p\bigr)}\log^{1-\f2p}\w{t}.
\eeno
 Hence, by virtue of Lemma \ref{S0lem3}, we
infer for $\ell$ in~$ \{0,1\}$ and $s_1$ ans~$s_2$ in~$ ]0,1[,$
\beno
\|\p_z^\ell\uh(t)\|_{\cB^{s_1,s_2}_p}
& \leq &
\|\p_z^\ell\uh(t)\|_{L^p}^{(1-s_1)(1-s_2)}\|\p_z^{\ell+1}\uh(t)\|_{L^p}^{(1-s_1)s_2}\\
&& \qquad\qquad\qquad\qquad{}
\times \|\nh \p_z^\ell\uh(t)\|_{L^p}^{s_1(1-s_2)}\|\nh\p_z^{\ell+1}\uh(t)\|_{L^p}^{s_1s_2}\\
& \leq &  \cC_0 \w{t}^{-\left(\f32-\f1p\right)(1-s_1)} \times \w
t^{-\left(2-\frac 1 p \right) s_1}.
 \eeno
 This proves \eqref{S7eq8}. A similar argument  yields \eqref{S7eq9}.

Since $s'$ is in~$\left]2/p-1,2/p\right[,$ by applying Lemma
\ref{S0lem3} and \eqref{S7eq0.2},  we get for $\ell$ in~$ \{0,1\},$
\beno \|\nh^2 \p_z^\ell\uh(t)\|_{\cB^{s',s_2}_p} & \leq &
\|\nh^2\p_z^\ell\uh(t)\|_{L^p_\v(L^2_\h)}^{\bigl(\f2p-s'\bigr)(1-s_2)}\|\nh^2\p_z^{\ell+1}\uh(t)\|_{L^p_\v(L^2_\h)}^{\bigl(\f2p-s'\bigr)s_2}\\
&&\qquad\qquad{}
\times\|\nh^3\p_z^\ell
\uh(t)\|_{L^p_\v(L^2_\h)}^{\bigl(1+s'-\f2p\bigr)(1-s_2)}\|\nh^3\p_z^{\ell+1}\uh(t)\|_{L^p_\v(L^2_\h)}^{\bigl(1+s'-\f2p\bigr)s_2}\\
& \leq &
\cC_0\w{t}^{-\bigl(\f52+\f12\bigl(s'-\f2p\bigr)\bigr)}\log^{\bigl(1+s'-\f2p\bigr)}\w{t}.
\eeno The same estimate holds for $\nh\p_z^\ell\Pi^\h.$ This leads
to \eqref{S7eq10}, and  the proof of the lemma is complete.
\end{proof}

\begin{rmk}
{\sl It is easy to observe that $a^\h$ satisfies the same estimate
as \eqref{S7eq12qw}, that is \beq \label{S7eq12qwqp}
\|a^\h\|_{L^\infty(\R^+; \cB^{s_1,s_2}_2\cap
\cB^{2+s_1,s_2}_2)}+\|\p_z a^\h\|_{L^\infty(\R^+;
\cB^{s_1,s_2}_2\cap \cB^{1+s_1,s_2}_2)}\leq \cC_0\eta,
 \eeq for any $s_1$ and~$s_2$ in~$ ]0,1[.$
 }
\end{rmk}

Let us now turn to the proof of Proposition
\ref{consequenceSection345}.

\begin{proof}[Proof of Proposition \ref{consequenceSection345}]
It follows from Lemma \ref{S7lem4.5} and interpolation inequality in
Besov spaces that for $\ell$ in~$ \{0,1\}$
\ben
\|\p_zv^\h(t)\|_{\cB^{\f34,\f34}_2}
& \leq&
\cC_0\w{t}^{-\frac   {11} 8},\nonumber\\
\label{S7eq12}
\|\p_z^\ell v^\h(t)\|_{\cB^{1,\f12}_2}
& \lesssim &
\|\p_z^\ell v^\h(t)\|_{\cB^{\f12,\f12}_2}^{\f12} \|\nh\p_z^\ell
v^\h(t)\|_{\cB^{\f12,\f12}_2}^{\f12}\leq
\cC_0\w{t}^{-\frac 3 2} \andf\\
\|\nh v^\h(t)\|_{\cB^{1,\f12}_2} & \lesssim & \|\nh
v^\h(t)\|_{\cB^{\f12,\f12}_2}^{\f12} \|\nh^2
v^\h(t)\|_{\cB^{\f12,\f12}_2}^{\f12}\leq
\cC_0\w{t}^{-2}\log^{\f14}\w{t}. \nonumber
\een
This implies
\beq
\label{consequenceSection345eq1}
\begin{split}
\|v^\h\|_{L^2(\R^+;\cB^{1,\f12}_2)}+&\|\p_zv^\h\|_{L^2(\R^+;\cB^{1,\f12}_2)}+\|\p_zv^\h\|_{L^1(\R^+;\cB^{1,\f12}_2)}\\
&+\|\p_zv^\h\|_{L^1(\R^+;\cB^{\f34,\f34}_2)}+\|\nh
v^\h\|_{L^1(\R^+;\cB^{1,\f12}_2)}\leq \cC_0.\end{split} \eeq

It remains to handle $\|\p_z\Pi^\h\|_{L^1(\R^+;\cB^{\f12,
\f12}_2)}.$ Indeed it is easy to observe from \eqref{INS2dParameter}
that \beq\label{consequenceSection345eq2} \D_\h
\Pi^\h=-\dive_\h\bigl(a^\h(\nh \Pi^\h-\D_\h
\uh)\bigr)-\dive_\h\dive_\h(\uh\otimes\uh), \eeq from which, we
deduce from  the law of  product \eqref{lawofproductanisobasic} that
$$
\longformule{
\|\Pi^\h\|_{L^1(\R^+;\cB^{\f12,\f12}_2)}\leq C\Bigl(\|\uh\otimes
\uh\|_{L^1\R^+;(\cB^{\f12,\f12}_2)}
}
{ {}
+\|a^\h\|_{L^\infty(\R^+;\cB^{1,\f12}_2)}\bigl(\|
\Pi^\h\|_{L^1(\R^+;\cB^{\f12,\f12}_2)}+\|\nh
\uh\|_{L^1(\R^+;\cB^{\f12,\f12}_2)}\bigr) \Bigr).
}
$$
 Whereas it follows from \eqref{S7eq12qwqp} that
 \beq
\label{S7eq12qwqpa} \|a^\h\|_{L^\infty(\R^+;
\cB^{1,\f12}_2)}+\|\p_za^\h\|_{L^\infty(\R^+; \cB^{1,\f12}_2)}\leq
\cC_0\eta. \eeq When we take $\eta$ so small that $C\cC_0\eta\leq
\f12,$ \eqref{S7eq12qwqpa} implies \beq\label{S5eq10}
\begin{split} \| \Pi^\h\|_{L^1(\R^+;\cB^{\f12,\f12}_2)}\leq &C\Bigl(
\|a^\h\|_{L^\infty(\R^+;\cB^{1,\f12}_2)}\|
\nh\uh\|_{L^1(\R^+;\cB^{\f12,\f12}_2)}\\
&\qquad\qquad\qquad+\|\uh\|_{L^\infty_t(\cB^{\f12,\f12}_2)}
\|\uh\|_{L^1(\cB^{1,\f12}_2)}\Bigr)\leq \cC_0.
\end{split}
\eeq Similar to the proof of \eqref{S5eq10}, we also have
$$
\longformule{\!\!\!\| \Pi^\h_z\|_{L^1(\R^+;\cB^{\f12,\f12}_2)}\leq
C\Bigl(\|\p_z a^\h\|_{L^\infty(\R^+;\cB^{1,\f12}_2)}\bigl(\|
\Pi^\h\|_{L^1(\R^+;\cB^{\f12,\f12}_2)}+\|
\nh\uh\|_{L^1(\R^+;\cB^{\f12,\f12}_2)}\bigr)}{{} +
\|a^\h\|_{L^\infty(\R^+;\cB^{1,\f12}_2)}\|
\nh\p_z\uh\|_{L^1(\R^+;\cB^{\f12,\f12}_2)}+\|\p_z\uh\|_{L^\infty_t(\cB^{\f12,\f12}_2)}
\|\uh\|_{L^1(\cB^{1,\f12}_2)}\Bigr).}
$$
Hence by virtue of Lemma \ref{S7lem4.5}, \eqref{S7eq12qwqpa} and
\eqref{S5eq10}, we infer
 \beno \|
\Pi^\h_z\|_{L^1(\R^+;\cB^{\f12,\f12}_2)}\leq \cC_0. \eeno Together
with \eqref{consequenceSection345eq1}, we  complete the proof of
Proposition \ref{consequenceSection345}.
\end{proof}

\setcounter{equation}{0}
\section {The equation on $w_\e$ and estimates of some error terms}
\label {estimw}

 The purpose of this section is to study  the
equation that determines the correction term in~$\uapp$. Let us
recall it. \beq \label{eqonwas} \left \{
\begin{array}{c}
\partial_t w^\h_\e  -\D_\e w^\h_\e    =  -\nabla_\h \Pi^1_\e \\
\partial_t w^3_\e  -\D_\e w^3_\e    =  -\e^2\partial_z  \Pi^1_\e  +\partial_z   \Pi^\h_L\\
\dive w_\e =0\andf {w_\e}_{|t=0} = 0.
\end{array}
\right. \with  \Pi^\h_L \eqdefa -\D_\h^{-1} \dive _\h \partial_t
\bigl( \varrho^\h  v^\h   \bigr). \eeq We have the following
proposition.
\begin{prop}
\label {estimwdetaila} {\sl  Let~$(w_\e,\Pi^1_\e)$ be the unique
solution of the above system, then we have
$$
\longformule{ \|(\e w^\h_\e ,w^3_\e )\|_{ L^2(\R^+;\cB^{1,\frac
12}_2)} +\|\na_\e(\e w^\h_\e ,w^3_\e )\|_{ L^2(\R^+;\cB^{0,\frac
12}_2)} +\|\nabla _\e  (\e w^\h_\e ,w^3_\e
)\|_{L^1(\R^+;\cB^{1,\frac 12}_{2})}   } { {}  \leq
C\|\varrho_0 v^\h_0\|_{\cB^{-1,\frac 32}_2}
 + C\|\varrho^\h  v^\h\|_{L^2(\R^+;\cB^{0,\frac 32}_2)\cap L^1(\R^+;\cB^{1,\frac 32}_2)
 },
}
$$ and
$$ \e\|(\e w^\h_\e ,w^3_\e )\|_{ L^4(\R^+;\cB^{\f12,\frac
12}_2)}  \leq C\|\varrho_0 v^\h_0\|_{\cB^{0,\frac 12}_2}
 + C\|\varrho^\h  v^\h\|_{L^4(\R^+;\cB^{\f12,\frac 12}_2)}.
 $$ Moreover for any positive~$\al$  less than~$1$, we have
$$
\bigl\| \D_\e (\e w^\h_\e  ,w^3_\e )\|_{L^1(\R^+;\cB^{\al,\frac
12})} +\bigl\|\e\nabla_\e\Pi^1_\e \bigr\|_{L^1(\R^+;\cB^{\al,\frac
12})} \leq C_\al\bigl\|\partial_t ( \varrho^\h  v^\h)\|_{L^1(\R^+; \cB_2^{-1+\al,\frac 32})}.
 $$}
\end{prop}

\begin{proof}
Let us first compute~$\Pi^1_\e $. Applying the divergence operator
to the System \eqref{eqonwas}   gives \beq \label {eqpressurew}
-\D_\e\Pi^1_\e +\partial_z ^2\Pi^\h_L=0. \eeq
 This together with \eqref{eqonwas} gives
  \beq
  \label {estimwdetaildemoeq-1}
\begin{split}
w^\h_\e  & =  \int_0^t e^{(t-t')\D_\e} \nabla_\h \D_\e^{-1} \partial_z ^2 \D_\h^{-1} \dive_\h  \partial_t \bigl( \varrho^\h  v^\h   \bigr)(t') dt'\andf\\
w^3_\e  & =  -\int_0^t e^{(t-t')\D_\e}  \partial_z   \D_\e^{-1}
\dive_\h  \partial_t \bigl( \varrho^\h  v^\h   \bigr)(t') dt'.
\end{split}
\eeq
By integration by parts, we get
\beq
\label {estimwdetaildemoeq-0.5}
\begin{split}
w^\h_\e  & =  \nabla_\h \D_\e^{-1} \partial_z ^2 \D_\h^{-1} \dive_\h
\bigl( \varrho^\h  v^\h   \bigr)(t)
-   e^{t \D_\e} \nabla_\h \D_\e^{-1} \partial_z ^2 \D_\h^{-1} \dive_\h  (\varrho_0 v^\h_0) \\
&  \qquad\qquad\qquad\qquad\qquad{} +\int_0^t e^{(t-t')\D_\e} \nabla_\h \partial_z ^2 \D_\h^{-1} \dive_\h \bigl( \varrho^\h  v^\h   \bigr)(t') dt'\andf\\
w^3_\e  & =  - \partial_z   \D_\e^{-1}  \dive_\h   \bigl( \varrho^\h
v^\h   \bigr)(t) +e^{t\D_\e} \partial_z   \D_\e^{-1}  \dive_\h
(\varrho_0
 v^\h_0 ) \\
&\qquad\qquad\qquad\qquad\qquad\qquad\qquad\quad{} - \int_0^t
e^{(t-t')\D_\e}  \partial_z    \dive_\h \bigl( \varrho^\h  v^\h
\bigr)(t') dt'.
\end{split}
\eeq
Written in term of Fourier transform with the notation~$\xi=(\xi_\h,\zeta)$ and using the fact that~$\dive_\h v^\h = 0$,
\beq
\label {estimwdetaildemoeq0}
\begin{split}
& \bigl| \bigl(\e\wh w^\h_\e (t,\xi),\wh w^3_\e (t,\xi)\bigr)\bigr|
\leq |\xi_\h|^{-1} \bigl|\cF\bigl(\partial_z
(\varrho^\h v^\h)\bigr)(t,\xi)\bigr|
\\
&\qquad\qquad\qquad\qquad{}
+ e^{-t(|\xi_\h^2+\e^2\zeta  ^2)} |\xi_\h|^{-1} \bigl|\cF\bigl(\partial_z (\varrho_0  v_0^\h)\bigr)(\xi)\bigr|\\
&\qquad\qquad\qquad\qquad\qquad{}+\int_0^t e^{-(t-t')(|\xi_\h^2+\e^2\zeta  ^2)} (\e|\zeta  |+|\xi_\h|) |\zeta|\,
\bigl|\cF\bigl( \varrho^\h  v^\h   \bigr)(t',\xi)\bigr| \,dt'\,.
\end{split}
\eeq
Applying the cutoff operator in the frequency space in the horizontal and the vertical directions gives, for any~$r$ in~$[1,2]$,
$$
\displaylines{ 2^{k \frac 2 r +\frac \ell 2 } \|\D_k^\h\D_\ell^\v
(\e w^\h_\e ,w^3_\e )(t)\|_{L^2} \leq C 2^{k \left( \frac 2 r
-1\right)+\frac {3\ell} 2 } \|\D_k^\h\D_\ell^\v
(\varrho^\h v^\h)(t)\|_{L^2} +\cW_{k,\ell}(t) \with\cr
 \cW_{k,\ell}(t) \eqdefa
 e^{-ct(2^{2k}+\e^22^{2\ell})} 2^{k \left( \frac 2 r -1\right)+\frac {3\ell} 2 } \|\D_k^\h\D_\ell^\v (\varrho_0 v_0^\h)\|_{L^2}\cr
 {}
 + \int_0^t   e^{-c(t-t') (2^{2k}+\e^22^{2\ell})}
 (2^k+\e2^\ell) 2^{k \frac 2 r +\frac{3\ell} 2 } \|\D_k^\h\D_\ell^\v (\varrho^\h v^\h)(t')\|_{L^2}dt'.
 }
$$
As~$r$ is in~$[1,2]$, we get by convolution inequality, \beno \|
\cW_{k,\ell}\|_{L^r(\R^+;L^2)}  & \leq & C  2^{-k +\frac {3\ell} 2 }
\|\D_k^\h\D_\ell^\v (\varrho_0   v_0^\h)\|_{L^2}
\\
&&\qquad\qquad{}+ (2^k+\e2^\ell)^{1- \frac 2  r } 2^{k \frac 2 r +\frac{3\ell} 2 } \|\D_k^\h\D_\ell^\v (\varrho^\h v^\h)\|_{L^1(\R^+;L^2)}\\
& \leq & C  2^{-k +\frac {3\ell} 2 } \|\D_k^\h\D_\ell^\v (\varrho_0
 v_0^\h)\|_{L^2}
+  2^{k +\frac{3\ell} 2 } \|\D_k^\h\D_\ell^\v (\varrho^\h v^\h)\|_{L^1(\R^+;L^2)}.
\eeno
By summation with respect to the indices~$k$ and~$\ell$, we get thanks to the Minkowski inequality,
$$
\sum_{k,\ell} \|\cW_{k,\ell}\|_{L^r(\R^+;L^2)} \leq C\|\varrho_0
v_0^\h\|_{\cB^{-1,\frac 32}_2} + C\|\varrho^\h
v^\h\|_{L^1(\R^+;\cB^{1,\frac 32})}.
$$
By definition of the~$\cB^{s,s'}_p$ norms, we infer that, for~$r$ in~$[1,2]$,
\beq
\label {firstpartoestimwdetaileq1}
\begin{split}
\|(\e w^\h_\e ,w^3_\e )\|_{ L^r(\R^+;\cB^{\frac 2 r,\frac
12}_2)}\leq & \sum_{(k,\ell)\in\Z^2}2^{k \frac 2 r +\frac \ell 2 }
\|\D_k^\h\D_\ell^\v (\e w^\h_\e ,w^3_\e )\|_{L^r(\R^+;L^2)} \\
\lesssim & \|\varrho_0 v^\h_0\|_{\cB^{-1,\frac 32}_2}
 +
\|\varrho^\h  v^\h\|_{L^r(\R^+;\cB^{\frac 2 r -1,\frac 32}_2)\cap L^1(\R^+;\cB^{1,\frac 32}_2) } .
\end{split}
\eeq Now let us estimate~$\|\e\p_z(\e w^\h_\e ,w^3_\e )\|_{
L^2(\R^+;\cB^{0,\frac 12}_2)}$ and $\| \e\p_z(\e w^\h_\e ,w^3_\e
)\|_{L^1(\R^+;\cB^{1,\frac 12}_2)}$. As~$w$ is a divergence free
vector field, we have \beno   \e \|\partial_z w^3_\e
\|_{L^2(\R^;\cB^{0,\frac 12}_2)} & \leq &
 C \|\e \dive_\h w^h \|_{L^2(\R^+;\cB^{0,\frac 12}_2)}\leq C \|\e w^h \|_{L^2(\R^+;\cB^{1,\frac
 12}_2)},\\
 \e \|\partial_z w^3_\e \|_{L^1(\R^;\cB^{1,\frac 12}_2)}  &\leq &
 C \|\e \dive_\h w^h \|_{L^1(\R^+;\cB^{1,\frac 12}_2)}\leq C \|\e w^\h_\e \|_{L^1(\R^;\cB^{2,\frac 12}_2)}.
\eeno Then Inequality\refeq {firstpartoestimwdetaileq1}  applied
with $r$ equal to~$2$ and~$r$ equal to~$1$ gives
\beq
\label {estimwdetaildemoeq1}
\begin{split} \e \|\partial_z w^3_\e \|_{L^2(\R^;\cB^{0,\frac 12}_2)}
+&\|\e\partial_z w^3_\e \|_{L^1(\R^;\cB^{1,\frac 12}_2)}\\
 \leq &C
\|\varrho_0 v^\h_0\|_{\cB^{-1,\frac 32}_2}
 +
C\|\varrho^\h  v^\h\|_{L^2(\R^+;\cB^{0,\frac 32}_2)\cap
L^1(\R^+;\cB^{1,\frac 32}_2) } . \end{split} \eeq Now let us
estimate~$\|\e^2 \p_z w^\h_\e \|_{L^1(\R^+;\cB^{1,\frac 12}_2)}$.
From Equality\refeq {estimwdetaildemoeq-0.5}, we infer that
\beq\label{opiu}
\begin{split}
 \e\|\D_k^\h\D_\ell^\v (\e w^\h_\e,& w^3_\e)(t) \|_{L^2} \lesssim
 \|\D_k^\h\D_\ell^\v(\varrho^\h v^\h)(t)\|_{L^2}+e^{-ct(2^{2k}+\e^22^{2\ell})}\|\D_k^\h\D_\ell^\v(\varrho_0 v_0^\h)\|_{L^2}\\
 &\qquad\quad+
 \int_0^t   e^{-c(t-t') (2^{2k}+\e^22^{2\ell})}\bigl(\e
 2^{\ell}+2^k\bigr)
 \e 2^{\ell} \|\D_k^\h\D_\ell^\v
 (\varrho^\h v^\h)(t')\|_{L^2}dt'.\end{split} \eeq
Taking the~$L^1$ norm in time gives
$$
 2^{k +\frac{3 \ell}  2 }\e
\|\D_k^\h\D_\ell^\v  (\e w^\h_\e,w^3_\e) \|_{L^1(\R^+;L^2)} \lesssim
2^{-k+\f{3\ell}2}\|\D_k^\h\D_\ell^\v(\varrho_0 v_0^\h)\|_{L^2} +
2^{k +\frac{3 \ell} 2 }
 \|\D_k^\h\D_\ell^\v (\varrho^\h v^\h)\|_{L^1(\R^+;L^2)}.
 $$
  By summation  with respect to the indices~$k$ and~$\ell$, we infer that
$$
\e\| (\e w^\h_\e,w^3_\e) \|_{L^1(\R^+;\cB^{1,\frac 32}_2)} \lesssim
\|\varrho_0 v^\h_0\|_{\cB^{-1,\frac 32}_2}
 +
\|\varrho^\h v^\h\|_{L^1(\R^+;\cB^{1,\frac 32}_2)}.
$$
Together with Inequalities\refeq {firstpartoestimwdetaileq1}
\and\refeq {estimwdetaildemoeq1}, this gives the first inequality of
the proposition.

To prove the second inequality, we get, by taking the~$L^4$ norm in
time of \eqref{opiu}, that
$$
2^{\f{k}2}2^{\f{\ell}2}\e\|\D_k^\h\D_\ell^\v (\e
w_\e^\h,w^3_\e)\|_{L^4(\R^+;L^2)}\lesssim
2^{\f{\ell}2}\|\D_k^\h\D_\ell^\v(\varrho_0v_0^\h)\|_{L^2}+2^{\f{k}2}2^{\f{\ell}2}\|\D_k^\h\D_\ell^\v
(\varrho^\h v^\h)\|_{L^4(\R^+;L^2)}. $$
 Summing up the above inequality  with respect to the indices~$k$
 and~$\ell$ yields $$ \e\| (\e w^\h_\e,w^3_\e)\|_{ L^4(\R^+;\cB^{\f12,\frac
12}_2)}  \leq C\|\varrho_0 v^\h_0\|_{\cB^{0,\frac 12}_2}
 + C\|\varrho^\h  v^\h\|_{L^4(\R^+;\cB^{\f12,\frac 12}_2)}.
 $$
 This proves the second inequality
of the proposition.

\medbreak Let us prove the third inequality.
Using\refeq{estimwdetaildemoeq-1}, we can write that
$$
\bigl | \cF\D_\e (\e \wh w^\h_\e ,w^3_\e )(t,\xi) \bigr | \leq
C\int_0^t e^{-(t-t')(|\xi_\h|^2+\e^2\zeta^2)}
(\e|\zeta|+|\xi_\h|)|\zeta| \bigl |\cF(\partial_t( \varrho^\h v^\h
)) (t',\xi) \bigr | dt'.
$$
Then applying the cut off operators in both horizontal and vertical
frequencies gives
$$
\longformule{ 2^{k\al +\frac \ell 2}\| \D_k^\h\D_\ell^\v  \D_\e (\e
\wh w^\h_\e ,w^3_\e )(t)\|_{L^2} } { {}\leq \int_0^t
e^{-c(t-t')(2^{2k}+\e^22^{2\ell}) } (\e 2^\ell +2^k) 2^{k\al +\frac
{3\ell}  2}\|\D_k^\h\D_\ell^\v  \partial_t ( \varrho^\h
v^\h)(t')\|_{L^2} dt'. }
$$
Using Young's inequality, we get
$$
2^{k\al +\frac \ell 2}\| \D_k^\h\D_\ell^\v  \D_\e (\e  w^\h_\e
,w^3_\e )\|_{L^1(\R^+; L^2)} \leq C  2^{k(-1+\al)  +\frac {3\ell} 2}
\|\D_k^\h\D_\ell^\v  \partial_t ( \varrho^\h
v^\h)\|_{L^1(\R^+;L^2)}.
$$
By summation the above inequality with respect to the indices~$k$
and~$\ell,$ we get
\beq \label {}
 \| \D_\e (\e  w^\h_\e ,w^3_\e
)\|_{L^1(\R^+; \cB^{\al,\frac 12}_2)} \leq C_\al \bigl\|\partial_t
(\varrho^\h v^\h)\bigr\|_{L^1(\R^+;\cB^{-1+\al,\frac 32}_2)}. \eeq
In order to get the estimates on the pressure term~$\Pi^1_\e $, let
us observe that Relation\refeq  {eqpressurew} implies that
$$
\e \nabla_\e \Pi^1_\e
=\left ( \begin{matrix}  \ds {\fauxatop} -\e
\nabla_\h \D_\e^{-1} \partial_z \dive_h\D_\h^{-1} \partial_z
\partial_t \bigl( \varrho^\h  v^\h   \bigr) \\ \e^2  \D_\e^{-1}
\partial_z ^2\dive_\h\D_\h^{-1} \partial_z \partial_t \bigl(
\varrho^\h  v^\h   \bigr)
\end{matrix} \right)
$$
As we have
$$
\frac { \e |\zeta  | } {|\xi_\h|^2+\e^2\zeta  ^2} \leq \frac 1 {|\xi_\h|} \andf  \frac { \e^2 |\zeta  |^2 } {|\xi_\h|^2+\e^2\zeta  ^2} \leq 1 $$
 we infer that, for any~$\al$ in~$]0,1[$,
 $$
 \e \|\nabla_\e \Pi^1_\e \|_{L^1(\R^+;\cB_2^{\al,\frac 12} )} \leq \bigl\|\partial_z \partial_t \bigl( \varrho^\h  v^\h   \bigr)
\bigr\|_{L^1(\R^+;\cB_2^{-1+\al, \frac 12})} .
 $$
 Then  the proposition is proved. \end{proof}

 Let us now turn to the proof of Proposition \ref{estimwdetail}.

\begin{proof}[Proof of Proposition \ref{estimwdetail}] It follows
from the law of product \eqref{lawofproductanisobasic} and Lemma
\ref{S7lem4.5} that \beno \|\varrho_0 v_0^\h\|_{\cB^{0,\f12}_2}
&\lesssim& \|\varrho_0
\|_{\cB^{\f12,\f12}_2}\|v_0^\h\|_{\cB^{\f12,\f12}_2}\leq
\cC_0\eta,\\
\|\varrho^\h v^\h\|_{L^4(\R^+;\cB^{\f12,\f12}_2)} &\lesssim &
\|\varrho^\h\|_{L^\infty(\R^+;\cB^{\f34,\f12}_2)}\|v^\h\|_{L^4(\R^+;\cB^{\f34,\f12}_2)}\leq\cC_0\eta,
\eeno and
$$
\longformule{
 \|\varrho^\h v^\h\|_{L^2(\R^+;\cB^{0,\f32}_2)}\leq
\|\varrho^\h\|_{L^\infty(\R^+;\cB^{\f12,\f12}_2)}\|\p_zv^\h\|_{L^1(\R^+;\cB^{\f12,\f12}_2)}
} { {}
+\|\p_z\r^\h\|_{L^\infty(\R^+;\cB^{\f12,\f12}_2)}\|v^\h\|_{L^1(\R^+;\cB^{\f12,\f12}_2)}\leq
\cC_0 \eta. }
$$ Similarly, we deduce from law of product
\eqref{lawofproductanisobasic} and \eqref{S7eq12} that
$$
\longformule{ \|\varrho^\h v^\h\|_{L^1(\R^+;\cB^{1,\f32}_2)}\leq
\|\varrho^\h\|_{L^\infty(\R^+;\cB^{1,\f12}_2)}\|\p_zv^\h\|_{L^1(\R^+;\cB^{1,\f12}_2)}
} { {}
+\|\p_z\r^\h\|_{L^\infty(\R^+;\cB^{1,\f12}_2)}\|v^\h\|_{L^1(\R^+;\cB^{1,\f12}_2)}\leq
\cC_0\eta. }
$$
Hence by virtue of the first two inequalities of Proposition
\ref{estimwdetaila} and the remark following \eqref{1.7dfg}, we
conclude the first inequality of Proposition \ref{estimwdetail}.

On the other hand, for any $\al$ in~$ ]0,1[,$ we get, by applying the
law of product \eqref{lawofproductanisobasic}, that
$$
\longformule{ \|\p_z(\p_t\varrho^\h
v^\h)\|_{L^1(\R^+;\cB^{-1+\al,\f12}_2)}\lesssim
\|\p_t\p_z\r^\h\|_{L^\infty(\R^+;\cB^{\f{\al}2,\f12}_2)}\|v^\h\|_{L^1(\R^+;\cB^{\f{\al}2,\f12}_2)}
} { {}
+\|\p_t\r^\h\|_{L^\infty(\R^+;\cB^{\f{\al}2,\f12}_2)}\|\p_zv^\h\|_{L^1(\R^+;\cB^{\f{\al}2,\f12}_2)},
}
$$
 and
 $$
\longformule{ \|\p_z(\varrho^\h
\p_tv^\h)\|_{L^1(\R^+;\cB^{-1+\al,\f12}_2)}\lesssim
\|\p_z\r^\h\|_{L^\infty(\R^+;\cB^{\f{\al}2,\f12}_2)}\|\p_tv^\h\|_{L^1(\R^+;\cB^{\f{\al}2,\f12}_2)}
} { {}
+\|\varrho^\h\|_{L^\infty(\R^+;\cB^{\f{\al}2,\f12}_2)}\|\p_z\p_tv^\h\|_{L^1(\R^+;\cB^{\f{\al}2,\f12}_2)},
}
$$
so that by applying Lemma \ref{S7lem4.5}, we obtain
$$\bigl\|\partial_t \bigl( \varrho^\h  v^\h
\bigr)\|_{L^1(\R^+; \cB_2^{-1+\al,\frac 32})}\leq \|\p_z(\r^\h_t
v^\h)\|_{L^1(\R^+;\cB^{-1+\al,\f12}_2)}+\|\p_z(\varrho^\h
v^\h_t)\|_{L^1(\R^+;\cB^{-1+\al,\f12}_2)}\leq \cC_0. $$ This
together the third inequality of Proposition \ref {estimwdetaila}
leads to the second inequality of  Proposition \ref {estimwdetail}.
\end{proof}

\begin{col}
\label {estimE2E3} {\sl Under the assumptions of Theorem \ref
{insslowvar}, there holds \eqref{estimE1}.}
\end{col}

\begin{proof} We first deduce from the law of product
\eqref{lawofproductanisobasic} that
\beno
\|v^\h\cdot\na_\h(\e w_\e^\h,
w^3_\e)\|_{L^1(\R^+;\cB^{0,\f12}_2)}
& \lesssim &
\|v^\h\|_{L^2(\R^+;\cB^{1,\f12}_2)}\|(\e
w_\e^\h,w^3_\e)\|_{L^2(\R^+;\cB^{1,\f12}_2)},\\
\|\e^2 w_\e\cdot\na (\e
w_\e^\h,w^3_\e)\|_{L^1(\R^+;\cB^{0,\f12}_2)}
& \lesssim &
\|\e w_\e\|_{L^2(\R^+;\cB^{1,\f12}_2)}\|\na_\e (\e
w^\h_\e,w^3_\e)\|_{L^2(\R^+;\cB^{0,\f12}_2)},
\eeno
and
$$
\longformule{ \|\e w_\e\cdot\na
(v^\h,0)\|_{L^1(\R^+;\cB^{0,\f12}_2)} \lesssim \|\e
w_\e^\h\|_{L^2(\R^+;\cB^{1,\f12}_2)}\|v^\h\|_{L^2(\R^+;\cB^{1,\f12}_2)}
} {{}
+\e\|w^3_\e\|_{L^4(\R^+;\cB^{\f12,\f12}_2)}\|\p_zv^\h\|_{L^{\f43}(\R^+;\cB^{\f12,\f12}_2)}.
}
$$
 Moreover, it follows from \eqref{decomppressureh} and
\eqref{lawofproductanisobasic} that
\beno
\|\p_z\Pi_Q^\h\|_{L^1(\R^+;\cB^{0,\f12}_2)}
&\lesssim &
\|\p_z(\r^\h v^\h\otimes v^\h)\|_{L^1(\R^+;\cB^{0,\f12}_2)}\\
 & \lesssim &
\|\r_z^\h\|_{L^\infty(\R^+;\cB^{\f12,\f12}_2)}\|v^\h\|_{L^2(\R^+;\cB^{\f12,\f12}_2)}\|v^\h\|_{L^2(\R^+;\cB^{1,\f12}_2)}\\
&&\qquad{}+
\bigl(1+\|\varrho^\h\|_{L^\infty(\R^+;\cB^{\f12,\f12}_2)}\bigr)\|v^\h\|_{L^2(\R^+;\cB^{\f12,\f12}_2)}\|\p_zv^\h\|_{L^2(\R^+;\cB^{1,\f12}_2)}.
\eeno Hence by virtue of \eqref {analyzeErrortermeq1}, Proposition
\ref{estimwdetail}, and Lemma \ref{S7lem4.5}, we conclude the proof
of  \eqref{estimE1}.
\end{proof}

\setcounter{equation}{0}
\section {The control of the term~$b_\e$}
\label {control_b}

The purpose of this section  is the control of the term~$b_\e$ which
satisfies Equation\refeq {transportequationb}  as described by
Proposition\refer  {propbRt}.  Namely we want to decompose the solution~$b_\e$ of
Equation\refeq {transportequationb} which is
\beq
\label {transportequationbrecall}
\p_tb_\e+u_\e\cdot\na b_\e = - R_\e\cdot\nabla
[a^\h]_\e-\e[w_\e\cdot\na a^\h]_\e \with {b_\e}_{|t=0} =0
\eeq
as~$b_\e=\overline b_\e+\wt b_\e$ such that
$$
\|\overline b_\e(t) \|_{\cB^{\f2p,\f1p}_p} \leq  \cC_0\eta(1+\cR_0)
\w{t}^{\frac12} \andf
 \|\wt b_\e(t)\|_{L^p} \leq
\cC_0\e^{1-\f1p}(1+\cR_0)^2\w{ t}.
$$
In order to do it, let
us introduce  the following decomposition of~$b_\e=b_{1,\e}+b_{2,\e}+b_{3,\e} $ with
 \beq
\label {propbRtdemoeq0}
\begin{split}
\partial_t b_{1,\e}  + u_\e\cdot \nabla b_{1,\e} =&  -\e R^3_\e[\partial_z a^\h]_\e \,,\\
\partial_tb_{2,\e} + [v^\h]_\e\cdot \nabla _\h b_{2,\e} =& -R^\h_\e\cdot\nabla_\h [a^\h]_\e-\e [w_\e\cdot\nabla a^\h]_\e
 \andf \\
 \partial_tb_{3,\e} + u_\e\cdot \nabla b_{3,\e} =& -\e \bigl((\e w^\h_\e,w^3_\e)+R_\e)\cdot\nabla  b_{2,\e} .
\end{split}
\eeq

Let us first estimate~$\|b_{1,\e}(t)\|_{L^p}$.   As the vector
field~$u_\e$ is divergence free, we have
\beno
\|b_{1,\e}(t)\|_{L^p}
 \leq  \e^{1-\frac 1p}  \int_0^t \|R^3_\e(t')\|_{L^\infty}
\|\partial_z a^\h(t')\|_{L^p} dt'.
\eeno
Interpolation inequality\refeq  {interpolBesovanisolinfty}  of Lemma\refer {S0lem3}  and the induction
hypothesis \eqref {defininductiontime1} implies that, for any~$t$
less than~$\overline T_\e$, we have \ben \int_0^t
\|R(t')\|_{L^\infty} dt'  & \leq  & \int_0^t \|R(t')\|^{\frac 2 3}
_{\cB^{-\frac 12 +\frac 2 p,\frac 1 p}_p}
 \|\nabla _\h R(t')\|^{\frac 1 3} _{\cB^{\frac 2 p,\frac 1 p}_p} dt'\nonumber\\
 & \leq  & t^{\frac 12}
 \Bigl(\int_0^t \|R(t')\| _{\cB^{-\frac 12 +\frac 2 p,\frac 1 p}_p}^4dt'\Bigr)^{\f14\times \frac 2 3}
 \Bigl(\int_0^t\|\nabla _\h R(t')\|_{\cB^{\frac 2 p,\frac 1 p}_p} dt'\Bigr)^{\frac 1 3}\label {propbRtdemoeq1nm}\\
 & \leq   & \cR_0 t^{\frac 12}.\nonumber
\een Then using   the Estimate \eqref{S7eq0.3}, we infer that, for
any~Ê$t$ less than~$\overline T_\e $, \beq \label {propbRtdemoeq1}
 \|b_{1,\e}(t)\|_{L^p}    \leq  \cC_0\eta\e^{1-\frac 1p} t^{\frac 12}\cR_0.
\eeq

In order to estimate~$\|b_{3,\e}(t)\|_{L^p}$, we need to estimate~$\|\nabla b_{2,\e}(t)\|_{L^p}$. Let us observe that
$$
\longformule{
\partial_t \nabla b_{2,\e} + [v^\h]_\e\cdot \nabla _\h \nabla b_{2,\e} = -\nabla R^\h_\e\cdot\nabla_\h [a^\h]_\e
- R^\h_\e\cdot\nabla_\h \nabla [a^\h]_\e}{{} -\nabla [v^\h]_\e\cdot
\nabla _\h  b_{2,\e}-\e[\na_\e w_\e\cdot\na a^\h]_\e-\e[
w_\e\cdot\na\na_\e a^\h]_\e.}
$$
Using the fact that~$v^\h$ is divergence free, we get, \beq \label
{propbRtdemoeq1a}\begin{split} \f{d}{dt}\|\na
b_{2,\e}(t)\|_{L^p}\leq \e^{-\f1p}\Bigl(&\bigl(\|\na
R^\h_\e\|_{L^\infty}+\|\e\na_\e
w_\e\|_{L^\infty}\bigr)\|\na a^\h\|_{L^p} \\
&+\bigl(\|R^\h_\e\|_{L^\infty}+\|\e w_\e\|_{L^\infty}\bigr)\|\na^2
a^\h\|_{L^p}\Bigr)+\|\na v^\h\|_{L^\infty}\|\nh b_{2,\e}\|_{L^p}.
\end{split} \eeq Estimate\refeq {S7eq0.3}  together with Sobolev
embedding implies that \beq \label {propbRtdemoeq1b} \forall
t<\overline T_\e\,,\ \|\nabla a^\h(t)\|_{L^p} \leq \eta\, \cC_0
\andf \|\nabla^2 a^\h(t)\|_{L^p} \leq \eta\, \cC_0 . \eeq Induction
hypothesis\refeq {defininductiontime1} and Proposition\refer
{estimwdetail} implies that
$$
\int_0^{\overline T_\e} \bigl(\|\na
R^\h_\e(t)\|_{L^\infty}+\|\e\na_\e w_\e(t)\|_{L^\infty}\bigr)dt
\lesssim \cR_0+\cC_0.
$$
Together with\refeq {propbRtdemoeq1b} this implies that \beq \label
{propbRtdemoeq2c} \int_0^{\overline T_\e} \bigl(\|\na
R^\h_\e(t)\|_{L^\infty}+\|\e\na_\e w_\e(t)\|_{L^\infty}\bigr) \|\na
a^\h(t) \|_{L^p}dt \lesssim  \eta \, \cC_0(1+ \cR_0). \eeq
 Proposition \ref{estimwdetail} yields that \beq \label
{propbRtdemoeq2cmn} \int_0^t\|\e w_\e(t')\|_{L^\infty}dt'\lesssim
t^{\f12}\|(\e w^\h_\e,w^3_\e)\|_{L^2(\R^+;\cB^{1,\f12})}\leq \cC_0
t^{\frac 12}. \eeq Proposition \ref{consequenceSection345} claims in
particular that~$\|\nabla v^\h(t)\|_{L^\infty}$  is an integrable
function on~$R^+$ the integral of which is less than some~$\cC_0$.
Applying Gronwall's Lemma to \eqref{propbRtdemoeq1a} and using the
Estimates \eqref{propbRtdemoeq1nm}, \refeq {propbRtdemoeq2c} and
\eqref{propbRtdemoeq2cmn}, we get for any~$t$ less than~$\overline
T_\e $,
$$
\|\nabla b_{2,\e}(t)\|_{L^p}  \leq \cC_0\e^{-\frac 1p
}\bigl(1+\cR_0\bigr)\w{ t}^{\frac 12}.
$$
Now let us consider the equation on~$b_{3,\e}$
in\refeq{propbRtdemoeq0}. As~$u_\e$ is divergence free, we get, by
applying again  the Estimates \eqref{propbRtdemoeq1nm} and
\eqref{propbRtdemoeq2cmn}, that \beq \label {propbRtdemoeq2}
\begin{split}
\|b_{3,\e}(t)\|_{L^p}   \leq &
\e\int_0^t \bigl(\|R_\e(t')\|_{L^\infty}  +\|\e
w_\e^\h(t')\|_{L^\infty}+\|w_\e^3(t')\|_{L^\infty}\bigr)
\|\nabla b_{2,\e}(t')\|_{L^p}  dt'\\
\leq &  \cC_0\e^{1-\frac 1p}\bigl(1+\cR_0\bigr)^2\w{t}.
\end{split}
\eeq Defining~$\wt b_\e=b_{1,\e}+b_{3,\e}$ ensures the Inequality
\eqref{4.18} of Proposition\refer  {propbRt}. As  there no power
of~$\e$ in the right hand side of the equation on~$b_{2,\e}$ we must
use another norm to measure the size of~$b_{2,\e}$ The fact  that
the convection vector field involved in the equation of~$b_{2,\e}$
has no vertical component will allow us to propagate the anisotropic
regularity thanks to the following lemma.

\begin{prop}
\label {propageanisoh}
{\sl  Given a smooth vector field $v^\h$ with
$\dive_\h v^\h=0,$ we consider the following transport equation with
a parameter $z$
\begin{equation}
\left\{\begin{array}{c}
 \displaystyle \p_tb(t,x_\h,z)+v^{\rm h}(t,x_\h,z)\cdot\na_{\rm
h}b(t,x_\h,z) =g(t,x_\h,z), 
\\
\displaystyle b(0,x_\h,z)=b_0(x_{\rm h},z).
\end{array}\right. \label{3.1}
\end{equation}
Let $p$ be in~$]2,4[$. Let us define
$$
 \cV_p(t) \eqdefa  \sup_{z'\in \R} \int_0^t \|\nabla_\h
v^\h(t',\cdot,z')\|_{(\cB^{\frac 2p}_p)_\h}dt'.
$$
Then for $s$ in~$ ]0, 2/p],$ we have
 \beq
 \label {propageanisoheq1}
 \begin{split}
 &\exp\bigl(-C\cV_p(t)\bigr)  \|b\|_{L^\infty_t(\cB^{s,\f1p}_p)}\leq
C\|b_0\|_{\cB^{s,\f1p}_p}+ C\|g\|_{L^1_t(\cB^{s,\f1p}_p)}
\\
&\qquad\qquad\qquad\qquad{}+C\|v^\h\|_{L^1_t(\cB^{\f2p,\f1p}_p)}\bigl(\|\nh
b_0\|_{L^\infty_{\rm v}(\cB^{s}_p)_\h}+\|\nh g\|_{L^\infty_{\rm
v}(L^1_t(\cB^{s}_p)_\h)}\bigr).
\end{split}
\eeq
}
\end{prop}

\begin{proof} Let us first observe that  Theorem 3.14 of\ccite{BCD} implies that for any~$\s$ in~$\ds \bigl[0,1+2/p\bigr]$, we have, for any~$z$ in~$\R$
\beq
\label{S6eq9b}
\begin{split}
\|b(t,\cdot,z)\|_{(\cB^{\s}_p)_\h} \leq
&\Bigl(\|b_0(\cdot,z)\|_{(\cB^{\s}_p)_\h} +\int_0^t
\|g(t',\cdot,z)\|_{(\cB^{\s}_p)_\h}dt'\Bigr) \exp
\bigl(C\cV_p(t)\bigr).
\end{split}
\eeq
Let us define $(\tau_{-z}b)(x_\h, z') \eqdefa
b(x_\h,z'+z)$ and~$\d_{-z} a \eqdefa  \tau_{-z} a-a.$ Then in view of \eqref{3.1}, one~has
\beno
\p_t\tau_{-z}b+\tau_{-z}v^\h\cdot\nh\tau_{-z}b=\tau_{-z}g.
\eeno
Subtracting \eqref{3.1} from the above equation, we get
 \beq
 \label {propageanisohdemoeq1}
 \p_t\d_{-z}b+v^\h\cdot\nh\d_{-z}b+\d_{-z}v^\h\cdot\nh
 \tau_{-z} b=\d_{-z}g.
 \eeq
Applying again Theorem 3.14 of\ccite{BCD}, we infer that \beq \label
{S6eq9bb}
\begin{split}
&  \exp\bigl(-C\cV_p(t)\bigr) \|\d_{-z}b(t,\cdot ,z') \|_{(\cB^{s}_p)_\h} \leq
\|\d_{-z}b_0(\cdot ,z') \|_{(\cB^{s}_p)_\h}
\\
&\qquad\qquad{}+ C \int_0^t
\|\d_{-z}g(t',\cdot,z')\|_{(\cB^{s}_p)_\h} dt' +\int_0^t
\bigl\|\d_{-z}v^\h(t',\cdot,z')\cdot\nh \tau_{-z}
b(t')\bigr\|_{(\cB^{s}_p)_\h} dt'.
\end{split}
\eeq The law of product in~$\R^2,$ which claims that for~$s$
in~$[0,2/p],$ ~$\|ab\|_{\cB^s_p} \leq C \|a\|_{\cB^{\frac 2p} _p}
\|b\|_{\cB^s_p},$  together with Inequality\refeq{S6eq9b}
 implies that \beno
I(t,z,z')& \eqdefa & \int_0^t \bigl\|\d_{-z}v^\h(t',\cdot,z')\cdot\nh \tau_{-z}b(t',\cdot, z')\bigr\|_{(\cB^{s}_p)_\h}  dt' \\
& \leq  & \sup_{\substack{ t'\in [0,t]\\ z'\in \R}}\|\nh b(t',\cdot, z')\bigr\|_{(\cB^{s}_p)_\h}
 \int_0^t \|\d_{-z}v^\h(t',\cdot,z')\bigr\|_{(\cB^{\frac 2 p}_p)_\h}dt'\\
 & \leq  &
\Bigl(  \|\nh b_0\|_{L^\infty_\v(\cB^{s}_p)_\h}+ \sup_{z\in
\R}\int_0^t \|\nh g(t',\cdot,z)\|_{(\cB^{\s}_p)_\h}dt'\Bigr)
\exp \bigl(C\cV_p(t)\bigr)\\
& & \qquad\qquad\qquad\qquad\qquad\qquad\qquad
\qquad{}\times \int_0^t \|\d_{-z}v^\h(t',\cdot,z')\bigr\|_{(\cB^{\frac 2 p}_p)_\h}dt'.
\eeno
Plugging this into Inequality\refeq {S6eq9bb} gives
$$
\longformule
{
\exp\bigl(-C \cV_p(t)\bigr)   \|\d_{-z}b(t,\cdot ,z') \|_{(\cB^{s}_p)_\h} \leq
\|\d_{-z}b_0(\cdot ,z') \|_{(\cB^{s}_p)_\h}  + C \int_0^t
 \|\d_{-z}g(t',\cdot,z')\|_{(\cB^{s}_p)_\h} dt'
 }
 {
{}+ \Bigl(\|\nh b_0\|_{L^\infty_\v(\cB^{s}_p)_\h}+\sup_{z\in
\R}\int_0^t  \|\nh g(t',\cdot,z)\|_{(\cB^s_p)_\h}dt'\Bigr) \int_0^t
\|\d_{-z}v^\h(t',\cdot,z')\|_{(\cB^{\f2p}_p)_\h} dt'. }
$$
Taking $L^p$ norm of the  above inequality with respect to the
vertical variable~$z'$ yields
$$
\longformule{ \exp\bigl(-C\cV_p(t)\bigr)
\|\d_{-z}b(t)\|_{L^p_\v(\cB^{s}_p)_\h} \leq
\|\d_{-z}b_0\|_{L^p_\v(\cB^{s}_p)_\h}  + C \int_0^t
\|\d_{-z}g(t')\|_{L^p_{\rm v} (\cB^{s}_p)_\h} dt' } { {}+ C
\Bigl(\|\nh b_0\|_{L^\infty_\v(\cB^{s}_p)_\h}+\sup_{z\in \R}\int_0^t
\|\nh
g(t',\cdot,z)\|_{(\cB^s_p)_\h}dt'\Bigr)\int_0^t\|\d_{-z}v^\h(t')\|_{L^p_{\rm
v}(\cB^{\f2p}_p)_\h} dt' . }
$$
 Dividing the above inequality by $|z|^{\f1p}$ and taking the $L^p$ norm of
the resulting inequality with the measure  $\ds\f{dz}{|z|}$ over~$\R,$ we obtain Inequality\refeq {propageanisoheq1} and thus the proposition.
\end{proof}

\begin{proof}[Continuation of the proof to Proposition \ref{propbRt}]
Let us observe that Proposition\refer  {propageanisoh} implies that
for $t$ less than~$\overline T_\e$,
$$\longformule{
\|b_{2,\e}\|_{L^\infty_t(\cB^{\f2p,\f1p}_p)}\leq
\cC_0\Bigl(\|R^\h_\e\cdot\nh[a^\h]_\e\|_{L^1_t(\cB^{\f2p,\f1p}_p)}+\e\|w_\e\cdot\na
a^\h\|_{L^1_t(\cB^{\f2p,\f1p}_p)}}{{}+\|v^\h\|_{L^1_t(\cB^{\f2p,\f1p}_p)}\bigl(
\|\nh(R^\h_\e\cdot\nh[a^\h]_\e)\|_{L^\infty_\v(L^1_t(\cB^{\f2p}_p)_\h)}+\e\|\nh(w_\e\cdot\na
a^\h)\|_{L^\infty_\v(L^1_t(\cB^{\f2p}_p)_\h)}\bigr)\Bigr).}$$ Hence
it follows from the law of product \eqref{lawofproductanisobasic},
Proposition \ref {estimwdetail} and the the induction hypothesis
\eqref {defininductiontime1} that \beq \label{propbRteq1a}
\begin{split}
\|b_{2,\e}\|_{L^\infty_t(\cB^{\f2p,\f1p}_p)}\leq
&\Bigl(\|R^\h_\e\|_{L^1_t(\cB^{\f2p,\f1p}_p)}+\|\e
w_\e\|_{L^1_t(\cB^{\f2p,\f1p}_p)}+\|\nh
R^\h_\e\|_{L^1_t(\cB^{\f2p,\f1p}_p)}\\
& \ +\|\e\nh w_\e\|_{L^1_t(\cB^{\f2p,\f1p}_p)}\Bigr)\left(\|\na
a^\h\|_{L^\infty_t(\cB^{\f2p,\f1p}_p)}+\|\na
a^\h\|_{L^\infty_\v(L^\infty_t(\cB^{1+\f2p}_p)_\h)}\right)\\
\leq & \cC_0(1+\cR_0)\w{t}^{\f12}\left(\|\na
a^\h\|_{L^\infty_t(\cB^{\f2p,\f1p}_p)}+\|\na
a^\h\|_{L^\infty_\v(L^\infty_t(\cB^{1+\f2p}_p)_\h)}\right).
\end{split} \eeq Yet it follows from Lemma \ref{lemBern}, \eqref{S7eq12qwqp} and \eqref{S7eq0.3} that
$$
\displaylines { \|\na a^\h\|_{L^\infty_t(\cB^{\f2p,\f1p}_p)}\lesssim
\|\na a^\h\|_{L^\infty_t(\cB^{1,\f12}_2)}\lesssim \|\na
a^\h\|_{L^\infty_t(\cB^{\f12,\f12}_2)}^{\f12}\|\nh\na
a^\h\|_{L^\infty_t(\cB^{\f12,\f12}_2)}^{\f12}\leq \cC_0\eta, \cr
\|\na a^\h\|_{L^\infty_\v(L^\infty_t(\cB^{1+\f2p}_p)_\h)}\lesssim
\|\na a^\h\|_{L^\infty_\v(L^\infty_t(\cB^{2}_2)_\h)}\lesssim \|\na
a^\h\|_{L^\infty_\v(L^\infty_t(\dot H^{1}_\h))}^{\f12}\|\na
a^\h\|_{L^\infty_\v(L^\infty_t(\dot H^{3}_\h))}^{\f12}\leq\cC_0\eta.
}
$$
This implies that for all~$t$  less than~$\overline T_\e $
$$
\|b_{2,\e}
\|_{L^\infty_t(\cB^{\frac 2p,\frac 1p}_{p}) } \leq
\cC_0\eta(1+\cR_0)\w{ t}^{\frac 12}.
$$
Taking $\overline{b}_\e=b_{2,\e}$ leads to \eqref{4.13}. We thus
 complete the
proof of Proposition \ref{propbRt}.
\end{proof}

\begin{col}\label{Sectcolcontrolb}
{\sl Under the assumptions of Theorem \ref {insslowvar},  the
inequalities of Assertion\refeq {estiE4}  holds namely, for~$p$ in~$ ]3,4[$ and $\d$ in~$\bigl]0, 1-3/p\bigr[.$
$$
\|E^{4,1}_\e\|_{{L^1_T(\cB^{-1+\f2p,\f1p}_p)}}\leq
\cC_0\eta(1+\cR_0)\andf
\|E^{4,2}_\e\|_{{L^1_T(\cB^{-1+\d+\f3p,-\d}_p)}}\leq
\cC_0\e^{1-\f1p}(1+\cR_0)^2,
$$
}
\end{col}

\begin{proof}  We deduce
 from a similar derivation of \eqref{S5eq10} that
 \beno \|
\Pi^\h(t)\|_{\cB^{\f2p,\f1p}_p}\leq C\Bigl(
\|a^\h\|_{L^\infty(\R^+;\cB^{\f2p,\f1p}_p)}\|
\nh\uh(t)\|_{\cB^{\f2p,\f1p}_p}+\|\uh(t)\|_{\cB^{\f2p,\f1p}_p}^2\Bigr),
\eeno which together with \eqref{S7eq9} and \eqref{S7eq12qwqp}
ensures that
 \beq
\label{S8eq1o}
\|\D_\h\uh(t)\|_{\cB^{-1+\f2p,\f1p}_p}+\|\nh\Pi^\h(t)\|_{\cB^{-1+\f2p,\f1p}_p}
\leq \cC_0\w{t}^{-2}\log^{\bigl(1-\f2p\bigr)\f2p}\w{t}. \eeq
Moreover, as $p$ belongs to~$ ]3,4[$ and  $\d$ to~$
\bigl]0,1-3/p\bigr[,$ we have
$$
-1+\d+\frac 5 p\in \Bigl]0\virgp\,\frac 2p\Bigr[\andf -\d+\frac 1 p\in ]0,1[.
$$
Then it follows from Inequality \eqref{S7eq10} that
 \beq
 \label{S8eq1p}
\begin{split}
\int_{\R^+}\w{t}&\Bigl(\|\D_\h\uh(t)\|_{\cB^{-1+\d+\f5p,-\d+\f1p}_p}+\|\nh
\Pi^\h(t)\|_{\cB^{-1+\d+\f5p,-\d+\f1p}_p}\Big) dt\\
&\qquad\qquad\qquad\qquad\qquad\qquad\qquad\leq
\int_{\R^+}\w{t}^{-\bigl(1+\f\d2+\f3{2p}\bigr)}\log^{\d+\f3p}dt\leq
\cC_0.
\end{split}
\eeq

On the other hand, it follows from the the law of product
\eqref{lawofproductanisobasic} that
$$
\longformule{
 \bigl\|\bar b\bigl([\D_\h\uh]_\e+[\nh
\Pi^\h]_\e\bigr)\bigr\|_{L^1_t(\cB^{-1+\f2p,\f1p}_p)}
}
{{}
 \lesssim
\int_0^t\|\bar{b}(t')\|_{\cB^{\f2p,\f1p}_p}
\bigl(\|\D_\h\uh(t')\|_{\cB^{-1+\f2p,\f1p}_p}+\|\nh \Pi^\h(t')\|_{\cB^{-1+\f2p,\f1p}_p}\bigr)dt',
}
$$
from which, \eqref{4.13} and \eqref{S8eq1o}, we infer  the first
inequality of the Corollary. Similarly, again~$p$ belongs to~$
]3,4[$ and  $\d$ to~$ \bigl]0,1-3/p\bigr[,$ the the law of product
\eqref{lawofproductanisobasic} ensures that
$$
\longformule
{
 \bigl\|\tilde
b\bigl([\D_\h\uh]_\e+[\nh
\Pi^\h]_\e\bigr)\bigr\|_{L^1_t(\cB^{-1+\d+\f3p,-\d}_p)}
}
{{}
 \lesssim
\int_0^t\|\tilde{b}(t')\|_{L^p}\bigl(\|\D_\h\uh(t')\|_{\cB^{-1+\d+\f5p,-\d+\f1p}_p}+\|\nh
\Pi^\h(t')\|_{\cB^{-1+\d+\f5p,-\d+\f1p}_p}\bigr)dt',
}
$$
 which together with \eqref{4.18} and \eqref{S8eq1p} gives rise to the second inequality of the corollary.
\end{proof}

\setcounter{equation}{0}
 \section {Conclusion of the proof of the main theorem}
 \label {conclusif}

\begin{proof}[Proof of Theorem\refer{insslowvar}]
 Let us first observe that law of products implies that, if $p\in ]3,4[$ and ~$\|a\|_{\cB^{\frac 2p,\frac 1p}_p}$ is less
 than~$c_p,$ which is the case for ~$\|a_\e(t)\|_{\cB^{\frac 2p,\frac 1p}_p}$ for $t\leq \overline{T}_\e$ thanks to Corollary \ref{fulltransport0}
 provided that
 $\eta$ is sufficiently small,
 $\PP_a$ given by Definition\refer{definmodifiedLerayproj} maps continuously from ~$\cB^{s_1,s_2}_p$ into itself for any~$s_1$ in~$\left]-2/p,2/p\right]$ and
$s_2$ in~$\left]-1/p,1/p\right],$  which~reads
 \beq
 \label {conclusifeq1}
  \|\PP_a g\|_{\cB^{s_1,s_2}_{p}} \lesssim
\|g\|_{\cB^{s_1,s_2}_{p}} . \eeq  Let us now fix $p$ in~$ ]3,4[$ and
$\d$ in~$ \bigl]0,1-3/p\bigr[,$ which is determined by
\eqref{estiE2}. For $t$ less than~$\overline T_\e $ defined by
Relation~\eqref {defininductiontime1}, we denote
\beq\label{conclusifeq2}
\begin{split} g_\lam(t) \eqdefa&
 g(t) \exp\Bigl(-\lam \int_0^t U_{\e,{\rm app}}(t')dt'\Bigr)\with
  U_{\e,{\rm app}}(t)\eqdefa \|u_{\e,{\rm
app}}(t)\|_{\cB^{-\f12+\f2p,\f1p}_p}^{4}\\
&+ \|u_{\e,{\rm app}}(t)\|^2_{\cB^{\frac 2p,\frac
1p}_p}+\|\p_zv^\h(t)\|_{\cB^{-\f12+\f2p,\f1p}_p}^{\f43}+\e\|(\e\p_z
w^\h, \p_z w^3)(t)\|^2_{\cB^{0,\frac 12}_2}.
\end{split}  \eeq
 Then we deduce from Equality\eqref{S1eq5} that
$$
\longformule{
R_{\e,\la}(t)=\int_0^t\exp\Bigl(-\lam \int_{t'}^t U_{\e,{\rm
app}}(t'') dt'' \Bigr)e^{(t-t')\D}\mathbb {P}_{a_\e}\Bigl(a_\e\D
R_{\e,\la}
}
{ {}- \dive \bigl(u_{\e,{\rm app}}\otimes R_{\e,\la} +R_{\e,\la}\otimes
\uapp +\e R_\e\otimes R_{\e,\la}\bigr)-E_{\e,\la}\Bigr)(t')dt'.
}
$$
So that for the norm $\|\cdot\|_{X(t)}$ given by Definition
\ref{anibesov}, we have \beq\label{conclusifeq3}
\begin{split}
\|R_{\e,\la}\|_{X(t)}\leq \Bigl\|&\exp\Bigl(-\lam \int_{t'}^t
U_{\e,{\rm app}}(t'') dt'' \Bigr)\mathbb {P}_{a_\e}\bigl(a_\e\D
R_{\e,\la}\\
& - \dive \bigl(u_{\e,{\rm app}}\otimes R_{\e,\la}
+R_{\e,\la}\otimes \uapp +\e R_\e\otimes
R_{\e,\la}\bigr)-E_{\e,\la}\bigr)\Bigr\|_{\cF_p(t)}. \end{split}
\eeq It is easy to observe from Inequality\refeq{cFL1Besovaniso},
the law of product \eqref{lawofproductanisobasic} and  Corollary
\ref{fulltransport0} that \beno \bigl\|\mathbb {P}_{a_\e}(a_\e \D
R_{\e,\la})\bigr\|_{\cF_p(t)} &\leq& C\Bigl(\|a_\e\D_\h
R_{\e,\la}\|_{L^1_t(\cB^{-1+\f2p,\f1p}_p)}
  +\|a_\e\p_3^2R_{\e,\la}\|_{L^{1}_t(\cB^{-1+\d+\frac 3 p,-\d}_p)}\Bigr)\\
&\leq & C\|a_\e\|_{L^\infty_t(\cB^{\f2p,\f1p}_p)}\Bigl(\|
R_{\e,\la}\|_{L^1_t(\cB^{1+\f2p,\f1p}_p)}+\|\p_3^2R_{\e,\la}\|_{L^{1}_t(\cB^{-1+\d+\frac
3 p,-\d}_p)}\Bigr)\\
&\leq & \cC_0\eta\exp(C\cR_0)\Bigl(\|
R_{\e,\la}\|_{L^1_t(\cB^{1+\f2p,\f1p}_p)}+\|\p_3^2R_{\e,\la}\|_{L^{1}_t(\cB^{-1+\d+\frac
3 p,-\d}_p)}\Bigr).
 \eeno
Along the same lines, we get
$$
\longformule{ \bigl\|\mathbb {P}_{a_\e}\dive(\e R_\e\otimes
R_{\e,\la}\bigr))\bigr\|_{\cF_p(t)}\lesssim \bigl\|\mathbb
{P}_{a_\e}\bigl(\e R_\e^\h\cdot\nh R_{\e,\la}+\e R_\e^3\p_3
R_{\e,\la}\bigr)\bigr\|_{L^1_t(\cB^{-1+\f2p,\f1p}_p)}}{{} \lesssim
\e\Bigl(\|R_\e\|_{L^2_t(\cB^{\f2p,\f1p}_p)}\|R_{\e,\la}\|_{L^2_t(\cB^{\f2p,\f1p}_p)}
+\|R_\e\|_{L^4_t(\cB^{-\f12+\f2p,\f1p}_p)}\|\p_3R_{\e,\la}\|_{L^{\f43}_t(\cB^{-\f12+\f2p,\f1p}_p)}\Bigr).
} $$ Using Inequality\refeq {cFL1Besovaniso},  we deduce  from
Inequalities~\eqref{estimE1},~\eqref{estiE2}, \eqref{estiE3}
and~\eqref{estiE4}  that \beno \bigl\|\mathbb {P}_{a_\e}
E_\e\bigr\|_{\cF_p(t)}\leq
\cC_0\Bigl(1+\bigl(\eta+\e^{1-\d-\f1p}\bigr)\exp(C\cR_0)\Bigr).
\eeno Now let us turn to the estimates of the last two terms in the
rightanside of Inequality~\eqref{conclusifeq3}. We deduce again from
Inequality\refeq{cFL1Besovaniso} that \beno &&\Bigl\|\exp\Bigl(-\lam
\int_{t'}^t U_{\e,{\rm app}}(t'') dt'' \Bigr)\dive_\h
\bigl(u_{\e,{\rm app}}\otimes R_{\e,\la}^\h
+R_{\e,\la}\otimes \uapp^\h\bigr)\Bigr\|_{\cF_p(t)}\\
&&\quad{}\lesssim \int_0^t\exp\Bigl(-\lam \int_{t'}^t U_{\e,{\rm app}}(t'')
dt'' \Bigr)\|R_{\e,\la}\otimes \uapp(t')\|_{\cB^{\f2p,\f1p}_p}dt'\\
&&\quad{}\lesssim \int_0^t\exp\Bigl(-\lam \int_{t'}^t U_{\e,{\rm app}}(t'')
dt'' \Bigr)\|\uapp(t')\|_{\cB^{\f2p,\f1p}_p}\|R_{\e,\la}(t')\|_{\cB^{\f2p,\f1p}_p}dt'\\
&&\quad{}\lesssim \biggl(\int_0^t\exp\Bigl(-2\lam \int_{t'}^t
U_{\e,{\rm app}}(\tau) dt''
\Bigr)\|\uapp(t')\|_{\cB^{\f2p,\f1p}_p}^2dt'\biggr)^{\f12}\|R_{\e,\la}\|_{L^2_t(\cB^{\f2p,\f1p}_p)},
\eeno which together with \eqref{conclusifeq2} ensures that
$$
\Bigl\|\exp\Bigl(-\lam \int_{t'}^t U_{\e,{\rm app}}(t'') dt''
\Bigr)\dive_\h \bigl(u_{\e,{\rm app}}\otimes R_{\e,\la}^\h
+R_{\e,\la}\otimes \uapp^\h\bigr)\Bigr\|_{\cF_p(t)}\lesssim
\f1{\la^{\f12}}\|R_{\e,\la}\|_{L^2_t(\cB^{\f2p,\f1p}_p)}.
$$
Along the same lines, we have
$$
\longformule{ \Bigl\|\exp\Bigl(-\lam \int_{t'}^t U_{\e,{\rm
app}}(t'') dt'' \Bigr)\p_3 \bigl(u_{\e,{\rm app}}\otimes
R_{\e,\la}^3 +R_{\e,\la}\otimes \uapp^3\bigr)\Bigr\|_{\cF_p(t)} } {
{} \lesssim \int_0^t\exp\Bigl(-\lam \int_{t'}^t U_{\e,{\rm
app}}(t'') dt'' \Bigr)\|\bigl(\p_3 R_{\e,\la}\otimes \uapp(t')+
R_{\e,\la}\otimes \p_3\uapp\bigr)(t')\|_{\cB^{-1+\f2p,\f1p}_p}dt', }
$$
By the definition of $\uapp$ given by \eqref{S2eq0aq}, we infer
\beno &&\int_0^t\exp\Bigl(-\lam \int_{t'}^t U_{\e,{\rm app}}(t'')
dt''\Bigr)\| R_{\e,\la}\otimes
\p_3\uapp(t')\|_{\cB^{-1+\f2p,\f1p}_p}dt'\\
&&\qquad{}\lesssim  \int_0^t\exp\Bigl(-\lam \int_{t'}^t U_{\e,{\rm
app}}(t'')
dt''\Bigr)\Bigl(\e\|\p_zv^\h(t')\|_{\cB^{-\f12+\f2p,\f1p}_p}\|R_{\e,\la}(t')\|_{\cB^{-\f12+\f2p,\f1p}_p}\\
&&\qquad\qquad\qquad\qquad\qquad\qquad\qquad\qquad\quad+\e^2\bigl\|(\e\p_zw^\h,\p_zw^3)\|_{\cB^{0,\f12}_2}
\|R_{\e,\la}(t')\|_{\cB^{\f2p,\f1p}_p}\Bigr)dt'. \eeno Then due to
Definition\refeq{conclusifeq2} of~$U_{\e,{\rm app}},$ H\"older
inequality implies that
$$
\longformule{
\int_0^t\exp\Bigl(-\lam \int_{t'}^t U_{\e,{\rm app}}(t'')
dt''\Bigr)\| R_{\e,\la}\otimes
\p_3\uapp(t')\|_{\cB^{-1+\f2p,\f1p}_p}dt'
}
{ {}\lesssim
\f{1}{\la^{\f34}}\|R_{\e,\la}\|_{L^{4}_t(\cB^{-\f12+\f2p,\f1p}_p)}+\f 1 {\la^{\f12}}\|R_{\e,\la}\|_{L^{2}_t(\cB^{\f2p,\f1p}_p)}.
}
$$
Following the same lines we get
$$
\int_0^t\exp\Bigl(-\lam \int_{t'}^t U_{\e,{\rm app}}(t'')
dt'' \Bigr)\| \p_3 R_{\e,\la}\otimes
\uapp(t')\|_{\cB^{-1+\f2p,\f1p}_p}dt'
\lesssim
\f{1}{\la^{\f14}}\|\p_3R_{\e,\la}\|_{L^{\f43}_t(\cB^{-\f12+\f2p,\f1p}_p)}.
$$
Resuming the above estimates into Inequality~\eqref{conclusifeq3} and using the
definition of the norm~$\|\cdot\|_{X(t)}$ given by Definition\refer{anibesov}, we infer  that, for any~$\lam$ greater than~$1$,
$$
\longformule{\|R_{\e,\la}\|_{X(t)}\leq
\cC_0\Bigl(1+\bigl(\eta+\e^{1-\d-\f1p}\bigr)\exp(C\cR_0)\Bigr)
+\Bigl(\cC_0\eta\exp(C\cR_0)+\f{C}{\la^{\f14}}\Bigr)
\|R_{\e,\la}\|_{X(t)} } {{}
+C\e\Bigl(\|R_{\e}\|_{L^{4}_t(\cB^{-\f12+\f2p,\f1p}_p)}+\|R_{\e}\|_{L^{2}_t(\cB^{\f2p,\f1p}_p)}\Bigr)\|R_{\e,\la}\|_{X(t)}.
}
$$
which together with the induction assumption \eqref
{defininductiontime1} ensures that for $t$ less than ~$ \overline
T_\e $ and for any~$\lam$ greater than~$1$, \beq
\label{conclusifeq4}
\left(1-C\Bigl(\cC_0\eta\exp(C\cR_0)+\f{1}{\la^{\f14}}+\e\cR_0\Bigr)\right)
\|R_{\e,\la}\|_{X(t)}\leq
\cC_0\Bigl(1+\bigl(\eta+\e^{1-\d-\f1p}\bigr)\exp(C\cR_0)\Bigr). \eeq
Let us take $\e,\eta$ sufficiently small and $\lam$ sufficiently
large so that
$$
\displaylines{ \eta< \f1{6C\cC_0}, \quad \f{C}{\la^{\f14}}\leq \f16
\andf 
\cR_0\leq \f1{C}\min\Bigl\{-\ln\bigl(2\e^{1-\d-1/p}\bigr),
-\ln(2\eta),-\ln(6C\cC_0\eta), 1/{6\e}\Bigr\}. }
$$ Then we deduce from Inequality\eqref{conclusifeq4} that~$\|R_{\e,\la}\|_{X(t)}\leq 4\cC_0,
$ from which and from Inequality\refeq {conclusifeq2}, we infer
 \beno
\|R_{\e}\|_{X(t)}\leq \|R_{\e,\la}\|_{X(t)} \exp\Bigl(\lam \int_0^t
U_{\e,{\rm app}}(t')dt'\Bigr)\leq 4\cC_0\exp\Bigl(\lam \int_0^t
U_{\e,{\rm app}}(t')dt'\Bigr). \eeno However, let us notice from
Equality\refeq{S2eq0aq} and Lemma \ref{lemBern} that \beno
\|u_{\e,{\rm app}}(t)\|_{\cB^{-\f12+\f2p,\f1p}_p} & \lesssim &
\|v^\h(t)\|_{\cB^{\f12,\f12}_2}+\e\|(\e w^\h,w^3)(t)\|_{\cB^{\f12,\f12}_2},\\
\|u_{\e,{\rm app}}(t)\|_{\cB^{\frac 2p,\frac 1p}_p}
 &\lesssim &
\|v^\h(t)\|_{\cB^{1,\f12}_2}+\e\|(\e
w^\h,w^3)(t)\|_{\cB^{1,\f12}_2}, \eeno which together with
Proposition \ref {estimwdetail}, \eqref{S7eq8} and \eqref{S7eq12}
ensures that \beq \label{conclusifeq5} \int_{\R^+} U_{\e,{\rm
app}}(t')dt'\leq \cC_0\andf \|R_{\e}\|_{X(t)}\leq
4\cC_0\exp\bigl(\cC_0\lam \bigr)\eqdefa \cC_0'. \eeq We take
$\cR_0=2\cC_0'$ and take $\e,\eta$ so small that \beq
\label{conclusifeq6} \eta< \f1{6C\cC_0} \andf 2\cC_0'\leq
\f1{C}\min\Bigl\{-\ln\bigl(2\e^{1-\d-1/p}\bigr),
-\ln(2\eta),-\ln(6C\cC_0\eta), 1/{6\e}\Bigr\}. \eeq Then we deduce
from Inequality \eqref{conclusifeq5} that
$$
\forall t\leq \overline T_\e\,,\ \|R_{\e}\|_{X(t)}\leq \f{\cR_0}2\,\cdotp
$$
The necassary condition for blow up implies that~$\overline T_\e $  equals to infinity.
This completes the proof of Theorem\refer{insslowvar}.
\end{proof}

\renewcommand{\theequation}{\thesection.\arabic{equation}}
\setcounter{equation}{0}

\appendix
\setcounter{equation}{0}
\section{The proof of
\eqref{cFL1Besovaniso}}

\begin{proof}[Proof of \refer{cFL1Besovaniso}] For $j=0,1,2,$ we
get, by applying Lemma 2.4 of \cite{BCD}, that
\beno
\Bigl\|\D_k^\h\D_\ell^\v\int_0^te^{(t-t')\D}\p_3^jf(t')dt'\Bigr\|_{L^q_T(L^p)}
&\lesssim &
2^{j\ell}\Bigl\|\int_0^te^{-c(t-t')\bigl(2^{2k}+2^{2\ell}\bigr)}\|\D_k^\h\D_\ell^\v
f(t')\|_{L^p}dt'\Bigr\|_{L^q_T}\\
& \lesssim &
\f{2^{j\ell}}{\bigl(2^{2k}+2^{2\ell}\bigr)^{\f1q}}\|\D_k^\h\D_\ell^\v
f\|_{L^1_T(L^p)}\\
& \lesssim &
d_{k,\ell}\f{2^{-k\al}2^{\ell(j-\beta)}}{\bigl(2^{2k}+2^{2\ell}\bigr)^{\f1q}}\|
f\|_{L^1_T(\cB^{\al,\beta}_p)},
\eeno
where $(d_{k,\ell})_{k,\ell\in\Z^2}$ denotes a generic element
of $\ell^1(\Z^2)$ so that $\sum_{k,\ell\in\Z^2}d_{k,\ell}=1.$ This
together with Definition \ref{anibesov} ensures that \beno
\Bigl\|\int_0^te^{(t-t')\D}\p_3^jf(t')dt'\Bigr\|_{L^q_T(\cB^{s,s'}_p)}\lesssim
\sum_{k,\ell\in\Z^2}d_{k,\ell}\f{2^{k(s-\al)}2^{\ell(s'+j-\beta)}}{\bigl(2^{2k}+2^{2\ell}\bigr)^{\f1q}}\|
f\|_{L^1_T(\cB^{\al,\beta}_p)}. \eeno In the particular case when
$$ \al\leq s, \quad \beta\leq s'+j \andf \al+\beta=s+s'+j-\f2q,$$ we
have \beq\label{appendixa}
\Bigl\|\int_0^te^{(t-t')\D}\p_3^jf(t')dt'\Bigr\|_{L^q_T(\cB^{s,s'}_p)}\lesssim
\| f\|_{L^1_T(\cB^{\al,\beta}_p)}. \eeq This  together with the
definition of the norm $\|\cdot\|_{\cF_p(T)}$ given by Definition
\ref{anibesov} leads to Inequality\refeq{cFL1Besovaniso}.
\end{proof}

\medbreak \noindent {\bf Acknowledgments.} Part of this work was
done when J.-Y. Chemin was visiting Morningside Center of the
Academy of Mathematics and Systems Sciences, CAS and when P. Zhang
was visiting J. L. Lions Laboratory of Universit\'e Pierre et Marie
Curie. We appreciate the hospitality and the financial support from
MCM,  National Center for Mathematics and Interdisciplinary Sciences
and Universit\'e Pierre et Marie Curie. P. Zhang is partially
supported by NSF of China under Grant 11371347, the fellowship from
Chinese Academy of Sciences and innovation grant from National
Center for Mathematics and Interdisciplinary Sciences.

\end{document}